\documentclass[12pt, centertags]{amsart}
\usepackage{amsmath, amstext, amsthm, a4, amssymb, amscd}
\usepackage{amssymb, graphics, epsfig, color}
\usepackage[mathscr]{eucal}
\usepackage[all]{xy}
\usepackage{mathrsfs}
\usepackage{imakeidx}
\makeindex

\numberwithin{equation}{section}

  \def\mC{\mathcal{C}}   
 \def\mE{\mathcal{E}}    
  \def\mO{\mathcal{O}}  \def\mQ{\mathcal{Q}}

      \def\Re{{\rm Re}}
 \def\Im{{\rm Im}}

 \DeclareMathOperator{\Ker}{Ker}  
  \DeclareMathOperator{\rk}{rk}  \DeclareMathOperator{\supp}{supp}
 \DeclareMathOperator{\tr}{Tr}

 \newtheorem{thm}{Theorem}[section] \newtheorem{lemma}[thm]{Lemma} \newtheorem{conj}[thm]{Conjecture} \newtheorem{prop}[thm]{Proposition}
  \theoremstyle{definition} \newtheorem{rem}[thm]{Remark} \theoremstyle{definition}
 \newtheorem{defn}[thm]{Definition} %
 \theoremstyle{remark} %
 \newcommand{\be}{\begin{eqnarray}}
   
 \newcommand{\comment}[1]{}

\makeindex

\begin{document}

 \title{Gluing formula of real analytic torsion forms and Adiabatic limit}
 \date{\today}
 \author{Jialin Zhu}
 \address{Chern Institute of Mathematics, Nankai University, Tianjin 300071, P.R. China}
 \email{jialinzhu@nankai.edu.cn}

\begin{abstract} In this article we use the adiabatic method to prove the gluing formula of real analytic torsion forms for a flat vector bundle on
 a smooth fibration under the assumption that the fiberwise twisted cohomology groups associated to the fibration of
 the cutting hypersurface are vanished.
In this paper we assume that the metrics have product structures near the cutting hypersurface.
\end{abstract}

\maketitle
\tableofcontents
\setcounter{section}{-1}
\section{Introduction}\label{s1}

Real analytic torsion is a spectral invariant of a compact Riemannian manifold equipped with a flat Hermitian vector bundle, that was introduced by Ray-Singer \cite{RaySing71}. Ray and Singer conjectured that for unitarily flat vector bundles, this invariant coincides with Reidemeister torsion, a topological invariant \cite{Milnor66}. This conjecture was established by Cheeger \cite{Chg} and M\"{u}ller \cite{Mu78}, and extended by M\"{u}ller
\cite{Mu93} for unimodular flat vector bundles, and by Bismut-Zhang \cite{BZ92} to arbitrary flat vector bundles.

In \cite{BL}, Bismut and Lott introduced what we now call Bismut-Lott analytic torsion form for a smooth fibration with a flat vector bundle as a natural higher degree generalization of the Ray-Singer analytic torsion. One of the significant facts is that the real analytic torsion form enters in a differential form version of a $C^{\infty}-$analog of the Riemann-Roch-Grothendieck theorem for holomorphic submersions.
 Bismut and Lott also showed that
under some appropriate conditions the Bismut-Lott torsion form is closed and its de Rham cohomology class is independent of the choices
of the geometric data in its definition (cf. \cite[Cor. 3.25]{BL}), thus it's a smooth invariant of the fibration with a flat vector bundle.

 Inspired by the work of Bismut and Lott,
Igusa
\cite{Igusa02} constructed a higher version of Reidemeister-Franz torsion by using the parameterized Morse theory.
 The reader refers to the books of Igusa \cite{Igusa02} and \cite{Igu05} for more information about the higher Igusa-Klein
torsion (IK-torsion). A second version of higher Reidemeister-Franz torsion (DWW-torsion) was defined by Dwyer, Weiss and Williams \cite{DwWeiWill03} in the homotopy theoretical approach.
Bismut and Goette \cite{BGo} obtained a family version of the Bismut-Zhang Theorem under the assumption that there exists a fiberwise
Morse function for the fibration in question. Goette \cite{Goette01}, \cite{Goette03} did more work towards the precise relation on BL-torsion and IK-torsion. The survey
\cite{Goette08} of Goette gives an overview about these higher torsion invariants for families. The
reader can refer to \cite{BGo}, \cite{Bunke00} for the equivariant BL-torsion forms and to \cite{BisMaZh} for the recent works on the analytic torsion forms.

In Igusa's axiomatization of higher torsion invariants (cf. \cite[$\S$3]{Igu08}), he summarized two axioms: Additivity Axiom and Transfer Axiom,
 to characterize the higher torsions, up to an universal cohomology class depending only on the underlaying manifold. In \cite[$\S$5]{Igu08}, Igusa established the additivity formula and the
transfer formula for IK-torsion. Roughly speaking, the additivity
formula of IK-torsion corresponds to the gluing formula of BL-torsion, and the transfer formula of IK-torsion corresponds to
the functoriality of BL-torsion with respect to the composition of two submersions, which
has been established by Ma \cite{Ma02}.
The main results of Igusa in \cite{Igu08} were first developed and announced during
the conference \cite{Conf03} on the higher torsion invariants in G\"{o}ttingen in September 2003.
To study the gluing problem of BL-torsion was proposed as an
open problem during this conference in order to clarify the relation between BL-torsion and IK-torsion.
Once we have established the gluing formula for BL-torsion, then it will imply basically that there exist a constant
$c$ and a cohomology class $R\in H^{*}(S)$ such than $\tau_{\rm IK}=c\tau_{\rm BL}+R$,
when they are well-defined as cohomology classes. This is the main motivation of this paper.

L\"{u}ck \cite{Luck93} established the gluing formula for the Ray-Singer analytic torsion for unitary flat vector bundles when the Riemannian metric has product structure near the boundary by using the results in \cite{LoRo91}.
 There are also other works on the gluing problem of the
 analytic torsion (cf. \cite{Has},  \cite{Vishik95}). Finally,
Br\"{u}ning and Ma \cite{BruMa06} established the anomaly formula for the analytic torsion on manifolds with boundary,
then they \cite{BruMa12} proved the gluing formula for the analytic torsion for any flat vector bundles and without any assumptions of product structures near the boundary.

In this paper, we will consider the gluing problem of Bismut-Lott torsion form. Let $\pi:M\rightarrow S$ be a smooth fibration over a compact manifold $S$.
We suppose that $X$ is a compact hypersurface in $M$ such that $M=M_{1}\cup_{X} M_{2}$ and $M_{1}, M_{2}$ are
 manifolds respectively with the common boundary $X$. We also assume that
\begin{align}\begin{aligned}\label{e.777}
Z_{1}\rightarrow M_{1}\overset{\pi}\rightarrow S, \quad Z_{2}\rightarrow M_{2}\overset{\pi}\rightarrow S \quad \text{and } Y\rightarrow X\overset{\pi}\rightarrow S
\end{aligned}\end{align}
are all smooth fibrations with fiber $Z_{1, b}$, $Z_{2, b}$ and $Y_{b}$ at $b\in S$ such that $Z_{b}=Z_{1, b}\cup_{Y_{b}}Z_{2, b}$.
In other words, the fibrations $M_{1}$ and $M_{2}$ can be glued into $M$ along $X$.

Let $TZ$ be the vertical tangent bundle of $M$ with a vertical Riemannian metric $g^{TZ}$. Let $T^{H}M$ be a horizontal tangent bundle of $M$ such that $TM=T^{H}M\bigoplus TZ$. Let $U_{\varepsilon}\simeq X\times (-\varepsilon, \varepsilon)$ be a product neighborhood of $X$ in $M$, and $\psi_{\varepsilon}:X\times (-\varepsilon, \varepsilon)\rightarrow X$ be the projection on the first factor.
We \textbf{assume} that $T^{H}M$ and $g^{TZ}$ have product structures on $U_{\varepsilon}$, i.e.,
\begin{align}\begin{aligned}\label{e.7593}
(T^{H}M)|_{X}\subset TX,\quad\big(T^{H}M\big)|_{U_{\varepsilon}}=\psi_{\varepsilon}^{*}\big((T^{H}M)|_{X}\big),
\end{aligned}\end{align}
\begin{align}\begin{aligned}\label{e.360}
g^{TZ}|_{(x', x_{m})}=g^{TY}(x')+dx^{2}_{m}, \quad (x', x_{m})\in X\times (-\varepsilon, \varepsilon).
\end{aligned}\end{align}
Then $T^{H}X:=(T^{H}M)|_{X}$ gives a horizontal bundle of fibration $X$, such that $TX=T^{H}X\oplus TY$.

 Let $F$ be a flat vector bundle over $M$ with flat connection $\nabla^{F}$, i.e., $(\nabla^{F})^{2}=0$. We trivialize $F$ along $x_{m}$-direction, by using the parallel transport with respect to $\nabla^{F}$, then we have
\begin{align}\begin{aligned}\label{e.7601}
\quad (F, \nabla^{F})|_{X \times (-\varepsilon, \varepsilon)}= \psi^{*}_{\varepsilon}(F|_{X}, \nabla^{F}|_{X}).
\end{aligned}\end{align}
Let $h^{F}$ be a Hermitian metric on $F$. We \textbf{assume} that under the identification (\ref{e.7601}), we have
\begin{align}\begin{aligned}\label{e.361}
h^{F}|_{U_{\varepsilon}}= \psi^{*}_{\varepsilon}(h^{F}|_{X}).
\end{aligned}\end{align}
If $h^{F}$ is flat, i.e., $\nabla^{F}h^{F}=0$, then (\ref{e.361}) is a consequence of the flatness of $h^{F}$.
In all of this paper, we \textbf{assume} that the triple $(T^{H}M, g^{TZ}, h^{F})$ has the product structures on $X_{[-\varepsilon,\varepsilon]}$ (cf. \cite{Zhu13}), i.e.,
\begin{align}\begin{aligned}\label{e.7611}
\text{ (\ref{e.7593}), (\ref{e.360}) and (\ref{e.361}) hold.}
\end{aligned}\end{align}

Let $\mathscr{T}_{\rm abs}(T^{H}M_{1}, g^{TZ_{1}}, h^{F})$ (resp. $\mathscr{T}_{\rm rel}(T^{H}M_{2}, g^{TZ_{2}}, h^{F})$) be the Bismut-Lott torsion form with absolute (resp. relative) boundary conditions that we introduced in \cite{Zhu13}.
Let $H^{p}(Z_{1}, F)$ (resp. $H^{p}(Z_{2}, Y, F)$) denote the flat vector bundle on $S$ with its canonical flat connection $\nabla^{H^{p}(Z_{1}, F)}$ (resp. $\nabla^{H^{p}(Z_{2}, Y, F)}$), whose fiber is isomorphic to the absolute (resp. absolute) cohomology group
$H^{p}(Z_{1, b}, F)$ (resp. $H^{p}(Z_{2,b},Y, F)$) at $b\in S$ (cf. \cite[$\S$ 1.3]{Zhu13}). Then we have a long exact sequence $(\mathscr{H}, \delta)$ of flat vector bundles (cf. \cite[(0.16)]{BruMa12}):
\begin{align}\begin{aligned}\label{e.4340}
\cdots \longrightarrow H^{p}(Z, F)\overset{\delta}{\longrightarrow} H^{p}(Z_{1}, F) \overset{\delta}{\longrightarrow}
H^{p+1}(Z_{2}, Y, F) \overset{\delta}{\longrightarrow}\cdots.
\end{aligned}\end{align}
 We denote the $L^{2}-$metric on $\mathscr{H}$ by $h^{\mathscr{H}}_{L^{2}}$ induced by Hodge theory
and the canonical flat connection by $\nabla^{\mathscr{H}}$.
Then we associate a torsion form $T_{f}(A^{\mathscr{H}}, h_{L^{2}}^{\mathscr{H}})$
to the triple $(\mathscr{H},  A^{\mathscr{H}}:=\delta+\nabla^{\mathscr{H}},h^{\mathscr{H}}_{L^{2}})$
 for $f(x)=xe^{x^{2}}$ (cf. \cite[Def. 2.20]{BL}).

Let $Q^{S}$ be the vector space of real even forms on $S$ and $Q^{S, 0}$ be the vector space of real exact even forms on $S$.
Let $\chi(Y)$ be the Euler characteristic of $Y$.

We formulate a conjecture
about the general gluing formula of analytic torsion forms in order to answer the open problem proposed in the conference  \cite{Conf03}
on higher torsion
invariants at G\"{o}ttingen 2003.
\begin{conj}\label{c.1}
With the assumption of product structures (\ref{e.7611}), the following identity holds in $Q^{S}/Q^{S, 0}$
\begin{align}\begin{aligned}\label{e.474}
\mathscr{T}(T^{H}M, g^{TZ}, h^{F})-\mathscr{T}_{\rm abs}(T^{H}M_{1}, g^{TZ_{1}}, h^{F})&-\mathscr{T}_{\rm rel}(T^{H}M_{2}, g^{TZ_{2}}, h^{F})\\
&=\frac{\log 2}{2}\rk(F)\chi(Y)+T_{f}(A^{\mathscr{H}}, h_{L^{2}}^{\mathscr{H}}).
\end{aligned}\end{align}
\end{conj}

The $0-$degree component of (\ref{e.474}) is exactly the gluing formula of Br\"{u}ning and Ma \cite[(0.22)]{BruMa12} in the case with product structures. In \cite{Zhu13}, we have proved this formula under the assumption that there exists a fiberwise Morse function on the fibration.

A way to prove the gluing formula (\ref{e.474}) is through the adiabatic limit method which have been used in the
research of the gluing problem of $\eta-$invariant by Douglas and Wojciechowski \cite{DouWoj91} (cf. also \cite{PaWo06}). The general case of the gluing problem of
$\eta-$invariant was solved by Bunke \cite{bunke95} by adiabatic method. Formally speaking, the adiabatic method is a limiting process that one stretches
the original manifold along the normal direction of certain hypersurface into two manifolds with cylinder ends of infinite length. Because of some difficulties in analysis,
we still need a topological condition that the fiberwise cohomology groups of the boundary
fibration $\pi:X\rightarrow S$ are
vanishing, i.e.,
\begin{align}\begin{aligned}\label{e.356}
H^{\bullet}(Y_{b}, F)=0, \quad \text{ for all } b \in S.\end{aligned}\end{align}
The assumption (\ref{e.356}) is equivalent to the vanishing of the kernel of fiberwise Dirac operator $D^{Y}$, i.e.,
$\Ker \big(D^{Y_{b}}\big)=0$, for all $b\in S$.
Now we state the main theorem of this paper.

\begin{thm}\label{t.13}The following identity holds in $Q^{S}/Q^{S, 0}$ under the condition (\ref{e.356})
\begin{align}\begin{aligned}\label{e.354}
\mathscr{T}(T^{H}M, g^{TZ}, h^{F})&-\mathscr{T}_{\rm abs}(T^{H}M_{1}, g^{TZ_{1}}, h^{F})\\
&\qquad-\mathscr{T}_{\rm rel}(T^{H}M_{2}, g^{TZ_{2}}, h^{F})=T_{f}(A^{\mathscr{H}}, h_{L^{2}}^{\mathscr{H}}).
\end{aligned}\end{align}
\end{thm}

In order to prove Theorem \ref{t.13}, we use some smooth diffeomorphisms (see Lemma \ref{l.1}) to stretch the original fibrations $M$, $M_{1}$ and $M_{2}$
linearly along the normal direction of the cutting hypersurface $X$ to obtain the stretched fibrations denoted by $M_{R}$, $M_{1, R}$ and $M_{2, R}$.
We represent the relation between $M_{R}$, $M_{1, R}$ and $M_{2, R}$ as in Figure \ref{fig.4}.\\

\begin{figure}
\input{fibration_R.pstex_t}
\caption[Figure]{\label{fig.4}}
\end{figure}\index{Figure \ref{fig.4}}

We have an exact sequences of flat vector bundles on $S$ parameterized by $R\geq 0$
\begin{align}\begin{aligned}\label{e.366}
\mathscr{H}_{R}:\quad\cdots \rightarrow H^{p}(Z_{R}, F)\rightarrow H^{p}(Z_{1, R}, F_{R})\rightarrow H^{p+1}(Z_{2, R}, Y, F_{R})\rightarrow\cdots.\end{aligned}\end{align}
 Thanks to the anomaly formulas (cf. \cite[Thm. 3.24]{BL}, Theorem \ref{t.5}), we have an important identity holding in $Q^{S}/Q^{S, 0}$
in the process of adiabatic limit
\begin{align}\begin{aligned}\label{e.367}
&\mathscr{T}(T^{H}M, g^{TZ}, h^{F})-\mathscr{T}_{\rm abs}(T^{H}M_{1}, g^{TZ_{1}}, h^{F})\\
&\qquad\qquad\qquad\qquad-\mathscr{T}_{\rm rel}(T^{H}M_{2}, g^{TZ_{2}}, h^{F})-T_{f}(A^{\mathscr{H}}, h_{L^{2}}^{\mathscr{H}})\\
&=\mathscr{T}(T^{H}M_{R}, g^{TZ_{R}}, h^{F_{R}})-\mathscr{T}_{\rm abs}(T^{H}M_{1, R}, g^{TZ_{1, R}}, h^{F_{R}})
\\
&\qquad\qquad\qquad\qquad-\mathscr{T}_{\rm rel}(T^{H}M_{2, R}, g^{TZ_{2, R}}, h^{F_{R}})-T_{f}(A^{\mathscr{H}_{R}}, h_{L^{2}}^{\mathscr{H}_{R}}).
\end{aligned}\end{align}

To prove Theorem \ref{t.13}, we only need to compute the limit of the right side of (\ref{e.367}) when $R\rightarrow \infty$, which will be
 divided into three parts to treat that are the small time contribution $S(R)$ (see (\ref{e.92})), the large time contribution $L(R)$ (see (\ref{e.93})) and the torsion form
$T_{f}(A^{\mathscr{H}}, h_{L^{2}}^{\mathscr{H}_{R}})$.

 For the adiabatic limit of small time contribution $S(R)$,
which will be shown to be vanished, we make use of an ideal of Atiyah,
 Patodi and Singer in \cite{APS}, while the main tools are the Duhamel's principle and the finite propagation speed property for
the wave equation. Under the assumption (\ref{e.356}), there exists a uniform spectral
gap, for all $R>0$ large enough, of the fiberwise Dirac operators $D^{Z_{R}}$, $D^{Z_{1, R}}$ and $D^{Z_{2, R}}$  bounded away from $0$. This spectral
gap permits us to show that the adiabatic limit of large time contribution $L(R)$ also vanishes. Finally, we show that the exact sequence
$\mathscr{H}_{R}$ of flat vector bundles in (\ref{e.366}) is asymptotically split when $R\rightarrow \infty$, as a consequence of
 which the
adiabatic limit of $T_{f}(A^{\mathscr{H}}, h_{L^{2}}^{\mathscr{H}_{R}})$ vanishes.

 The whole paper is organized as follows. In Section \ref{s.4}, we introduce some preliminaries for the gluing problem of Bismut-Lott torsion form and describe the adiabatic method approach to solve it by using the anomaly formulas. In Section \ref{s.3}, we state our main theorem and treat the small time contribution in the process of adiabatic limit.
In Section \ref{s.2}, we treat the large time contribution in the process of adiabatic limit. In Section \ref{s.1}, we deal with the limit of the torsion forms $T_{f}(A^{\mathscr{H}_{R}}, h_{L^{2}}^{\mathscr{H}_{R}})$,
when $R\rightarrow \infty$.

\textbf{Acknowledgments.}
This paper is the second part of the author's Ph.D. thesis at Universit{\'e} Paris Diderot-Paris 7.
 He would like to thank his Ph.D. advisor Professor Xiaonan Ma for giving him patient instruction and constant encouragement
during the process of completing this thesis.

\section{Anomaly formula and gluing problem of Bismut-Lott torsion forms}\label{s.4}

In this section, we will introduce the geometric background for the gluing problem of Bismut-Lott torsion forms and describe the adiabatic approach to solve this problem.

This section is organized as follows.
In Section \ref{ss1.1}, we introduce some geometric objects for a fibration with boundary.
In Section \ref{ss1.4}, we recall the Bismut-Lott superconnection.
In Section \ref{ss1.7}, we establish the anomaly formulas of Bismut-Lott torsion forms in the case with boundary.
In Section \ref{ss1.6}, we establish some technical tools on the estimates of heat kernels on the stretched fibration $M_{R}$ with boundary for $R\geq 0$.
In Section \ref{ss1.8}, we formulate the gluing problem of Bismut-Lott torsion forms in detail. In Section \ref{ss1.9}, we construct
a diffeomorphism $\phi_{R}:M\rightarrow M_{R}$. In Section \ref{ss1.10}, we make use of the anomaly formulas and the diffeomorphism $\phi_{R}$ to establish the important identity
(\ref{e.367}) in the process of adiabatic limit.

\subsection{Smooth fibration with boundaries and fibration with cylinder end}\label{ss1.1}

 Let $S$ be a compact smooth manifold of dimension $n$.
 Let $TS$ be the tangent bundle of $S$ and $T^{*}S$ be the cotangent bundle. For a vector bundle $F$ on $S$, let $\Omega^{j}(S,F)$ be the space of $F$-valued smooth differential $j-$forms on $S$, $\Omega(S,F)=\bigoplus_{j=0}^{n}\Omega^{j}(S,F)$  and $\Omega^{\bullet}(S)=\Omega^{\bullet}(S,\mathbb{R})$.

 Let $E=E_{+}\oplus E_{-}$ be a $\mathbb{Z}_{2}$-graded complex vector bundle over $S$ with a flat connection $\nabla^{E}=\nabla^{E_{+}}\oplus\nabla^{E_{-}}$, i.e., the curvature $(\nabla^{E_{\pm}})^{2}$ is zero. By definition, a Hermitian metric $h^{E}$ on $\mathbb{Z}_{2}$-graded bundle $E$ is a Hermitian metric such that $E_{+}$ and $E_{-}$ are orthogonal.

Let $(\nabla^{E})^{*}$ be the adjoint of $\nabla^{E}$ with respect to $h^{E}$. Let
\begin{align}\begin{aligned}\label{e.499}
\omega(E, h^{E})=(\nabla^{E})^{*}-\nabla^{E}=(h^{E})^{-1}\nabla^{E}h^{E}\in \Omega^{1}(S, \text{End}(E)).
\end{aligned}\end{align}\index{$\omega(E, h^{E})$}
Let $\varphi:\Omega(S)\rightarrow \Omega(S)$ be the linear map such that for all $\beta\in \Omega^{k}(S)$,
\begin{align}\begin{aligned}\label{e.500}
\varphi\beta=(2i\pi)^{-k/2}\beta.
\end{aligned}\end{align}

In this paper, we always set
\begin{align}\begin{aligned}\label{e.762}
f(a)=a \exp(a^{2}),
\end{aligned}\end{align}
 which is a homomorphic odd function over $\mathbb{C}$.
\begin{defn}\label{d.23}
Put
\begin{align}\begin{aligned}\label{e.501}
f(\nabla^{E}, h^{E})=(2i\pi)^{1/2}\varphi\tr_{s}\left[f(\frac{\omega(E,h^{E})}{2})\right]\in \Omega(S),
\end{aligned}\end{align}
where $\tr_{s}[\cdot,\cdot]:=\tr|_{E_{+}}-\tr|_{E_{-}}$ denotes the supertrace (cf. \cite{BGV92}). It is a real, odd and closed form and its de Rham cohomology class does not depend on the choice of $h^{E}$ (cf. \cite[Theorems 1.8, 1.11]{BL}).
\end{defn}

\begin{defn}\label{d.8}
Let $h^{'E}$ be another Hermitian metric on $E$. As \cite[Def. 1.12]{BL}, we define
\begin{align}\begin{aligned}\label{e.51}
\widetilde{f}(\nabla^{E}, h^{E}, h^{'E})=\int_{0}^{1}\varphi\tr_{s}\left[\frac{1}{2}(h^{E}_{l})^{-1}\frac{\partial h^{E}_{l}}{\partial l}f'\big(\frac{\omega(E, h^{E}_{l})}{2}\big)\right]dl\in Q^{S}/Q^{S, 0},
\end{aligned}\end{align}
where $h^{E}_{l}, l\in[0, 1]$ is a smooth path of metrics on $E$ such that $h^{E}_{0}=h^{E}$ and $h^{E}_{1}=h^{'E}$.
\end{defn}

Then from \cite[Thm. 1.11]{BL}, we get
\begin{align}\begin{aligned}\label{e.52}
d\widetilde{f}(\nabla^{E}, h^{E}, h^{'E})=f(\nabla^{E}, h^{'E})-f(\nabla^{E}, h^{E}).
\end{aligned}\end{align}
Moreover,
the class $\widetilde{f}(\nabla^{E}, h^{E}, h^{'E})\in Q^{S}/Q^{S, 0}$ does not depend on the choice of the path  $h^{E}_{l}$.

  Let
 $\pi:M\rightarrow S$ be a smooth fibration with boundary $X:=\partial M$, and its standard fiber $Z$ is a compact manifold. We assume that the boundary $X$ of $M$ is a smooth fibration denoted
by $\pi_{\partial}:X\rightarrow S$ with fiber $Y$ such that $Y=\partial Z$.
\begin{defn}
Let $X$ be a compact manifold and $I$ (not be a point) be an interval of $\mathbb{R}$, we set $X_{I}:=X\times I$, for example $X_{[-R,R]}=X\times[-R,R]$, $X_{\mathbb{R}}=X\times(-\infty,+\infty)$.
\end{defn}

Let $X_{[-\varepsilon,0]}$ be a product neighborhood of $X$, and we identify $\partial M$ with $X\times \{0\}$.
Let $TM$ be the tangent bundle of $M$. Let $TZ$ be the vertical subbundle of $TM$. Let $T^{H}M$ be the horizontal subbundle of $TM$ verifying the assumption of product structure (\ref{e.7593}) on $X_{[-\varepsilon,0]}$, then we have $TM=TZ\oplus T^{H}M$.
 Let $TY$ be the vertical tangent bundle of the fibration $X$, then it's a subbundle of $TZ$ restricted on $X$.
Let $N$ be the normal bundle of $X\subset M$, i.e., $N:=TM/TX$, then by our assumption
we have $TZ/TY\cong N$. We note that in our case $N$ is a trivial oriented line bundle on $X$ (cf. \cite[p.54, p.66]{BottTu}).

Let $g^{TZ}$ be a metric on $TZ$ verifying the assumption of product structure (\ref{e.360}) on $X_{[-\varepsilon,0]}$, let $g^{T Y}$ be the metric on $TY$ induced by $g^{TZ}$.
Using the metric $g^{TZ}$, we identify $N$ with the orthogonal complement of $TY$ in $TZ$, thus we have $TZ|_{X}= TY\oplus N$.

Let $(F, \nabla^{F})$ be a flat complex vector bundle on $M$ with a flat connection $\nabla ^{F}$, i.e., $(\nabla ^{F})^{2}=0$. Let $h^{F}$ be
a Hermitian metric on $F$. We trivialize $F$ by $\nabla^{F}$ as in (\ref{e.7601}) on $X_{[-\varepsilon,0]}$ and assume that $h^{F}$ verifies the assumption of product structure (\ref{e.361}) on $X_{[-\varepsilon,0]}$.

\subsection{Bismut-Lott superconnection form}\label{ss1.4}

Let $\Omega^{\bullet}(Z,F|_{Z})$ be the infinite-dimensional $\mathbb{Z}-$graded vector bundle over $S$ whose fiber
is $\Omega^{\bullet}(Z_{b},F|_{Z_{b}})$  at $b\in S$. That is
\begin{align}\begin{aligned}\label{e.540}
\Omega^{\bullet}(M,F)=\Omega^{\bullet}(S,\Omega^{\bullet}(Z,F|_{Z})).
\end{aligned}\end{align}

 Let $o(TZ)$ be the orientation bundle of $TZ$ (cf. \cite[p.88]{BottTu}),
 which is a flat real line bundle on $M$.
Let $dv_{Z}$ be the Riemannian volume form on fibers $Z$ associated to $g^{TZ}$, which is a section of $\Lambda^{m}(T^{*}Z)\otimes o(TZ)$ over $M$.
 The metrics $g^{TZ}$ and $h^{F}$ induce a Hermitian metric on $\Omega^{\bullet}(Z,F|_{Z})$ such that
for $s, s'\in \Omega^{\bullet}(Z_{b},F|_{Z_{b}})$, $b\in S$,
\begin{align}\begin{aligned}\label{e.547}
\langle s, s'\rangle_{h^{\Omega^{\bullet}(Z,F|_{Z})}}(b):=\int_{Z_{b}}\langle s, s' \rangle_{g^{\Lambda(T^{*}Z)\otimes F}}(x)dv_{Z_{b}}(x).
\end{aligned}\end{align}

 Let $P^{TZ}$ denote the projection from $TM=T^{H}M\oplus TZ$ to $TZ$. For $U\in TS$, let
 $U^{H}$ be the horizontal lift of $U$ in $T^{H}M$, so that $\pi_{*}U^{H}=U$.

\begin{defn}\label{d.32}
For $s\in C^{\infty}(S, \Omega^{\bullet}(Z,F|_{Z}))$ and $U\in TS$, the Lie derivative $L_{U^{H}}$ acts on $C^{\infty}(S, \Omega^{\bullet}(Z,F|_{Z}))$. Then
$
\nabla^{\Omega^{\bullet}(Z,F|_{Z})}_{U}s:=L_{U^{H}}s
$
defines a connection on $\Omega^{\bullet}(Z,F|_{Z})$ preserving the $\mathbb{Z}-$grading.
\end{defn}

 Let $d^{Z}$ be the exterior differentiation along fibers $(Z, F,\nabla^{F})$.
If $U_{1}, U_{2} \in TS$, we put
\begin{align}\begin{aligned}\label{e.542}
T(U_{1}, U_{2})=-P^{TZ}[U^{H}_{1}, U^{H}_{2}]\in C^{\infty}(M, TZ),
\end{aligned}\end{align}
then $T$ is a tensor, i.e., $T\in C^{\infty}(M, \pi^{*}\Lambda^{2}(T^{*}S)\otimes TZ)$. Let $i_{T}$ be the interior multiplication by $T$ in the vertical direction.

The flat connection $\nabla^{F}$ extends naturally
 to be an exterior differential operator $d^{M}$ acting on $\Omega^{\bullet}(M, F)$,
then it defines a flat superconnection of total degree 1 on $\Omega^{\bullet}(Z,F|_{Z})$.
By \cite[Prop. 3.4]{BL}, we have the following identity
\begin{align}\begin{aligned}\label{e.544}
d^{M}=d^{Z}+\nabla^{\Omega^{\bullet}(Z,F|_{Z})}+i_{T}.
\end{aligned}\end{align}

Let $(\nabla^{\Omega^{\bullet}(Z,F|_{Z})})^{*}, \, (d^{M})^{*}, \, (i_{T})^{*}$, $(d^{Z})^{*}$ be the formal adjoints
 of $\nabla^{\Omega^{\bullet}(Z,F|_{Z})}, \, d^{M}, \, i_{T}$, $d^{Z}$ with respect to the Hermitian metric $h^{\Omega^{\bullet}(Z,F|_{Z})}$ in (\ref{e.547}).
Set
\begin{align}\begin{aligned}\label{e.548}
&D^{Z}=d^{Z}+(d^{Z})^{*}, \quad \nabla^{\Omega^{\bullet}(Z,F|_{Z}), u}=\frac{1}{2}(\nabla^{\Omega^{\bullet}(Z,F|_{Z})}+(\nabla^{\Omega^{\bullet}(Z,F|_{Z})})^{*}).
\end{aligned}\end{align}
Then the Hodge Laplacian associated to $g^{TZ}$ and $h^{F}$ along the fibers $Z$ is
\begin{align}\begin{aligned}\label{e.485}
(D^{Z})^{2}=d^{Z}(d^{Z})^{*}+(d^{Z})^{*}d^{Z}:\Omega^{\bullet}(Z, F|_{Z})\rightarrow \Omega^{\bullet}(Z, F|_{Z}).
\end{aligned}\end{align}

Let $N$\index{$N$} be the number operator on $\Omega^{\bullet}(Z, F|_{Z})$, i.e., it acts by multiplication by $k$ on $\Omega^{k}(Z, F|_{Z})$.
For $t>0$, we set
\begin{align}\begin{aligned}\label{e.549}
&C'_{t}=t^{N/2}d^{M}t^{-N/2}, \quad C''_{t}=t^{-N/2}(d^{M})^{*}t^{N/2}, \\
&C_{t}=\frac{1}{2}(C'_{t}+C''_{t}), \quad\quad D_{t}=\frac{1}{2}(C''_{t}-C'_{t}).
\end{aligned}\end{align}
Then $C''_{t}$ is the adjoint of $C'_{t}$ with respect to $h^{\Omega^{\bullet}(Z, F|_{Z})}$. We note that $C_{t}$ is a superconnection and $D_{t}$ is an odd element of $\Omega(S, \text{End}(\Omega^{\bullet}(Z, F|_{Z})))$. Moreover, we have
\begin{align}\begin{aligned}\label{e.550}
C^{2}_{t}=-D^{2}_{t}.
\end{aligned}\end{align}

Let $g^{TS}$ be a Riemannian metric on $TS$, then $g^{TM}=\pi^{*}g^{TS}\oplus g^{TZ}$ defines a Riemannian metric on
 $TM=T^{H}M\oplus TZ$. Let $\nabla^{TM}$ denote the Levi-Civita connection on $TM$. Then
\begin{align}\begin{aligned}\label{e.551}
\nabla^{TZ}=P^{TZ}\nabla^{TM}
\end{aligned}\end{align}
defines a connection on $TZ$, which is independent of the choice of $g^{TS}$ (cf. \cite[Def. 1.6, Thm. 1.9]{Bismut86}).

For $X\in TZ$, let $X^{*}\in T^{*}Z$ be the dual of $X$ by the metric $g^{TZ}$. Set
\begin{align}\begin{aligned}\label{e.553}
c(X)=X^{*}\wedge- \, i_{X}, \quad \widehat{c}(X)=X^{*}\wedge+\, i_{X},
\end{aligned}\end{align}
where $i_{\cdot}$ denotes the interior multiplication.

By \cite[Prop. 3.9]{BL}, we get
\begin{align}\begin{aligned}\label{e.556}
C_{t}=\frac{\sqrt{t}}{2}D^{Z}+\nabla^{\Omega^{\bullet}(Z, F|_{Z}), u}-\frac{1}{2\sqrt{t}}c(T),
\end{aligned}\end{align}
which is essentially the same as the Bismut superconnection (cf. \cite[$\S$III.a)]{Bismut86}).

For any $ t>0$, the operator $D_{t}$ in (\ref{e.549}) is a first-order fiberwise-elliptic differential operator, then $f(D_{t})$ is a fiberwise trace class operator. For $t>0$, we put:
\begin{align}\begin{aligned}\label{e.568}
&f(C'_{t}, h^{\Omega^{\bullet}(Z,F|_{Z})})=(2i\pi)^{1/2}\varphi \tr_{s}[f(D_{t})]\in \Omega (S), \\
&f^{\wedge}(C'_{t}, h^{\Omega^{\bullet}(Z,F|_{Z})}):=\varphi \tr_{s}\left[\frac{N}{2}f'(D_{t})\right]=\varphi \tr_{s}\left[\frac{N}{2}(1+2D_{t}^{2})e^{D_{t}^{2}}\right].\end{aligned}\end{align}

\subsection{Analytic torsion forms of boundary case and anomaly formulas}\label{ss1.7}
Let $H(Z,F)$ (resp. $H(Z,Y,F)$) be the flat vector bundle of fiberwise absolute (resp. relative) cohomology groups with the canonical connection $\nabla^{H(Z, F)}$ (resp. $\nabla^{H(Z,Y,F)}$)(cf. \cite[$\S$1.3]{Zhu13}).
Let $\mathscr{T}_{{\rm abs }}(T^{H}M, g^{TZ}, h^{F})\in \Omega(S)$ (resp. $\mathscr{T}_{{\rm rel }}(T^{H}M, g^{TZ}, h^{F})$) be the Bismut-Lott torsion forms with absolute (resp. relative) boundary conditions introduced in \cite[Def. 1.19]{Zhu13}.

Now we describe how the torsion forms depend on their arguments. Let $(T^{H}M, g^{TZ}, h^{F})$ and
 $(T'^{H}M, g'^{TZ}, h'^{F}) $ be two triples, such that they satisfy the assumption
 of product structures (\ref{e.7611}) on the same product neighborhood $X_{[-\varepsilon',0]}$ of $X\subset M$.
We will mark the objects associated to the second triple with $a\, '$.

We define a differential form associated to $g^{TZ}$ along the fibers by:
\begin{align}\begin{aligned}\label{1.103}
e(TZ,\nabla^{TZ}):=(-1)^{m/2}{\rm Pf}(R^{TZ})=(-1)^{m/2}\int^{B_{Z}}\exp(R^{TZ}).
\end{aligned}\end{align}
We connect $g^{TZ}$ and $g'^{TZ}$ linearly by a path $g^{TZ}_{s}=sg'^{TZ}+(1-s)g^{TZ}$,
which still satisfy the assumption (\ref{e.7611}) of product structures for each $s\in [0,1]$. Let $\nabla^{TZ}_{s}$ be
the Levi-Civita connection with respect to $g^{TZ}_{s}$ (see (\ref{e.551})) and its curvature is denoted by $R_{s}^{TZ}$.
Then we define in $Q^{M}/Q^{M, 0}$ (cf. \cite[Prop. 3.6]{Zhang})
\begin{align}\begin{aligned}\label{1.106}
\widetilde{e}(TZ,\nabla^{TZ},\nabla'^{TZ})=(-1)^{m/2}\int_{0}^{1}\int^{B_{Z}}\frac{d\nabla^{TZ}_{s}}{ds}\exp(R_{s}^{TZ})ds.
\end{aligned}\end{align}
This differential form $\widetilde{e}(TZ,\nabla^{TZ},\nabla'^{TZ})$ is of degree $\dim(Z)-1$ such that
\begin{align}\begin{aligned}\label{e.46}
d\widetilde{e}(TZ, \nabla^{TZ}, \nabla'^{TZ})=e(TZ, \nabla'^{TZ})-e(TZ, \nabla^{TZ}).
\end{aligned}\end{align}
If $\dim(Z)$ is odd, we take $\widetilde{e}(TZ, \nabla^{TZ}, \nabla'^{TZ})$ to be zero. For the exact definition of the
secondary Euler class in the sense of Chern-Simons, the reader can refer to \cite[Prop. 2.7]{BruMa06}.\\

Now we establish the anomaly formula for Bismut-Lott's torsion form in boundary case.

\begin{thm}\label{t.5}If $(T^{H}M, g^{TZ}, h^{F})$ and $(T'^{H}M, g'^{TZ}, h'^{F}) $ verify the assumption of product structures {\rm (see (\ref{e.7611}))} on the same neighborhood $X_{[-\varepsilon', 0]}$,
 then the following identity holds in $Q^{S}/Q^{S, 0}$ for absolute or relative boundary conditions
\begin{align}\begin{aligned}\label{e.54}
&\mathscr{T}_{\rm abs/rel}(T'^{H}M, g'^{TZ}, h'^{F})-\mathscr{T}_{\rm abs/rel }(T^{H}M, g^{TZ}, h^{F})\\
&=\int_{Z}\widetilde{e}(TZ, \nabla^{TZ}, \nabla'^{TZ})f(\nabla^{F}, h^{F})+\int_{Z}e(TZ, \nabla'^{TZ})\widetilde{f}(\nabla^{F}, h^{F}, h'^{F})\\
&\pm\frac{1}{2}\int_{Y}\widetilde{e}(TY, \nabla^{TY}, \nabla'^{TY})f(\nabla^{F}, h^{F})
\pm\frac{1}{2}\int_{Y}e(TY, \nabla'^{TY})\widetilde{f}(\nabla^{F}, h^{F}, h'^{F})\\
&\qquad\qquad-\widetilde{f}(\nabla^{H_{\rm abs/rel}(Z,F)}, h^{H_{\rm abs/rel}(Z,F)}, h'^{H_{\rm abs/rel}(Z,F)}),
\end{aligned}\end{align}
where we denote $H_{\rm abs}(Z,F)=H(Z,F)$ and $H_{\rm rel}(Z,F)=H(Z,Y,F)$.
\end{thm}

\begin{proof} First, a horizontal distribution on $M$ is simply a splitting of the exact sequence
$$
0\rightarrow TZ\rightarrow TM\rightarrow \pi^{*}TS\rightarrow0.
$$
As the space of splitting maps is affine, it follows that any pair of horizontal distributions can be connected by a smooth path of horizontal
distributions. Let $s\in [0,1]$ parameterize a smooth path $\{T^{H}_{s}M\},\, s\in [0,1]$ such that $T^{H}_{0}M=T^{H}M$ and $T^{H}_{1}M=T'^{H}M$.
Similarly, we set for $s\in [0,1]$
\begin{align}\begin{aligned}\label{e.727}
g^{TZ}_{s}=sg'^{TZ}+(1-s)g^{TZ},\quad h^{F}_{s}=sh'^{F}+(1-s)h^{F}.
\end{aligned}\end{align}
Let $\widetilde{\pi}:M_{[0,1]}\rightarrow S_{[0,1]}$ be the obvious projection, induced by $\pi:M\rightarrow S$, with fiber $\widetilde{Z}$.
Let $\widetilde{X}=X\times[0,1]$. Let $\widetilde{F}$ be the lift of $F$ to $M_{[0,1]}$.

Now $T^{H}(M_{[0,1]})|_{(0,s)}=T^{H}_{s}M\times \mathbb{R}$ defines a horizontal subbundle $T^{H}(M_{[0,1]})$ of $T(M_{[0,1]})$,
 and $T\widetilde{Z}$ and $\widetilde{F}$ are naturally equipped with metrics $g^{T\widetilde{Z}}$ and $h^{\widetilde{F}}$.
Since for all $s\in [0,1]$ the metrics $g^{TZ}_{s}$ and $h^{F}_{s}$ also satisfy, respectively, the assumptions (\ref{e.360}), (\ref{e.361}) of product structures
on the same product neighborhood $\widetilde{X}\times[-\varepsilon',0]$, by \cite[Thm. 1.20]{Zhu13},
we get
\begin{align}\begin{aligned}\label{e.728}
d \mathscr{T}_{{\rm bd }}&(T^{H}M, g^{T\widetilde{Z}}, h^{\widetilde{F}})
=\int_{\widetilde{Z}}e(T\widetilde{Z}, \nabla^{T\widetilde{Z}})f(\nabla^{\widetilde{F}}, h^{\widetilde{F}})\\
&+(-1)^{{\rm \mathbf{bd}}}\frac{1}{2}\int_{\widetilde{Y}}e(T\widetilde{Y}, \nabla^{T\widetilde{Y}})f(\nabla^{\widetilde{F}}, h^{\widetilde{F}})-f(\nabla^{H_{{\rm bd }}(\widetilde{Z},\widetilde{F})}, h^{H_{{\rm bd }}(\widetilde{Z},\widetilde{F})}).
\end{aligned}\end{align}
Let $\widetilde{\sigma}:S_{[0,1]}\rightarrow S$ be the projection onto the first factor. Then there is an equality of pairs $
\big(H(\widetilde{Z},\widetilde{F}),\nabla^{H(\widetilde{Z},\widetilde{F})}\big)=\sigma^{*}\big(H(Z,F),\nabla^{H(Z,F)}\big)$.
The restriction of $\mathscr{T}_{{\rm bd }}(T^{H}M, g^{T\widetilde{Z}}, h^{\widetilde{F}})$ to $S\times\{0\}$ (resp. $S\times\{1\}$) coincides with
$\mathscr{T}_{{\rm bd }}(T^{H}M, g^{TZ}, h^{F})$ (resp. $\mathscr{T}_{{\rm bd }}(T^{H}M, g'^{TZ}, h'^{F})$). Comparing the $ds-$terms of the two
sides of equation (\ref{e.728}) and integrating with respect to $s$ yields equation (\ref{e.54}).
\end{proof}

\begin{rem}\label{r.6}
As the proof of anomaly formulas in \cite{RaySing71} for the manifolds with boundary,
we should fix the normal vector of the boundary, when the vertical Riemannian metric $g^{TZ}$ is changed,
in order to have the same boundary conditions.
\end{rem}

\subsection{Off-diagonal estimates and comparison of heat kernels}\label{ss1.6}
Now we introduce a fibration with stretched cylinder end. Let $M$ be a fibration with boundary $X$. For $R\geq 0$, we let
\begin{align}\begin{aligned}\label{e.600}
M_{R}=M\cup_{X} X_{[0, R]},
\end{aligned}\end{align}
to make a new fibration by adding a cylinder end of length $R$ on $X\times (-\varepsilon, 0]$ with fiber
$Z_{R}:=Z\cup_{Y} Y_{[0, R]}$.
 The stretched fibration $M_{R}$ has a cylinder end $X_{[-\varepsilon, R]}$. Then by a change of coordinates
 \begin{align}\begin{aligned}\label{e.614}
 &X_{[-\varepsilon, R]}\longrightarrow X_{[-R-\varepsilon, 0]},\quad
 (x', u)\longmapsto (y', v)=(x', u-R),
\end{aligned}\end{align}
 we will always identify the cylinder end of $M_{R}$ with $X_{[-R-\varepsilon, 0]}$, such that $\partial M=X\times \{0\}$.

Using the product structures (\ref{e.7611}), we extend $T^{H}M$, $g^{TZ}$, $F$, $h^{F}$ and $\nabla^{F}$ naturally from $M$ to $M_{R}$ and denote the corresponding objects by $T^{H}M_{R}$, $g^{TZ_{R}}$, $F_{R}$, $h^{F_{R}}$ and $\nabla^{F_{R}}$. We note that these new objects also have the product structures on $X_{[-R-\varepsilon, 0]}$.

 Let
$\psi_{R+\varepsilon}:X_{[-R-\varepsilon, 0]}\rightarrow X$
 be the projection on the first factor. By this extension, we have
\begin{align}\begin{aligned}\label{e.537}
&g^{TZ_{R}}(x', x_{m})=g^{T Y}(x')\oplus dx^{2}_{m}, \quad (x', x_{m})\in X_{[-R-\varepsilon, 0]},\\
&h^{F_{R}}=\psi_{R+\varepsilon}^{*}\big(h^{F}\big)|_{X}, \quad \nabla^{F_{R}}=\psi_{R+\varepsilon}^{*}\big(\nabla^{F}\big)|_{X}
 \quad \text{ on } X_{[-R-\varepsilon, 0]}.
\end{aligned}\end{align}

In this subsection we work on $(M_{R},F_{R})$.
Recall that $$(D^{Z_{R}})^{2}:=d^{Z_{R}}(d^{Z_{R}})^{*}+(d^{Z_{R}})^{*}d^{Z_{R}}$$
is the fiberwise Hodge-Laplacian acting on
$\Omega(Z_{R},F_{R}|_{Z_{R}})$ with absolute or relative boundary conditions (cf. \cite[(1.52)]{Zhu13}).
For $t>0$, let $e^{-t(D^{Z_{R}})^{2}}$ be its heat operator with a smooth kernel denoted by
$e^{-t(D^{Z_{R}})^{2}}(x, x'),\, x, x'\in Z_{R}$.
Let $d(x, x')$ \index{$d(x, x')$} to
denote the Riemannian distance
 between two points $x, x'$ in $Z_{R}$ with respect to $g^{TZ_{R}}$.

We have the off-diagonal estimates on the heat kernel $e^{-t(D^{Z_{R}})^{2}}(x, x')$.
\begin{lemma}\label{l.42}
There exists $c>0$ such that for any $l\in \mathbb{N}$, there exists $C_{l}>0$ such that for any $R\geq 0$, $t>0$ and $x, x'\in Z_{R}$ with $d(x, x')\geq 1$, we have
\begin{align}\begin{aligned}\label{e.424}
\big|e^{-t(D^{Z_{R}})^{2}}(x, x')\big|_{\mathscr{C}^{l}}\leq C_{l}e^{-c\frac{d^{2}(x, x')}{t}}.
\end{aligned}\end{align}
\end{lemma}
\begin{proof}
Let $f(v)$ be an even smooth cut-off function on $\mathbb{R}$ such that
\begin{align}\begin{aligned}\label{e.435}
f(v):=\left\{
\begin{array}{cc}
 1 , & \text{ for }|v|\leq \frac{1}{2}, \\
 0 , & \text{ for }|v|\geq 1.
\end{array}
\right.
\end{aligned}\end{align}
For $a\in \mathbb{C}$, $u>0$, we denote that (cf. \cite[Def. 1.6.3]{MaMa07})
\begin{align}\begin{aligned}\label{e.437}
&F_{u}(a):=\int_{-\infty}^{+\infty}\cos \left(va\right) e^{-\frac{v^{2}}{2}}f\left(\sqrt{u}v\right)\frac{dv}{\sqrt{2\pi}}, \\
&G_{u}(a):=\int_{-\infty}^{+\infty}\cos \left(va\right)e^{-\frac{v^{2}}{2}}\left(1-f\big(\sqrt{u}v\big)\right)\frac{dv}{\sqrt{2\pi}},
\end{aligned}\end{align}
then we have
\begin{align}\begin{aligned}\label{e.438}
\exp\left( -t(D^{Z_{R}})^{2}\right)=F_{2t/r^{2}}(\sqrt{2t}D^{Z_{R}})+G_{2t/r^{2}}(\sqrt{2t}D^{Z_{R}}).
\end{aligned}\end{align}
Using the finite propagation speed of the wave operator (cf. \cite[Appendix D.2]{MaMa07}), we get
\begin{align}\begin{aligned}\label{e.439}
F_{2t/r^{2}}(\sqrt{2t}D^{Z_{R}})(x, x')=&\int_{-\infty}^{+\infty}\cos\left(\sqrt{2t}vD^{Z_{R}}\right)(x, x')
e^{-\frac{v^{2}}{2}}f\left(\frac{\sqrt{2t}v}{r}\right)\frac{dv}{\sqrt{2\pi}}\\
=&\quad 0, \quad \text{if}\,\,{\rm d}(x, x')\geq r.
\end{aligned}\end{align}

Using integration by parts, for $r\geq 1$ and $t>0$, we get (\ref{e.435})
\begin{align}\begin{aligned}\label{e.440}
\sup_{a\in \mathbb{R}}\left|a^{m}\right|\cdot&\big|G_{2t/r^{2}}(\sqrt{2t}a)\big|\\
=&\sup_{a\in \mathbb{R}}|a^{m}|\left|\int_{-\infty}^{+\infty}\cos(ua)e^{-\frac{u^{2}}{4t}}\left(1-f(\frac{u}{r})\right)\frac{du}{2\sqrt{\pi t}}\right|\\
\leq&\frac{1}{2\sqrt{\pi t}}\int_{|u|\geq \frac{r}{2}}\big|\frac{\partial^{m}}{\partial u^{m}}\big(e^{-\frac{u^{2}}{4t}}(1-f(\frac{u}{r}))\big)\big|du\\
\leq&C_{m}\int_{|u|\geq \frac{r}{2}}e^{-\frac{u^{2}}{4t}}Q_{1}\big(\frac{u}{t}, \frac{1}{r}, \frac{1}{t}\big)du
\leq C_{m}\int_{|u|\geq \frac{r}{2}}e^{-\frac{u^{2}}{4t}}Q_{2}\big(\frac{u}{\sqrt{t}}, \frac{1}{r}\big)du\\
\leq &C_{m}\int_{|u|\geq \frac{r}{2}}e^{-\frac{u^{2}}{8t}}Q_{3}(\frac{1}{r})du\leq C_{m}e^{-c\frac{r^{2}}{t}}.
\end{aligned}\end{align}
where $Q_{1}, Q_{2}$ and $Q_{3}$ are certain polynomials with positive coefficients. (We note that to prove $Q_{1}\big(\frac{u}{t}, \frac{1}{r}, \frac{1}{t}\big)\leq Q_{2}\big(\frac{u}{\sqrt{t}}, \frac{1}{r}\big)$ we have used
the facts that $\frac{1}{t}\leq \frac{2u}{t}$ and $\frac{u}{t}\leq \frac{2u^{2}}{t}\leq 2(\frac{u}{\sqrt{t}})^{2}$.)

 Let $H_{R}:=(D^{Z_{R}})^{2}$, then by the spectral theorem and (\ref{e.440}), for $m_{1}, \, m_{2}\in \mathbb{N}$ there exists $C_{m_{1},m_{2}}>0$ such that for any $t>0$ and $s\in \Omega(Z_{R},F_{R}|_{Z_{R}})$
\begin{align}\begin{aligned}\label{e.441}
\left\| H_{R}^{m_{1}}G_{2t/r^{2}}(\sqrt{2t}D^{Z_{R}})H_{R}^{m_{2}}s\right\|_{L^{2}(Z_{R})}\leq C_{m_{1},m_{2}}e^{-c\frac{r^{2}}{t}}\|s\|_{L^{2}(Z_{R})}.
\end{aligned}\end{align}
Now applying \cite[Thm. A.3.4]{MaMa07}, for $m_{1}, \, m_{2}\in \mathbb{N}$ and $\mathscr{R}$ a differential operator of order $m_{1}$ acting
on  $\Lambda(T^{*}Z_{R})\otimes F_{R}$ over $Z_{R}$, there exists $C>0$ such that for any $t>0$ and  $s\in \Omega(Z_{R},F_{R}|_{Z_{R}})$,
\begin{align}\begin{aligned}\label{e.741}
\left| \mathscr{R}G_{2t/r^{2}}(\sqrt{2t}D^{Z_{R}})H_{R}^{m_{2}}s\right|_{\mathscr{C}^{0}(Z_{R})}\leq Ce^{-c\frac{r^{2}}{t}}\|s\|_{L^{2}(Z_{R})}.
\end{aligned}\end{align}
And we have
\begin{align}\begin{aligned}\label{e.742}
\Big(\mathscr{R}G_{2t/r^{2}}&(\sqrt{2t}D^{Z_{R}})H_{R}^{m_{2}}s\Big)(x)\\
&=\int_{Z_{R}}\big(H^{m_{2}}_{R,x'}\mathscr{R}_{x}
G_{2t/r^{2}}(\sqrt{2t}D^{Z_{R}})(x, x')\big)s(x')dv_{Z_{R}}(x'),
\end{aligned}\end{align}
here $H_{R,x'}$ acts on $(\Lambda(T^{*}Z_{R})\otimes F_{R})^{*}$ by identifying $(\Lambda(T^{*}Z_{R})\otimes F_{R})^{*}$ to $\Lambda(T^{*}Z_{R})\otimes F_{R}$ through the metric. Thus uniformly of $x\in Z_{R}$, we have
\begin{align}\begin{aligned}\label{e.442}
 \left\|H^{m_{2}}_{R, \cdot}\mathscr{R}_{x}G_{2t/r^{2}}(\sqrt{2t}D^{Z_{R}})(x,\cdot)\right\|_{L^{2}(Z_{R})} \leq C_{m_{1},m_{2}}e^{-c\frac{r^{2}}{t}}.
\end{aligned}\end{align}
Let $m_{1}+m_{2}\geq m+l$, by applying Sobolev inequality and elliptic estimates to $x'-$variable, from (\ref{e.442}), we get for $x,x'\in Z_{R}$
\begin{align}\begin{aligned}\label{e.443}
\left| G_{2t/r^{2}}(\sqrt{2t}D^{Z_{R}})(x, x') \right|_{\mathscr{C}^{l}}\leq C_{l}e^{-c\frac{r^{2}}{t}},
\end{aligned}\end{align}
where we note that the constants are uniform with respect to $R\geq 1$ in the Sobolev inequalities and elliptic estimates,
 since in our case all the local geometry data are independent of $R$. Then the inequality (\ref{e.424})
follows from (\ref{e.438}), (\ref{e.439}) and (\ref{e.443}) by setting $r=d(x,x')$. The proof has been completed.
\end{proof}

Let $U_{R}$ be compact subset of $Z_{R}$. Let $(\widetilde{D}^{Z_{R}})^{2}$ be another Hodge-Laplacian such that
\begin{align}\begin{aligned}\label{e.425}
(D^{Z_{R}})^{2}=(\widetilde{D}^{Z_{R}})^{2} \quad \quad \text{ on }U_{R}\subset Z_{R}.
\end{aligned}\end{align}
Then we can compare the two heat kernels $e^{-t(D^{Z_{R}})^{2}}(x, x')$ and $e^{-t(\widetilde{D}^{Z_{R}})^{2}}(x, x')$ on a smaller
 compact subset of $U_{R}$.
\begin{rem}\label{r.11}
In our application of the following lemma, the operator $(\widetilde{D}^{Z_{R}})^{2}$ will be taken as $(D^{Z_{1,R}})^{2}$ (resp. $(D^{Z_{2,R}})^{2}$) in Lemma \ref{l.24}. And $U_{R}$ will be taken into the interior part of $Z_{1,R}$ (resp. $Z_{2,R}$) away from the boundary by a distance depending on $R$. So we can suppose that all the geometric data that we used to define $(\widetilde{D}^{Z_{R}})^{2}$ are locally independent of $R$.
\end{rem}
\begin{lemma}\label{l.43} As in Remark \ref{r.11}, we assume that all the geometric data used to define $(\widetilde{D}^{Z_{R}})^{2}$ are locally independent of $R$. There exists $c>0$, such that for any $l\in \mathbb{N}$, there exists $C_{l}>0$ such that for any $t>0$, $R\geq 0$, $r\geq 1$ and
$x, x'\in K=\{x\in U|\, d(x, \partial U)\geq r> 0\}$
\begin{align}\begin{aligned}\label{e.426}
\big|e^{-t(D^{Z_{R}})^{2}}(x, x')-e^{-t(\widetilde{D}^{Z_{R}})^{2}}(x, x')\big|_{\mathscr{C}^{l}}\leq C_{l}e^{-cr^{2}/t}.
\end{aligned}\end{align}
\end{lemma}
\begin{proof}
By (\ref{e.438}), we get
\begin{align}\begin{aligned}\label{e.690}
&\exp\left( -t(D^{Z_{R}})^{2}\right)=F_{2t/r^{2}}(\sqrt{2t}D^{Z_{R}})+G_{2t/r^{2}}(\sqrt{2t}D^{Z_{R}}),\\
&\exp\big( -t(\widetilde{D}^{Z_{R}})^{2}\big)=F_{2t/r^{2}}(\sqrt{2t}\widetilde{D}^{Z_{R}})+G_{2t/r^{2}}(\sqrt{2t}\widetilde{D}^{Z_{R}}).
\end{aligned}\end{align}
Let $B^{Z_{R}}(x,r)$ be the open ball in $Z_{R}$ with center $x$ and radius $r$. Since
 for $x,x'\in Z_{R}$, $F_{2t/r^{2}}(\sqrt{2t}D^{Z_{R}})(x,x')$ (resp. $F_{2t/r^{2}}(\sqrt{2t}\widetilde{D}^{Z_{R}})(x,x')$) only depends on
the restriction of $D^{Z_{R}}$ (resp. $\widetilde{D}^{Z_{R}}$) to $B^{Z_{R}}(x,r)$ (cf. \cite[Appendix D.2]{MaMa07}), we have
\begin{align}\begin{aligned}\label{e.693}
F_{2t/r^{2}}(\sqrt{2t}D^{Z_{R}})(x, x')- F_{2t/r^{2}}(\sqrt{2t}\widetilde{D}^{Z_{R}})(x, x')=0.
\end{aligned}\end{align}
As the proof of (\ref{e.443}), there exists $c>0$, such that for any $l\in \mathbb{N}$ there exists $C_{l}>0$ such that for any $R\geq 1$, $t>0$, $r\geq 1$ and
$x',x \in Z_{R},$
\begin{align}\begin{aligned}\label{e.694}
&\left| G_{2t/r^{2}}(\sqrt{2t}D^{Z_{R}})(x, x') \right|_{\mathscr{C}^{l}}\leq C_{l}e^{-c\frac{r^{2}}{t}},\,
\left| G_{2t/r^{2}}(\sqrt{2t}\widetilde{D}^{Z_{R}})(x, x') \right|_{\mathscr{C}^{l}}\leq C_{l}e^{-c\frac{r^{2}}{t}}.
\end{aligned}\end{align}
By (\ref{e.690}), (\ref{e.693}) and (\ref{e.694}), we get for any $x',x \in K$
\begin{align}\begin{aligned}\label{e.695}
&\big|e^{-t(D^{Z_{R}})^{2}}(x, x')-e^{-t(\widetilde{D}^{Z_{R}})^{2}}(x, x')\big|_{\mathscr{C}^{l}}\\
=&\big|G_{2t/r^{2}}(\sqrt{2t}D^{Z_{R}})(x, x')-G_{2t/r^{2}}(\sqrt{2t}\widetilde{D}^{Z_{R}})(x, x')\big|_{\mathscr{C}^{l}}\leq C_{l}e^{-c\frac{r^{2}}{t}}.
\end{aligned}\end{align}
From (\ref{e.695}), we get (\ref{e.426}). The proof is completed.
\end{proof}

\subsection{The gluing problem of analytic torsion forms}\label{ss1.8}

 Recall that $M$, $M_{1}$ and $M_{2}$ are the fibrations described in (\ref{e.777}).
For $\varepsilon>0$, we \textbf{assume} that $(T^{H}M, g^{TZ}, h^{F})$ verify the assumption of product structures (\ref{e.7611}) on the product neighborhood $X_{[-\varepsilon,\varepsilon]}$ of $X$ in $M$.

\textbf{
From now on, we always apply the absolute boundary conditions to $(M_{1}, X)$ and the relative boundary conditions to $(M_{2}, X)$.}

Let $F^{*}$ be the dual flat vector bundle of $F$. Let $H_{\bullet}(Z, F^{*})=\bigoplus^{m}_{p=0}H_{p}(Z, F^{*})$ (resp.  $H_{\bullet}(Z_{1}, F^{*})$, $H_{\bullet}(Z_{2}, Y, F^{*})$)
denote the singular homology of $Z$ (resp.  $Z_{1}$, $(Z_{2}, Y)$) with coefficients in $F^{*}$, and let $H^{\bullet}(Z, F)=\bigoplus_{p=0}^{m}H^{p}(Z, F)$
(resp.  $H^{\bullet}(Z_{1}, F)$, $H^{\bullet}(Z_{2}, Y, F)$) denote the singular cohomology of $Z$ (resp.  $Z_{1}$, $(Z_{2}, Y)$)
 with coefficients in $F$. Then for $0\leq p\leq m$, we have canonical identifications
\begin{align}\begin{aligned}\label{e.631}
&H_{p}(Z, F^{*})=(H^{p}(Z, F))^{*}, \quad H_{p}(Z_{1}, F^{*})=(H^{p}(Z_{1}, F))^{*}, \\ &H_{p}(Z_{2}, Y, F^{*})=(H^{p}(Z_{2}, Y, F))^{*}.\end{aligned}\end{align}

\begin{defn}\label{d.21}
For $\varepsilon>0$, let $\mathbb{K}_{Z}$ be the smooth triangulation of $Z$, such that it induces smooth sub-triangulations of
 $Y$, $Y_{[-\varepsilon,0]}$, $Y_{[0,\varepsilon]}$, $Z_{1}$, $Z_{2}$ denoted by $\mathbb{K}_{Y}$,
$\mathbb{K}_{Y_{[-\varepsilon,0]}}$, $\mathbb{K}_{Y_{[0,\varepsilon]}}$, $\mathbb{K}_{Z_{1}}$, $\mathbb{K}_{Z_{2}}$.
\end{defn}

The smooth triangulations $\mathbb{K}_{Z}$ (resp. $\mathbb{K}_{Z_{i}} $) consists of a finite set of simplex, $\mathfrak{a}$,
whose orientation is fixed once and for all.
Let $B$ be the finite subset of $Z$ of the barycenters of the simplexes in $\mathbb{K}_{Z}$.
Let $b:\mathbb{K}_{Z}\rightarrow B$ and
 $\sigma:B\rightarrow \mathbb{K}_{Z}$
 denote the obvious one-to-one maps. For $0\leq p\leq m$, $i=1, 2$,
let $\mathbb{K}^{p}_{Z}$ (resp.  $\mathbb{K}^{p}_{Z_{i}}$) be the union of the simplexes in $\mathbb{K}_{Z}$ of dimension $\leq p$, such that for
 $0\leq p\leq m$, $\mathbb{K}^{p}_{Z}\backslash \mathbb{K}^{p-1}_{Z}$ (resp.  $\mathbb{K}^{p}_{Z_{i}}\backslash \mathbb{K}^{p-1}_{Z_{i}}$ ) is
the union of simplexes of dimension $p$. If $\mathfrak{a}\in \mathbb{K}_{Z}$, let $[\mathfrak{a}]$ be the real line generated by $\mathfrak{a}$.
Let $(C_{\bullet}(\mathbb{K}_{Z}, F^{*}), \partial)$ be the complex of simplicial chains in $\mathbb{K}_{Z}$ with values in $F^{*}$. For
$0\leq p\leq m,\,i=1,2$, we define
\begin{align}\begin{aligned}\label{e.630}
&C_{p}(\mathbb{K}_{Z}, F^{*}):=\sum_{\mathfrak{a}\in \mathbb{K}^{p}_{Z}\backslash \mathbb{K}^{p-1}_{Z}}
[\mathfrak{a}]\otimes_{\mathbb{R}}F^{*}_{b(\mathfrak{a})}.
\end{aligned}\end{align}
The boundary operator $\partial$ sends $C_{p}(\mathbb{K}_{Z}, F^{*})$ into $C_{p-1}(\mathbb{K}_{Z}, F^{*})$. Set
\begin{align}\begin{aligned}\label{e.632}
C_{\bullet}\left(\mathbb{K}_{Z_{2}}/\mathbb{K}_{Y}, F^{*}\right):=C_{\bullet}(\mathbb{K}_{Z}, F^{*})/C_{\bullet}(\mathbb{K}_{Z_{1}}, F^{*}).
\end{aligned}\end{align}
The homologies of the complexes
$\big(C_{\bullet}(\mathbb{K}_{Z_{1}}, F^{*}), \partial\big)
\,\text{ and }\,\big(C_{\bullet}\left(\mathbb{K}_{Z_{2}}/\mathbb{K}_{Y}, F^{*}\right), \partial\big)$
 are canonically identified with the
singular homologies, respectively,
$H_{\bullet}(Z_{1}, F^{*})\,\text{ and }\,H_{\bullet}(Z_{2}, Y, F^{*}).$
 Naturally, we have a short exact sequence of chain groups:
\begin{align}\begin{aligned}\label{e.174}
\xymatrix{
0\ar[r]&C_{p}\left(\mathbb{K}_{Z_{1}}, F^{*}\right)\ar[r]^{i_{p}}\ar[d]^{\partial}
&C_{p}\left(\mathbb{K}_{Z}, F^{*}\right) \ar[r]^{j_{p}}\ar[d]^{\partial}
&C_{p}\left(\mathbb{K}_{Z_{2}}/\mathbb{K}_{Y}, F^{*}\right)\ar[r]\ar[d]^{\partial}
&0\, .\\
&&&&
}
\end{aligned}\end{align}

 If $\mathfrak{a}\in \mathbb{K}_{Z}$, let $[\mathfrak{a}]^{*}$ be the line dual to the line $[\mathfrak{a}]$.
Let $\big(C^{\bullet}(\mathbb{K}_{Z}, F), \widetilde{\partial}\big)$ be the complex dual to the complex
$\big(C_{\bullet}(\mathbb{K}_{Z}, F^{*}), \partial\big)$. In particular, for $0\leq p\leq m,\, i=1,2$, we have the identity
\begin{align}\begin{aligned}\label{e.744}
&C^{p}(\mathbb{K}_{Z}, F)=\sum_{\mathfrak{a}\in \mathbb{K}^{p}_{Z}\backslash \mathbb{K}^{p-1}_{Z}}
[\mathfrak{a}]^{*}\otimes_{\mathbb{R}}F_{b(\mathfrak{a})}.
\end{aligned}\end{align}

Let $\big(C^{p}\left(\mathbb{K}_{Z_{1}}, F\right), \widetilde{\partial}\big)$ be the dual complex of
$\big(C_{p}\left(\mathbb{K}_{Z_{1}}, F^{*}\right),\partial\big)$ and
$\big(C^{p}\left(\mathbb{K}_{Z_{2}}/\mathbb{K}_{Y}, F\right), \widetilde{\partial}\big)$ be the dual complex of
 $\big(C_{p}\left(\mathbb{K}_{Z_{2}}/\mathbb{K}_{Y}, F^{*}\right),\partial\big)$.
Naturally, we have the following short exact sequence of cochain groups
\begin{align}\begin{aligned}\label{e.175}
\xymatrix{
&&&&\\
0\ar[r]&C^{p}\left(\mathbb{K}_{Z_{2}}/\mathbb{K}_{Y}, F\right)\ar[r]^{j^{*}_{p}}\ar[u]^{\widetilde{\partial}}
&C^{p}\left(\mathbb{K}_{Z}, F\right) \ar[r]^{i^{*}_{p}}\ar[u]^{\widetilde{\partial}}
&C^{p}\left(\mathbb{K}_{Z_{1}}, F\right)\ar[r]\ar[u]^{\widetilde{\partial}}
&0,
}
\end{aligned}\end{align}
where $i^{*}_{p},\,j^{*}_{p}$ denote the dual maps of $i_{p},\,j_{p}$. The double complex (\ref{e.175}) yields a long exact sequence $(\mathscr{H}, \delta)$ \index{$(\mathscr{H}, \delta)$} of cohomology groups, i.e.,
\begin{align}\begin{aligned}\label{e.434}
\cdots \longrightarrow H^{p}(Z, F)\overset{\delta}{\longrightarrow} H^{p}(Z_{1}, F) \overset{\delta}{\longrightarrow}
H^{p+1}(Z_{2}, Y, F) \overset{\delta}{\longrightarrow}\cdots.
\end{aligned}\end{align}

\begin{defn}\label{d.14}
\textbf{(De Rham map)} Let $ \sigma\in \Omega^{\bullet}(Z, F), \mathfrak{a}\in C_{\bullet}\left(\mathbb{K}_{Z}, F^{*}\right)$,
we define a map $P^{\infty}:\Omega^{\bullet}(Z, F)\rightarrow C^{\bullet}\left(\mathbb{K}_{Z}, F\right)$ by
\begin{align}\begin{aligned}\label{e.176}
P^{\infty}(\sigma)(\mathfrak{a})=\int_{\mathfrak{a}}\sigma \, .
\end{aligned}\end{align}
Similarly, we can define $P^{\infty}_{1}:\Omega^{\bullet}(Z_{1}, F)\rightarrow C^{\bullet}\left(\mathbb{K}_{Z_{1}}, F\right)$ and
$P^{\infty}_{2}:\Omega^{\bullet}(Z_{2}, Y, F)\rightarrow C^{\bullet}\left(\mathbb{K}_{Z_{2}}/\mathbb{K}_{Y}, F\right)$.
\end{defn}\index{$P^{\infty}$, $P^{\infty}_{1}$,$P^{\infty}_{2}$}

The map (\ref{e.176}) induces isomorphisms from the bundle of harmonic forms to the bundle of cohomology group
\begin{align}\begin{aligned}\label{e.177}
&P^{\infty}:\mathscr{H}^{p}\left(Z, F\right)\cong H^{p}(Z, F)\\
(\text{resp.}\,\,&P^{\infty}_{1}:\mathscr{H}^{p}\left(Z_{1}, F\right)\cong H^{p}(Z_{1}, F),\,\,
 P^{\infty}_{2}:\mathscr{H}^{p}\left(Z_{2}, Y, F\right)\cong H^{p}(Z_{2}, Y, F)\,).
\end{aligned}\end{align}

\subsection{The stretching diffeomorphisms}\label{ss1.9}
For $R\geq 0$, let
\begin{align}\begin{aligned}\label{e.625}
M_{1, R}:=M_{1}\cup_{X}X_{[0, R]}, \quad M_{2, R}:=M_{2}\cup_{X}X_{[-R, 0]}.
\end{aligned}\end{align}
Then $M_{1, R}$ has a cylinder end $X_{[-\varepsilon, R]}$ and $M_{2, R}$ has a cylinder end $X_{[-R, \varepsilon]}$. Then by change of coordinates as in (\ref{e.614}),
we will always identify the cylinder end of $M_{1, R}$ with $X_{[-R-\varepsilon, 0]}$
and that of  $M_{2, R}$ with $X_{[0, R+\varepsilon]}$.
So we will identify $X$ with the common boundary $X\times\{0\}$ for both $M_{1,R}$ and $M_{2,R}$.
Set
 \begin{align}\begin{aligned}\label{e.628}
M_{R}=M_{1, R}\cup_{X} M_{2, R},
\end{aligned}\end{align}
then $M_{R}$ has a cylinder part $X_{[-R-\varepsilon, R+\varepsilon]}$ (see Figure \ref{fig.4}).

To apply the adiabatic methods, for $i=1, 2$, we begin to construct diffeomorphisms $\phi_{R}:M\rightarrow M_{R}$
and $\phi_{i, R}:M_{i}\rightarrow M_{i, R}$.\\

\begin{lemma}\label{l.1}
There exist a diffeomorphism $\phi_{R}:M\longrightarrow M_{R}$\index{$\phi_{R}$} such that: \\
{\rm(1)} The diffeomorphism $\phi_{R}$ restricted on the submanifold
$M\setminus X_{(-\frac{7\varepsilon}{8}, \frac{7\varepsilon}{8})} $ is an identity map to
 $M_{R}\setminus X_{(-R+\frac{\varepsilon}{8}, R-\frac{\varepsilon}{8})}$, and
\begin{align}\begin{aligned}\label{e.56}
\phi_{R}:X_{(-\frac{7\varepsilon}{8}, \frac{7\varepsilon}{8})} \rightarrow X_{
(-R+\frac{\varepsilon}{8}, R-\frac{\varepsilon}{8})} \text{ is an one-to-one map.}
\end{aligned}\end{align}
{\rm (2)} The following diagrams are commutative,
\begin{align}\begin{aligned}\label{e.57}
\xymatrix{
M\ar[r]^{\phi_{R}}\ar[d]_{\pi}&M_{R}, \ar[dl]^{\pi_{R}}\\
S&
}\quad
\quad\quad\quad\quad
\xymatrix{
X\times \left(-R, R\right)\ar[r]^{\quad\quad \quad\psi_{R}}& X.\\
 X\times(-\varepsilon, \varepsilon)\ar[ur]_{\psi_{\varepsilon}}\ar[u]^{\phi_{R}}&
}
\end{aligned}\end{align}
\end{lemma}

\begin{proof}

Let $\rho(x_{m}):[0, \varepsilon]\rightarrow [0,1]$ be a cut-off function, which is equal to
$0$ on $[0, \frac{1}{8}\varepsilon]$
and equal
to $1$ on $[\frac{2}{8}\varepsilon, \varepsilon]$.
Let $\chi(x_{m}):[0, \varepsilon]\rightarrow [0,1]$ be a cut-off function, which is equal to
$0$ on $[0, \frac{6}{8}\varepsilon]$
and equal
to $1$ on $[\frac{7}{8}\varepsilon,\varepsilon]$. We put $g_{R}(x_{m})=\frac{4R-\varepsilon}{3\varepsilon}x_{m}-\frac{R-\varepsilon}{6}$ and
then define
\begin{align}\begin{aligned}\label{e.734}
h_{R}(x_{m})=x_{m}\big(1-\rho(x_{m})\big)+\rho(x_{m})g_{R}(x_{m}).
\end{aligned}\end{align}
Then we set
\begin{align}\begin{aligned}\label{e.59}
\phi_{R}(x_{m})=h_{R}(x_{m})\big(1-\chi(x_{m})\big)+\chi(x_{m})(x_{m}-\varepsilon+R).
\end{aligned}\end{align}
It is easy to see that $\phi_{R}$ is a smooth function such that $\phi_{R}(0)=0$ and $\phi_{R}(\varepsilon)=R$. We extend $\phi_{R}$ from
$[0,\varepsilon]$ to $[-\varepsilon,\varepsilon]$ by setting $\phi_{R}(x_{m})=-\phi_{R}(-x_{m})$ for $x_{m}<0$. Then we see that the
extended $\phi_{R}$, denoted with the same notation, is a smooth function on $[-\varepsilon,\varepsilon]$, such that
\begin{align}\begin{aligned}\label{e.735}
\phi_{R}(-\varepsilon)=-R,\,\phi_{R}(0)=0 \quad{\rm and }\quad \phi_{R}(\varepsilon)=R.
\end{aligned}\end{align}

We make a smooth function
 $\phi_{R}:X_{[-\varepsilon, \varepsilon]}\rightarrow X_{[-R, R]}$ such that
\begin{align}\begin{aligned}\label{e.60}
\phi_{R}(y, x_{m}):=\big(y,\phi_{R}(x_{m})\big), \quad \text{for }(y, x_{m})\in X_{[-\varepsilon, \varepsilon]},
\end{aligned}\end{align}
and outside $X_{[-\varepsilon, \varepsilon]}$ in $M$, we define $\phi_{R}$ to be the identity map.
By our construction of $\phi_{R}(x_{m})$, we have for $R>\varepsilon$ and $x_{m}\in [-\varepsilon,\varepsilon]$
\begin{align}\begin{aligned}\label{e.61}
\frac{\partial \phi_{R}(x_{m})}{\partial x_{m}}\geq 1.
\end{aligned}\end{align}
From (\ref{e.735}) and (\ref{e.61}), we deduce that $\phi_{R}$ is a diffeomorphism.
\end{proof}

Let $\phi_{i, R}:M_{i}\longrightarrow M_{i, R}$ be the restriction of $\phi_{R}$ from $M_{R}$ to $M_{i,R}$, for $i=1,2$.\\

By our assumptions of product structures, we have
\begin{align}\begin{aligned}
\left(T^{H}M, F, h^{F}, \nabla^{F}\right)\big|_{X\times [-\varepsilon, \varepsilon]}=\psi^{*}_{\varepsilon}\left(T^{H}M, F, h^{F}, \nabla^{F}\right)\big|_{X}.
\end{aligned}\end{align}
As in Section \ref{ss1.1}, we can extend naturally all geometrical data from $M$, $M_{i}$ to $M_{R}$, $M_{i, R}$ by using the assumptions of product structures near $X$ on $M$, such that
\begin{align}\begin{aligned}\label{e.67}
&\left(T^{H}M_{R}, F_{R}, h^{F_{R}}, \nabla^{F_{R}}\right)\big|_{X\times(-R-\varepsilon, R+\varepsilon)}=\psi^{*}_{R}\left(T^{H}M, F, h^{F}, \nabla^{F}\right)\big|_{X}, \\
&g^{TZ_{R}}(y, x_{m})=g^{TY}(y)+dx^{2}_{m}, \quad \quad (y, x_{m})\in X\times [-R, R].
\end{aligned}\end{align}
We get the geometrical data for $M_{1, R}$ and $M_{2, R}$ by restrictions.

By our construction of $\phi_{R}$ in Lemma \ref{l.1}, for $i=1, 2$ we have
\begin{align}\begin{aligned}\label{e.70}
&\left(T^{H}M, F, h^{F}, \nabla^{F}\right)=\phi^{*}_{R}\left(T^{H}M_{R}, F_{R}, h^{F_{R}}, \nabla^{F_{R}}\right)\\
(\text{resp.}\,\,&\left(T^{H}M_{i}, F, h^{F}, \nabla^{F}\right)=\phi^{*}_{i, R}\left(T^{H}M_{i, R}, F_{R}, h^{F_{R}}, \nabla^{F_{ R}}\right)\,).
\end{aligned}\end{align}

Let
\begin{align}\begin{aligned}\label{e.782}
g^{TZ}_{R}:=\phi^{*}_{R}\left(g^{TZ_{R}}\right)\quad ( \text{resp.} \quad g^{TZ_{i}}_{R}:=\phi^{*}_{i, R}\left(g^{TZ_{i,R}}\right)).
\end{aligned}\end{align}
 Then we get by (\ref{e.70})
\begin{align}\begin{aligned}\label{e.72}
 \mathscr{T}(T^{H}M_{R}, g^{TZ_{R}}, h^{F_{R}})&=\mathscr{T}\left(\phi^{*}_{R}\left(T^{H}M_{R}\right), \phi^{*}_{R}\left(g^{TZ_{R}}\right), \phi^{*}_{R}\left(h^{F_{R}}\right)\right)\\
 &= \mathscr{T}\left(T^{H}M,g^{TZ}_{R} , h^{F}\right),
\end{aligned}\end{align}
and similarly we have by (\ref{e.70})
\begin{align}\begin{aligned}\label{e.73}
 \mathscr{T}_{\rm abs}(T^{H}M_{1, R}, g^{TZ_{1, R}}, h^{F_{ R}})
 = \mathscr{T}_{\rm abs}\left(T^{H}M_{1},g^{TZ_{1}}_{R} , h^{F}\right),\\
  \mathscr{T}_{\rm rel}(T^{H}M_{2, R}, g^{TZ_{2, R}}, h^{F_{ R}})
 = \mathscr{T}_{\rm rel}\left(T^{H}M_{2},g^{TZ_{2}}_{R} , h^{F}\right).
\end{aligned}\end{align}

\begin{rem}\label{r.5}
Through the above arguments, we see that in fact there are two equivalent geometrical settings in studying our problems,
which can be transformed to each other by the stretching diffeomorphisms $\phi_{R}$ (resp.  $\phi_{i, R}$, for $i=1, 2$).
The first one is the setting on the fibration $M$ (resp.  $M_{i}$ ) before the stretch with the following geometrical data:
\begin{align}\begin{aligned}\label{e.427}
 h^{F}, \nabla^{F}, T^{H}M \text{ and } g^{TZ}_{R}\quad\left(\text{resp.  } h^{F}, \nabla^{F}, T^{H}M_{i} \text{ and } g^{TZ_{i}}_{R}\right).
 \end{aligned}\end{align}
The second one is the setting on the stretched fibration $M_{R}$ (resp.  $M_{i, R}$ ) with the data:
\begin{align}\begin{aligned}\label{e.429}
 h^{F_{R}}, \nabla^{F_{R}}, T^{H}M_{R} \text{ and } g^{TZ_{R}}\quad\left(\text{resp.  } h^{F_{R}}, \nabla^{F_{R}}, T^{H}M_{i, R} \text{ and } g^{TZ_{i, R}}\right).\end{aligned}\end{align}
\end{rem}

 Moreover, we can compute $g^{TZ}_{R}$ (resp.  $g^{TZ_{i}}_{R}$) explicitly
on the cylinder part $X_{(-\varepsilon, \varepsilon)}$.

\begin{lemma}\label{l.44}
 There exist smooth even functions $\lambda_{2},\lambda_{1},\lambda_{0}$ on $[-\varepsilon,\varepsilon]$ with supports
in $[-\frac{7}{8}\varepsilon,-\frac{\varepsilon}{8}]\cup[\frac{\varepsilon}{8}, \frac{7}{8}\varepsilon]$
such that for $(y, x_{m})\in X_{(-\varepsilon,\varepsilon)}$
\begin{align}\begin{aligned}\label{e.430}
g^{TZ}_{R}=g^{TZ}(y,x_{m})+\big(\lambda_{0}(x_{m})+\lambda_{1}(x_{m})R+\lambda_{2}(x_{m})R^{2}\big)dx^{2}_{m}.
\end{aligned}\end{align}
Moreover, for $i=1, 2$, we have the same expression for $g^{TZ_{i}}_{R}:=\big(g^{TZ}_{R}\big)|_{Z_{i}}$.
\end{lemma}
\begin{proof}
By (\ref{e.59}) and (\ref{e.60}), we get that for $(y, x_{m})\in X_{(-\varepsilon, \varepsilon)}$
\begin{align}\begin{aligned}\label{e.431}
g^{TZ}_{R}(y, x_{m})
=g^{TY}(y)+\big(\frac{\partial \phi_{R}}{\partial x_{m}}\big)^{2}dx^{2}_{m}.
\end{aligned}\end{align}
By our construction of $\phi_{R}(x_{m})$, we see that $\frac{\partial \phi_{R}}{\partial x_{m}}$ is an even smooth function depending
linearly on $R$, moreover we have
$\frac{\partial \phi_{R}}{\partial x_{m}}=1$ for $x_{m}\in [-\varepsilon,-\frac{7\varepsilon}{8}]\cup[-\frac{\varepsilon}{8}, \frac{\varepsilon}{8}]\cup [\frac{7\varepsilon}{8},\varepsilon]$.
Consequently, there exist two smooth even functions $\mu_{1}(x_{m}),\,\mu_{0}(x_{m}):[-\varepsilon, \varepsilon]\rightarrow [0,\infty)$ with support
in $[-\frac{7}{8}\varepsilon,-\frac{\varepsilon}{8}]\cup [\frac{\varepsilon}{8}, \frac{7}{8}\varepsilon]$ such that (see (\ref{e.61}))
\begin{align}\begin{aligned}\label{e.432}
\frac{\partial \phi_{R}}{\partial x_{m}}=1+ \mu_{0}(x_{m})+\mu_{1}(x_{m})\cdot R\geq1.
\end{aligned}\end{align}
Set $\lambda_{2}=\mu^{2}_{1},\lambda_{1}=2\mu_{1}(1+\mu_{0}),\lambda_{0}=2\mu_{0}(1+\mu_{0})$,
then by (\ref{e.431}) and (\ref{e.432}), for $(y, x_{m})\in X_{(-\varepsilon, \varepsilon)}$
\begin{align}\begin{aligned}\label{e.433}
&g^{TZ}_{R}(y, x_{m})=g^{TY}(y)+\big(1+\mu_{0}(x_{m})+\mu_{1}(x_{m})R\big)^{2}dx^{2}_{m}\\
=&g^{TZ}(y, x_{m})+dx^{2}_{m}
+\big(\lambda_{0}(x_{m})+\lambda_{1}(x_{m})R+\lambda_{2}(x_{m})R^{2}\big)dx^{2}_{m}.
\end{aligned}\end{align}
Since $\phi_{R}$ is an identity map on $M\backslash X_{[-\varepsilon, \varepsilon]}$, (\ref{e.430}) follows from (\ref{e.433}).
\end{proof}

\subsection{An identity in the process of adiabatic limit}\label{ss1.10}

\begin{defn}\label{d.21}
Recall that $\mathbb{K}_{Z}$ is the smooth triangulation of $Z$ in Definition \ref{d.21},
then the stretching diffeomorphism $\phi_{R}:Z\longrightarrow Z_{R}$ induces a smooth triangulation
 $\mathbb{K}_{Z_{R}}$ of $Z_{R}$ from $\mathbb{K}_{Z}$, such that there are smooth sub-triangulations $\mathbb{K}_{Y}$, $\mathbb{K}_{Z_{1, R}}$
 and $\mathbb{K}_{Z_{2, R}}$ by our construction of $\phi_{R}$.
\end{defn}

\begin{rem}\label{r.1}
 Consequently, the volume of simplex in $\mathbb{K}_{Y_{[-R,R]}}=\mathbb{K}_{Y_{[-R,0]}}\cup \mathbb{K}_{Y_{[0,R]}}$ growths linearly with
respect to the Riemannian volume form induced by $g^{TZ_{R}}$ when $R$ goes to infinity, while the volume of the other simplexes
are unchanged.
\end{rem}

Analogue to (\ref{e.434}), we have the long exact sequence $(\mathscr{H}_{R}, \nabla^{\mathscr{H}_{R}})$ of flat vector bundles:
\begin{align}\begin{aligned}\label{e.434}
\cdots \overset{\delta_{R}}\longrightarrow H^{p}(Z_{R}, F_{R})\overset{\delta_{R}}\longrightarrow H^{p}(Z_{1, R}, F_{R}) \overset{\delta_{R}}\longrightarrow H^{p+1}(Z_{2, R}, Y, F_{R}) \overset{\delta_{R}}\longrightarrow\cdots.
\end{aligned}\end{align}
Here we use $\nabla^{\mathscr{H}_{R}}$ to denote the canonical flat connection on $\mathscr{H}_{R}$ induced by $\nabla^{F}$.

Let $P^{\infty}_{R}:\Omega(Z_{R}, F)\rightarrow C^{\bullet}\left(\mathbb{K}_{Z_{R}}, F\right)$, $P^{\infty}_{1, R}:\Omega(Z_{1, R}, F_{R})\rightarrow C^{\bullet}\left(\mathbb{K}_{Z_{1,R}}, F_{R}\right)$ and
$P^{\infty}_{2, R}:\Omega(Z_{2, R}, Y, F_{R})\rightarrow C^{\bullet}\left(\mathbb{K}_{Z_{2, R}}/\mathbb{K}_{Y}, F_{R}\right)$ be the de Rham maps introduced in
Definition \ref{d.14}.
They induce the isomorphisms between the space of harmonic forms and the cohomology groups, respectively,
\begin{align}\begin{aligned}\label{e.177}
&P^{\infty}_{R}:\mathscr{H}^{p}\left(Z_{R}, F_{R}\right)\cong H^{p}(Z_{R}, F_{R});\,\,
P^{\infty}_{1, R}:\mathscr{H}^{p}\left(Z_{1, R}, F_{R}\right)\cong H^{p}(Z_{1, R}, F_{R}); \\
& P^{\infty}_{2, R}:\mathscr{H}^{p}\left(Z_{2, R}, Y, F_{R}\right)\cong H^{p}(Z_{2, R}, Y, F_{R}).
\end{aligned}\end{align}

We use $\nabla_{R}^{TZ}$ (resp.  $\nabla_{R}^{TZ_{1}}$, $\nabla_{R}^{TZ_{2}}$) to denote the Levi-Civita connection with respect to $g^{TZ}_{R}$
(resp.  $g^{TZ_{1}}_{R}$, $g^{TZ_{2}}_{R}$ ), and use $h_{L^{2}, R}^{H(Z, F)}$
(resp.  $h_{L^{2}, R}^{H(Z_{1}, F)}$, $h_{L^{2}, R}^{H(Z_{2}, Y, F)}$) to denote the $L^{2}-$metric induced by $g^{TZ}_{R}$
(resp.  $g^{TZ_{1}}_{R}$, $g^{TZ_{2}}_{R}$ ) and $h^{F}$.
By the construction of $\phi_{R}$, $\phi_{1, R}$ and $\phi_{2, R}$ in Lemma \ref{l.1} and (\ref{e.782}), we see that
\begin{align}\begin{aligned}\label{e.746}
g^{TZ}_{R}=g^{TZ_{1}}_{R}\cup_{X}g^{TZ_{2}}_{R}.
\end{aligned}\end{align}
By Lemma \ref{l.44}, we see that on $X_{(-\frac{\varepsilon}{8},0]}$ (resp. $X_{[0,\frac{\varepsilon}{8})}$)
$g^{TZ_{1}}_{R}$ (resp. $g^{TZ_{2}}_{R}$) has the product structure (\ref{e.360}),
it means that they satisfy the condition of Theorem \ref{t.5}.

If we let $\mathscr{T}_{\rm abs}(T^{H}M_{1}, g^{TZ_{1}}, h^{F})$ (resp. $\mathscr{T}_{\rm rel}(T^{H}M_{2}, g^{TZ_{2}}, h^{F})$) to denote the analytic torsion form on $M_{1}$ (resp. $M_{2}$) with absolute (relative) boundary conditions,
then by \cite[Thm. 3.24]{BL}, Theorem \ref{t.5}, (\ref{e.72}) and (\ref{e.73}), we have in $Q^{S}/Q^{S, 0}$
\begin{align}\begin{aligned}\label{e.74}
 \mathscr{T}&\left(T^{H}M_{R}, g^{TZ_{R}}, h^{F_{R}}\right)-\mathscr{T}\left(T^{H}M, g^{TZ}, h^{F}\right)\\
&= \mathscr{T}\left(T^{H}M, g^{TZ}_{R}, h^{F}\right)-\mathscr{T}(T^{H}M, g^{TZ}, h^{F})\\
&=\int_{Z}\widetilde{e}(TZ, \nabla^{TZ}, \nabla_{R}^{TZ})f(\nabla^{F}, h^{F})
-\widetilde{f}(\nabla^{H(Z, F)}, h_{L^{2}}^{H(Z, F)}, h_{L^{2}, R}^{H(Z, F)}),
\end{aligned}\end{align}
\begin{align}\begin{aligned}\label{e.75}
 \mathscr{T}_{\rm abs}&(T^{H}M_{1, R}, g^{TZ_{1, R}}, h^{F_{R}})-\mathscr{T}_{\rm abs}(T^{H}M_{1}, g^{TZ_{1}}, h^{F})\\
&= \mathscr{T}_{\rm abs}(T^{H}M_{1}, g^{TZ_{1}}_{R}, h^{F})-\mathscr{T}_{\rm abs}(T^{H}M_{1}, g^{TZ_{1}}, h^{F})\\
&=\int_{Z_{1}}\widetilde{e}(TZ_{1}, \nabla^{TZ_{1}}, \nabla_{R}^{TZ_{1}})f(\nabla^{F}, h^{F})-\widetilde{f}(\nabla^{H(Z_{1}, F)}, h_{L^{2}}^{H(Z_{1}, F)}, h_{L^{2}, R}^{H(Z_{1}, F)}),
\end{aligned}\end{align}
\begin{align}\begin{aligned}\label{e.76}
& \mathscr{T}_{\rm rel}(T^{H}M_{2, R}, g^{TZ_{2, R}}, h^{F_{R}})-\mathscr{T}_{\rm rel}(T^{H}M_{2}, g^{TZ_{2}}, h^{F})\\
=& \mathscr{T}_{\rm rel}(T^{H}M_{2}, g^{TZ_{2}}_{R}, h^{F})-\mathscr{T}_{\rm rel}(T^{H}M_{2}, g^{TZ_{2}}, h^{F})\\
=&\int_{Z_{2}}\widetilde{e}(TZ_{2}, \nabla^{TZ_{2}}, \nabla_{R}^{TZ_{2}})f(\nabla^{F}, h^{F})-\widetilde{f}(\nabla^{H(Z_{2}, Y, F)}, h_{L^{2}}^{H(Z_{2}, Y, F)}, h_{L^{2}, R}^{H(Z_{2}, Y, F)}).
\end{aligned}\end{align}

For the long exact sequence $(\mathscr{H}, \nabla^{\mathscr{H}})$ in (\ref{e.434}) of flat vector bundles over $S$,
by \cite[Thm. 2.24]{BL} we have in $Q^{S}/Q^{S, 0}$
\begin{align}\begin{aligned}\label{e.79}
T_{f}(A^{\mathscr{H}}, h_{L^{2}}^{\mathscr{H}})-T_{f}(A^{\mathscr{H}}, h_{L^{2}, R}^{\mathscr{H}})
=-\widetilde{f}(\nabla^{\mathscr{H}}, h_{L^{2}}^{\mathscr{H}}, h_{L^{2}, R}^{\mathscr{H}}).
\end{aligned}\end{align}
From \cite[Def. 1.12]{BL}, we get
\begin{align}\begin{aligned}\label{e.81}
&-\widetilde{f}\left(\nabla^{\mathscr{H}}, h_{L^{2}}^{\mathscr{H}}, h_{L^{2}, R}^{\mathscr{H}}\right)=
\widetilde{f}\left(\nabla^{H(Z, F)}, h_{L^{2}}^{H(Z, F)}, h_{L^{2}, R}^{H(Z, F)}\right)\\
&-\widetilde{f}\left(\nabla^{H(Z_{1}, F)}, h_{L^{2}}^{H(Z_{1}, F)}, h_{L^{2}, R}^{H(Z_{1}, F)}\right)
-\widetilde{f}\left(\nabla^{H(Z_{2}, Y, F)}, h_{L^{2}}^{H(Z_{2}, Y, F)}, h_{L^{2}, R}^{H(Z_{2}, Y, F)}\right).\end{aligned}\end{align}

Hence, from (\ref{1.106}), (\ref{e.746})--(\ref{e.81}), we obtain that in $Q^{S}/Q^{S, 0}$

\begin{align}\begin{aligned}\label{e.80}
\mathscr{T}&(T^{H}M, g^{TZ}, h^{F})-\mathscr{T}_{\rm abs}(T^{H}M_{1}, g^{TZ_{1}}, h^{F})\\
&\quad\quad\quad\quad\quad-\mathscr{T}_{\rm rel}(T^{H}M_{2}, g^{TZ_{2}}, h^{F})-T_{f}(A^{\mathscr{H}}, h_{L^{2}}^{\mathscr{H}})\\
&=\mathscr{T}(T^{H}M_{R}, g^{TZ_{R}}, h^{F_{R}})-\mathscr{T}_{\rm abs}(T^{H}M_{1, R}, g^{TZ_{1, R}}, h^{F_{R}})\\
&\quad\quad\quad\quad-\mathscr{T}_{\rm rel}(T^{H}M_{2, R}, g^{TZ_{2, R}}, h^{F_{R}})-T_{f}(A^{\mathscr{H}}, h_{L^{2}, R}^{\mathscr{H}}).
\end{aligned}\end{align}

For the long exact sequence $(\mathscr{H}_{R}, \nabla^{\mathscr{H}_{R}})$ of flat vector bundle over $S$ introduced in (\ref{e.434}),
we have the following lemma:
\begin{lemma}\label{l.3} The following identity holds in $Q^{S}/Q^{S, 0}$
\begin{align}\begin{aligned}\label{e.83}
T_{f}(A^{\mathscr{H}}, h_{L^{2}, R}^{\mathscr{H}})=T_{f}(A^{\mathscr{H}_{R}}, h_{L^{2}}^{\mathscr{H}_{R}}).\end{aligned}\end{align}
\end{lemma}
\begin{proof}
We have the following flat double complex of complex vector bundles (cf. \cite[Appendix I, (d)]{BL}) on $S$
\begin{align}\begin{aligned}\label{e.783}
\xymatrix{
\cdots\ar[r]&H^{\bullet}(Z_{2,R},Y,F_{2,R})\ar[d]^{\phi_{2,R}^{*}}\ar[r]& H^{\bullet}(Z_{R},F_{R})\ar[d]^{\phi_{R}^{*}}\ar[r]&H^{\bullet}(Z_{1,R},F_{1,R})\ar[d]^{\phi_{1,R}^{*}}\ar[r]& \cdots\\
\cdots\ar[r]&H^{\bullet}(Z_{2},Y,F)\ar[r]
&H^{\bullet}(Z,F) \ar[r]
& H^{\bullet}(Z_{1},F)\ar[r]
&\cdots\,.}
\end{aligned}\end{align}
Each line and column of (\ref{e.783}) is exact. Since $\phi_{R}^{*}$, $\phi_{1,R}^{*}$ and
$\phi_{2,R}^{*}$ are isomormorphisms keeping the $L^{2}-$metrics, the identity (\ref{e.83})
follows from \cite[Thm. A1.4, Thm. A1.1, (c)]{BL}. The proof is completed.
\end{proof}

Using (\ref{e.80}) and Lemma (\ref{l.3}), we establish the important identity (\ref{e.367}) in the process of adiabatic limit.

\section{The gluing formula by using adiabatic limit methods}\label{s.3}
In this section under the assumption (\ref{e.356}), we utilize the adiabatic method
 to prove the gluing formula (\ref{e.354}) for the analytic torsion forms of Bismut-Lott. We will divide the right hand side of (\ref{e.367}) into three parts and handle them separately.

This section is organized as follows.
In Section \ref{ss3.1}, we state our main result and divide the right hand side of (\ref{e.367}) into three parts: small time contribution, large time contribution and $T_{f}(A^{\mathscr{H}_{R}}, h_{L^{2}}^{\mathscr{H}_{R}})$. In Section \ref{ss3.2}, we treat the small time contribution.

\subsection{The gluing formula when $H(Y, F)=0$}\label{ss3.1}

In this section, we use the notation of Section \ref{ss1.8}. Recall that $\pi:M\rightarrow S$ is a fibration divided into two fibrations $M_{1}$
and $M_{2}$ by a hypersurface $X$. For $R>0$, the stretched fibrations $M_{R},M_{1,R},M_{2,R}$ are introduced in Section
\ref{ss1.9}.

Set
\begin{align}\begin{aligned}\label{e.189}
h^{(p)}:=\rk H^{p}(Z, F), \, h^{(p)}_{1}:=\rk H^{p}(Z_{1}, F),\,h^{(p)}_{2}:=\rk H^{p}(Z_{2}, Y, F).
\end{aligned}\end{align}
By our assumption (\ref{e.356}), we see that the long exact sequence $\mathscr{H}_{R}$ (see (\ref{e.434})) of flat vector bundles splits into the direct sum
of short exact sequences: for $0\leq p\leq m$
\begin{align}\begin{aligned}\label{e.86}
\mathscr{H}^{p}_{R}:\quad 0 \longrightarrow H^{p}(Z_{2, R}, Y, F_{R})\longrightarrow H^{p}(Z_{R}, F_{R})\longrightarrow H^{p}(Z_{1, R}, F_{R}) \longrightarrow 0.
\end{aligned}\end{align}
Hence we have in $Q^{S}/Q^{S, 0}$
\begin{align}\begin{aligned}\label{e.87}
T_{f}(A^{\mathscr{H}_{R}}, h_{L^{2}}^{\mathscr{H}_{R}})=\sum_{i=0}^{m}(-1)^{p}T_{f}(A^{\mathscr{H}^{p}_{R}}, h_{L^{2}}^{\mathscr{H}^{p}_{R}}).\end{aligned}\end{align}
By (\ref{e.189}) and (\ref{e.86}), we also get for $0\leq p\leq m$
\begin{align}\begin{aligned}\label{e.89}
h^{(p)}=h_{1}^{(p)}+h_{2}^{(p)}.
\end{aligned}\end{align}

By \cite[Def. 1.19]{Zhu13}, (\ref{e.367}) and (\ref{e.89}), we get in $Q^{S}/Q^{S, 0}$
\begin{align}\begin{aligned}\label{e.94}
&\mathscr{T}(T^{H}M, g^{TZ}, h^{F})-\mathscr{T}_{\rm abs}(T^{H}M_{1}, g^{TZ_{1}}, h^{F})\\
&\quad\quad\quad\quad\quad\quad\quad -\mathscr{T}_{\rm rel}(T^{H}M_{2}, g^{TZ_{2}}, h^{F})-T_{f}(A^{\mathscr{H}}, h_{L^{2}}^{\mathscr{H}})\\
=&-\int_{0}^{+\infty}\Big[f^{\wedge}(C'_{R,t}, h^{W_{R}})-f^{\wedge}(C'_{1, R, t}, h^{W_{1, R}})-f^{\wedge}(C'_{2, R, t}, h^{W_{2, R}})\Big]\frac{dt}{t}\\
&\quad\quad\quad\quad\quad\quad\quad -T_{f}(A^{\mathscr{H}_{R}}, h_{L^{2}}^{\mathscr{H}_{R}}).
\end{aligned}\end{align}

The rest part of Section \ref{s.3} will be contributed to prove Theorem \ref{t.13}.
Let $\varepsilon>0$ be a small positive constant, we divide the integral at the right side of (\ref{e.94}) into two parts:\\
(1) The small time contribution:
\begin{align}\begin{aligned}\label{e.92}
S(R):=-\int_{0}^{R^{2-\varepsilon}}\Big[f^{\wedge}(C'_{R,t}, h^{W_{R}})-f^{\wedge}(C'_{1, R, t}, h^{W_{1, R}})-f^{\wedge}(C'_{2, R, t}, h^{W_{2, R}})\Big]\frac{dt}{t}.
\end{aligned}\end{align}
(2) The large time contribution:
\begin{align}\begin{aligned}\label{e.93}
L(R):=-\int_{R^{2-\varepsilon}}^{\infty}\Big[f^{\wedge}(C'_{R,t}, h^{W_{R}})-f^{\wedge}(C'_{1, R, t}, h^{W_{1, R}})-f^{\wedge}(C'_{2, R, t}, h^{W_{2, R}})\Big]\frac{dt}{t}.
\end{aligned}\end{align}\index{$S(R),\,L(R)$}

\subsection{Adiabatic limit of the small time contribution $S(R)$}\label{ss3.2}

In this section, we will compute the limit of the small time contribution (\ref{e.92}).

\begin{thm}\label{t.6}Under the assumption (\ref{e.356}), we have
$\lim_{R\rightarrow \infty}S(R)=0$.
\end{thm}
The main methods used to prove this theorem are the Duhamel's principle (cf. \cite[$\S$2.7]{BGV92}) and the finite propagation speed property of the wave operator
 (cf. \cite[Appendix D.2]{MaMa07}).

For $i=1, 2$ and $t>0$, we define $C_{R, t}, D_{R, t}$ (resp. $C_{i, R, t}, D_{i, R, t}$) \index{$C_{R, t}, D_{R, t}, C_{i, R, t}, D_{i, R, t}$}
in the same way as (\ref{e.549}) for $(M_{R},F_{R})$ (resp. $(M_{i,R},F_{i,R})$), moreover we have
\begin{align}\begin{aligned}\label{e.122}
C_{R,t}^{2}=-D_{R,t}^{2}, \quad C_{i, R,t}^{2}=-D_{i, R,t}^{2}.
\end{aligned}\end{align}
Similarly, for $\pi_{\partial}: X \rightarrow S$ with the objects $T^{H}X,g^{TY}, \nabla^{F|_{X}}, h^{F|_{X}}$ induced by
 $T^{H}M$, $g^{TZ}$, $\nabla^{F}$, $h^{F}$, we define $\widetilde{C}_{t}$ and $\widetilde{D}_{t}$ for $t>0$ as in (\ref{e.549}). Then we have
\begin{align}\begin{aligned}\label{e.100}
\widetilde{C}_{t}^{2}=-\widetilde{D}_{t}^{2}.
\end{aligned}\end{align}
\begin{defn}\label{d.9}For $i=1, 2$, we let
\begin{align}\begin{aligned}\label{e.102}
\mathscr{F}_{R}:=4C_{R,1}^{2}=(D^{Z_{R}})^{2}+\mathscr{F}_{R}^{[+]}, \quad \mathscr{F}_{i, R}:=4C_{i, R,1}^{2}=(D^{Z_{i, R}})^{2}+\mathscr{F}_{i, R}^{[+]},
\end{aligned}\end{align}\index{$\mathscr{F}_{R}$, $\mathscr{F}_{R}^{[+]}$, $\mathscr{F}_{i, R}$, $\mathscr{F}_{i, R}^{[+]}$}
where $(D^{Z_{R}})^{2}$ (resp. $(D^{Z_{i, R}})^{2}$), the corresponding $0$-degree part in $\Lambda (T^{*}S)$, is a smooth family of Hodge Laplacians
along the fibers $Z_{R}$ (resp. $Z_{i, R}$) and $\mathscr{F}_{R}^{[+]}$ (resp. $\mathscr{F}_{i, R}^{[+]}$) represents the part of positive
degrees in $\Lambda (T^{*}S)$. Similarly, we set
\begin{align}\begin{aligned}\label{e.105}
\widetilde{\mathscr{F}}:= 4\widetilde{C}_{1}^{2}=(D^{Y})^{2}+\widetilde{\mathscr{F}}^{[+]},
\end{aligned}\end{align}
where the $0-$degree component $(D^{Y})^{2}$ is a smooth family of Hodge Laplacians along the fibers $Y$ and
 $\widetilde{\mathscr{F}}^{[+]}$ is its positive degree part in
$\Lambda (T^{*}S)$.
\end{defn}

For $t>0$, let $\psi_{t}\in \text{End}(\Omega(S))$ such that if $\alpha\in \Omega^{k}(S)$, then $\psi\alpha=t^{-k/2}\alpha$.
Following \cite[Prop. 3.17]{BGo} and (\ref{e.122}), we have
\begin{prop}\label{p.3} For $t>0$ and $i=1, 2$, the following identities hold
\begin{align}\begin{aligned}\label{e.101}
&-D^{2}_{R, t}=C^{2}_{R, t}=\frac{t}{4}\psi_{t}^{-1}\mathscr{F}_{R}\psi_{t},\quad-D^{2}_{i, R, t}=C^{2}_{i, R, t}=\frac{t}{4}\psi_{t}^{-1}\mathscr{F}_{i,R}\psi_{t}.
\end{aligned}\end{align}
\end{prop}

By using the product structures (\ref{e.7611}), (\ref{e.102})
 and (\ref{e.105}), we get the following lemma.

\begin{lemma}\label{l.52}
On the product neighborhood $X_{[-R, R]}$ (resp. $X_{[-R, 0]}$, $X_{[0, R]}$), we have
\begin{align}\begin{aligned}\label{e.654}
\mathscr{F}_{R}=\widetilde{\mathscr{F}}-\frac{\partial^{2}}{\partial x_{m}^{2}} \,\,\,(\text{resp.}\,
\mathscr{F}_{i,R}=\widetilde{\mathscr{F}}-\frac{\partial^{2}}{\partial x_{m}^{2}}).
\end{aligned}\end{align}
\end{lemma}

For $b\in S$, let $\mathscr{F}_{R,b}$ (resp. $\mathscr{F}_{1,R,b}$, $\mathscr{F}_{2,R,b}$ and $\widetilde{\mathscr{F}}_{b}$) be the restriction of $\mathscr{F}$
 (resp. $\mathscr{F}_{1,R}$, $\mathscr{F}_{2,R}$ and $\widetilde{\mathscr{F}}$) on $Z_{b}$ (resp. $Z_{1,b}$, $Z_{2,b}$ and $Y_{b}$). For $b\in S$,
 there exist a neighborhood $U\subset S$ of $b$ such that
\begin{align}\begin{aligned}\label{e.663}
\pi^{-1}(U)=U\times Z \quad&
(\text{resp. } \pi^{-1}(U)=U\times Z_{1},\\
&\pi^{-1}(U)=U\times Z_{2},\,\pi^{-1}(U)=U\times Y).
\end{aligned}\end{align}
For $b\in U$, $\mathscr{F}_{R,b}$ (resp. $\mathscr{F}_{1,R,b}$, $\mathscr{F}_{2,R,b}$ and $\widetilde{\mathscr{F}}_{b}$) is a smooth family of second order elliptic differential
operator on $Z_{R}$ (resp, $Z_{1,R}$, $Z_{2,R}$ and $Y$). Then we denote the heat kernel of $e^{-t\mathscr{F}_{R,b}}$ by $e^{-t\mathscr{F}_{R,b}}(x,x')$, which
is $C^{\infty}$ in $(t,x,x',b)\in (0,\infty)\times Z_{R}\times Z_{R}\times U$. And for $i=1,2$, we denote the heat kernel of $e^{-t\mathscr{F}_{i,R,b}}$
by $e^{-t\mathscr{F}_{i,R,b}}(x,x')$, which is $C^{\infty}$ in $(t,x,x',b)\in (0,\infty)\times Z_{i,R}\times Z_{i,R}\times U$ (cf. \cite[Lemma 1.14]{Zhu13}) with the absolute (resp. relative) boundary conditions for $i=1$
(resp. $i=2$). Similarly, let $e^{-t\widetilde{\mathscr{F}}_{b}}(y,y')$ denote the smooth heat kernel of $e^{-t\widetilde{\mathscr{F}}_{b}}$.

We have the off-diagonal estimates.
\begin{lemma}\label{l.47} There exists $c>0$, such that for any $l\in \mathbb{N}$ there exist $C_{l}>0$ such that for any $t>0$, $r_{0}>0$ and
 $x, x'\in Z_{R}$ with ${\rm d}(x, x')\geq r_{0}$
\begin{align}\begin{aligned}\label{5.44}
 \|e^{-t\mathscr{F}_{R}}(x, x') \|_{\mathscr{C}^{l}}\leq C_{l}e^{-c_{l}\frac{d^{2}(x,x')}{t}}.
\end{aligned}\end{align}
\end{lemma}
\begin{proof}
Since the principle symbol of $\mathscr{F}_{R}$ is equal to $|\xi|^{2}$, the proof is essentially the same as that of Lemma \ref{l.42}.
\end{proof}

\begin{defn}\label{d.11}
Let $\mathbb{R}_{\leq 0}=(-\infty,0]$ and $\mathbb{R}_{\geq 0}=[0,+\infty)$. For $t>0$, let $e^{-t\frac{\partial^{2}}{\partial x_{m}^{2}}}(u, v)\in \text{End}\big(\Lambda (T^{*}\mathbb{R})\big)$ be the heat kernel for
$(u,v)\in\mathbb{R}^{2}$,
$e_{{\rm abs}}^{-t\frac{\partial^{2}}{\partial x_{m}^{2}}}(u, v)$ (resp. $e_{{\rm rel}}^{-t\frac{\partial^{2}}{\partial x_{m}^{2}}}(u, v)$) be the heat kernel on $\mathbb{R}_{\leq 0}^{2}$ (resp. $\mathbb{R}_{\geq 0}^{2}$) with the absolute (resp. relative) boundary conditions
at $0$.
\end{defn}

For $t>0$, let $e_{\rm Dir}^{-t\frac{\partial^{2}}{\partial x^{2}_{m}}}(x_{m},x'_{m})$ (resp. $e_{\rm Neu}^{-t\frac{\partial^{2}}{\partial x^{2}_{m}}}(x_{m},x'_{m})$) be the smooth heat kernel on
 $\mathbb{R}_{\geq 0}$ with Dirichlet (resp. Neumann) boundary condition. Then we have for $(x_{m},x'_{m})\in \mathbb{R}_{\geq 0}^{2}$
\begin{align}\begin{aligned}\label{e.668}
&e_{\rm Dir/Neu}^{-t\frac{\partial^{2}}{\partial x^{2}_{m}}}(x_{m},x'_{m})=\frac{1}{\sqrt{4\pi t}}\Big(e^{-\frac{|x_{m}-x'_{m}|^{2}}{4t}}\mp e^{-\frac{|x_{m}+x'_{m}|^{2}}{4t}}\Big).
\end{aligned}\end{align}
Let $1, dx_{m}$ be a basis of $\Lambda(T^{*}\mathbb{R})$ and $1^{*}, (dx_{m})^{*}$ be its dual basis, then we can write explicitly
\begin{align}\begin{aligned}\label{e.139}
e^{-t\frac{\partial^{2}}{\partial x_{m}^{2}}}(x_{m}, x'_{m})=&\frac{1}{\sqrt{4\pi t}}e^{-\frac{|x_{m}-x'_{m}|^{2}}{4t}}dx_{m}\otimes(dx'_{m})^{*}
+\frac{1}{\sqrt{4\pi t}}e^{-\frac{|x_{m}-x'_{m}|^{2}}{4t}}1\otimes 1^{*},\\
e_{{\rm abs/rel}}^{-t\frac{\partial^{2}}{\partial x_{m}^{2}}}(x_{m}, x'_{m})=&e_{\rm Dir/Neu}^{-t\frac{\partial^{2}}{\partial x^{2}_{m}}}(x_{m},x'_{m})dx_{m}\otimes(dx'_{m})^{*}
+e_{\rm Neu/Dir}^{-t\frac{\partial^{2}}{\partial x^{2}_{m}}}(x_{m},x'_{m})\cdot1\otimes 1^{*}.
\end{aligned}\end{align}

\begin{defn}\label{d.10} We define $Y_{c}:=Y_{\mathbb{R}}$ (resp. $Y_{c, 1}:=Y _{(-\infty,0]}$, $Y_{c, 2}:=Y_{[0,+\infty)}$). We extend
the geometrical data $F|_{Y}, g^{TY}, h^{F|_{Y}}, \nabla^{F|_{Y}}$ trivially from $Y$ to $Y_{c}$ (resp. $Y_{c,1}$, $Y_{c,2}$), which will be denoted by
$$
F_{c}, g^{TY_{c}}, h^{F_{c}}, \nabla^{F_{c}} \qquad ({\rm resp}.\quad F_{c, i}, g^{TY_{c, i}}, h^{F_{c, i}}, \nabla^{F_{c, i}}).
$$
\end{defn}

Under the identification $\pi^{-1}(U)=Y\times U$, for $b\in U$, we construct a smooth family of second order operator
 $\mathscr{F}_{c,b}$ (resp. $\mathscr{F}_{c,1,b}$, $\mathscr{F}_{c,2,b}$ ) on $Y_{\mathbb{R}}\times U$
(resp. $Y_{(-\infty,0]}\times U$, $Y_{[0,+\infty)}\times U$)  such that for
$(y,x_{m},b)\in Y_{\mathbb{R}}\times U$
(resp. $Y_{(-\infty,0]}\times U$, $Y_{[0,+\infty)}\times U$)
\begin{align}\begin{aligned}\label{e.104}
\mathscr{F}_{c,b}=\widetilde{\mathscr{F}}_{b}-\frac{\partial^{2}}{\partial x_{m}^{2}}
\quad (\text{resp.}\, \mathscr{F}_{c, i,b}=\widetilde{\mathscr{F}}_{b}-\frac{\partial^{2}}{\partial x_{m}^{2}}).
\end{aligned}\end{align}
For $b\in U$, $\mathscr{F}_{c,b}$ acts on the bundle $\Lambda(T^{*}_{b}S)\otimes\big(\Lambda(T^{*}Y_{c})\otimes F\big)$
 over $Y_{\mathbb{R}}$, and for $i=1,2,$
$\mathscr{F}_{c,i,b}$ acts on the bundle $\Lambda(T^{*}_{b}S)\otimes\big(\Lambda(T^{*}Y_{c,i})\otimes F\big)$
 over $Y_{c,i}$.

\begin{defn}\label{d.36}
For $\big((y,u), (y',v),b \big)\in Y_{\mathbb{R}}^{2}\times U$, we set
\begin{align}\begin{aligned}\label{e.664}
\mE_{c,b}\big(t, (y, u), (y', v)\big)&=e^{-t\widetilde{\mathscr{F}}_{b}}(y, y')\otimes e^{-t\frac{\partial^{2}}{\partial x_{m}^{2}}}(u, v)\\
&\in \Lambda(T^{*}_{b}S)\otimes\big(\Lambda(T^{*}Y_{c})\otimes F\big)_{(y,u)}\otimes\big(\Lambda(T^{*}Y_{c})\otimes F\big)^{*}_{(y',v)}.
\end{aligned}\end{align}
Similarly, for $\big((y,u), (y',v),b \big)\in Y_{(-\infty,0]}^{2}\times U$ (resp. $Y_{[0,+\infty)}^{2}\times U$), we set
\begin{align}\begin{aligned}\label{e.665}
\mE_{c, 1,b}\big(t, (y, u), (y', v)\big)&=e^{-t\widetilde{\mathscr{F}}_{b}}(y, y')\otimes e_{{\rm abs}}^{-t\frac{\partial^{2}}
{\partial x_{m}^{2}}}(u, v)\\
\big(\text{resp.}\,\,\, \mE_{c, 2,b}\big(t, (y, u), (y', v)\big)&=e^{-t\widetilde{\mathscr{F}}_{b}}(y, y')\otimes e_{{\rm rel}}^{-t\frac{\partial^{2}}{\partial x_{m}^{2}}}(u, v)\,\big).
\end{aligned}\end{align}
\end{defn}

By Definition \ref{d.36} and (\ref{e.104}), for $b\in U$, $i=1,2$ we have
\begin{align}\begin{aligned}\label{e.666}
&\big(\partial_{t}+\mathscr{F}_{c,b}\big)\mE_{c,b}\big(t, (y, u), (y', v)\big)=0,\\
&\big(\partial_{t}+\mathscr{F}_{c,i,b}\big)\mE_{c,i,b}\big(t, (y, u), (y', v)\big)=0.
\end{aligned}\end{align}
and
\begin{align}\begin{aligned}\label{e.667}
\mE_{c,1,b}\, (\text{resp.}\,\mE_{c,2,b}) \text{ verifies the absolute (resp. relative) boundary}\\ \text{ conditions at}\quad Y\times \{0\}.
\end{aligned}\end{align}

For $d>a>0$, let $\rho(a, d):\mathbb{R}\rightarrow [0, 1]$ be an even cut-off function such that
 \begin{align}\begin{aligned}\label{e.108}
\rho(a, d)(v)=\left\{ \begin{array}{ll}
     0, & \hbox{$0\leq |v|\leq a $;}\\
     1, & \hbox{$ |v|\geq d.$}
    \end{array}
 \right.\end{aligned}\end{align}
Then we set
 \begin{align}\begin{aligned}\label{e.109}
\begin{array}{lllll}
 \phi_{1, R}(v)=&1-\rho(\frac{5}{7}, \frac{6}{7})(\frac{v}{R}), & & \psi_{1, R}(v)=&1-\rho(\frac{3}{7}, \frac{4}{7})(\frac{v}{R}), \\
 \phi_{2, R}(v)=&\rho(\frac{1}{7}, \frac{2}{7})(\frac{v}{R}), & & \psi_{2, R}(v)=&\rho(\frac{3}{7}, \frac{4}{7})(\frac{v}{R}).
\end{array}
\end{aligned}\end{align}
For $x\in Z_{R}$, we set
 \begin{align}\begin{aligned}\label{e.652}
\phi_{1, R}(x)=\left\{
                 \begin{array}{ll}
                   \phi_{1,R}(x_{m})& \hbox{ for $x=(y,x_{m})\in Y_{[-R,R]}$;} \\
                   0 & \hbox{ for  $x\notin Y_{[-R,R]}$.}
                 \end{array}
               \right.
\\
\psi_{1, R}(x)=\left\{
                 \begin{array}{ll}
                   \psi_{1,R}(x_{m})& \hbox{ for $x=(y,x_{m})\in Y_{[-R,R]}$;} \\
                   0 & \hbox{ for  $x\notin Y_{[-R,R]}$.}
                 \end{array}
               \right.
\end{aligned}\end{align}
and
 \begin{align}\begin{aligned}\label{e.653}
\phi_{2, R}(x)=\left\{
                 \begin{array}{ll}
                   \phi_{2,R}(x_{m})& \hbox{ for $x=(y,x_{m})\in Y_{[-R,R]}$;} \\
                   1 & \hbox{ for  $x\notin Y_{[-R,R]}$.}
                 \end{array}
               \right.
\\
\psi_{2, R}(x)=\left\{
                 \begin{array}{ll}
                   \psi_{2,R}(x_{m})& \hbox{ for $x=(y,x_{m})\in Y_{[-R,R]}$;} \\
                   1 & \hbox{ for  $x\notin Y_{[-R,R]}$.}
                 \end{array}
               \right.
\end{aligned}\end{align}
For $i=1, 2$, we set
 \begin{align}\begin{aligned}\label{e.120}
 \begin{array}{lllll}
 \phi_{1, R}^{(i)}=& \phi_{1, R}|_{Z_{i, R}},& & \psi_{1, R}^{(i)}=&\psi_{1, R}|_{Z_{i, R}}, \\
 \phi_{2, R}^{(i)}=& \phi_{2, R}|_{Z_{i, R}},& & \psi_{2, R}^{(i)}=&\psi_{2, R}|_{Z_{i, R}}.
\end{array}
\end{aligned}\end{align}

Now we define three parametrixes for $(x,x',b)\in (Z_{R})^{2}\times U$ (resp. $(Z_{i,R})^{2}\times U,\,i=1,2$)
 \begin{align}\begin{aligned}\label{e.110}
 &\mQ_{R,b}(t, x, x'):=\phi_{1, R}(x)\mE_{c,b}(t, x, x')\psi_{1, R}(x')+\phi_{2, R}(x)e^{-t\mathscr{F}_{R,b}}(x, x')\psi_{2, R}(x'), \\
 &\mQ_{i, R,b}(t, x, x'):=\phi_{1, R}^{(i)}(x)\mE_{c, i,b}(t, x, x')\psi^{(i)}_{1, R}(x')+\phi^{(i)}_{2, R}(x)e^{-t\mathscr{F}_{R,b}}(x, x')\psi^{(i)}_{2, R}(x').
\end{aligned}\end{align}\index{$\mQ_{R,b},\,\mQ_{i, R,b}$}
with the corresponding error terms:
\begin{align}\begin{aligned}\label{e.111}
 &\mC_{R,b}(t, x, x'):=\left(\partial_{t}+\mathscr{F}_{R,b,x}\right)\mQ_{R,b}(t, x, x'), \\
 &\mC_{i,R,b}(t, x, x'):=\left(\partial_{t}+\mathscr{F}_{i,R,b,x}\right)\mQ_{i,R,b}(t, x, x').
\end{aligned}\end{align}\index{$\mC_{R,b},\,\mC_{i,R,b}$}
Set
\begin{align}\begin{aligned}\label{e.112}
&(e^{-t\mathscr{F}_{R,b}}*\mC_{R,b})(t, x, x')=\int_{0}^{t}\int_{Z_{R}}e^{-(t-s)\mathscr{F}_{R,b}}(x, z)\mC_{R,b}(s, z, x')dzds, \\
&(e^{-t\mathscr{F}_{i,R,b}}*\mC_{i,R,b})(t, x, x')=\int_{0}^{t}\int_{Z_{i, R}}e^{-(t-s)\mathscr{F}_{i,R,b}}(x, z)\mC_{i,R,b}(s, z, x')dzds.
\end{aligned}\end{align}
By Duhamel's principle, we get the following lemma.
\begin{lemma}\label{l.51} For $R>0,t>0$ and $i=1,2$, we have
\begin{align}\begin{aligned}\label{e.113}
 &e^{-t\mathscr{F}_{R,b}}(x, x')=\mQ_{R,b}(t, x, x')-(e^{-t\mathscr{F}_{R,b}}*\mC_{R,b})(t, x, x'), \\
 &e^{-t\mathscr{F}_{i,R,b}}(x, x')=\mQ_{i,R,b}(t, x, x')-(e^{-t\mathscr{F}_{i,R,b}}*\mC_{i,R,b})(t, x, x').
\end{aligned}\end{align}
\end{lemma}
\begin{proof}
These equations follow from the uniqueness of heat kernel (cf. \cite[Thm. 2.48]{BGV92}).
\end{proof}

Using the heat equation and (\ref{e.654}), we can rewrite the error terms as:
\begin{align}\begin{aligned}\label{e.114}
&\mC_{R,b}(t, x, x')=\\
&-\frac{\partial^{2} \phi_{1, R}}{\partial x^{2}_{m}}(x)\mE_{c,b}(t, x, x')\psi_{1, R}(x')-2\frac{\partial\phi_{1,R} }{\partial x_{m}}(x)\frac{\partial \mE_{c,b}}{\partial x_{m}}(t, x, x')\psi_{1, R}(x')
\\
&-\frac{\partial^{2} \phi_{2, R}}{\partial x^{2}_{m}}(x)e^{-t\mathscr{F}_{R,b}}(x, x')\psi_{2, R}(x')-2\frac{\partial \phi_{2, R}}{\partial x_{m}}(x)\frac{\partial e^{-t\mathscr{F}_{R,b}}}{\partial x_{m}}( x,x')\psi_{2, R}(x'), \\
&\mC_{i,R,b}(t, x,x')=\\
&-\frac{\partial^{2} \phi^{(i)}_{1, R}}{\partial x^{2}_{m}}(x)\mE_{c, i,b}(t, x,x')\psi^{(i)}_{1, R}(x')-2\frac{\partial\phi^{(i)}_{1, R} }{\partial x_{m}}(x)\frac{\partial \mE_{c,i,b}}{\partial x_{m}}(t, x,x')\psi^{(i)}_{1, R}(x')
\\
&-\frac{\partial^{2} \phi^{(i)}_{2, R}}{\partial x^{2}_{m}}(x)e^{-t\mathscr{F}_{i,R,b}}(x,x')\psi^{(i)}_{2, R}(x')-2\frac{\partial \phi^{(i)}_{2, R}}{\partial x_{m}}(x)\frac{\partial e^{-t\mathscr{F}_{i,R,b}}}{\partial x_{m}}(x,x')\psi^{(i)}_{2, R}(x').
\end{aligned}\end{align}
From (\ref{e.652}), (\ref{e.653}) and (\ref{e.114}), we get
\begin{lemma}\label{l.5}
For $x'\in Z_{R}$ {\rm (resp. $Z_{1,R}, Z_{2,R}$)} fixed and $t>0$, we have
\begin{align}\label{e.115}
\supp_{x}\left\{\mC_{R,b}(t, x, x') \right\}\subset Y_{[-\frac{6R}{7}, -\frac{R}{7}]}\cup Y_{[\frac{R}{7}, \frac{6R}{7}]},
\end{align}
\begin{align}\begin{aligned}\label{e.116}
\supp_{x}\left\{\mC_{1,R,b}(t, x, x') \right\}\subset Y_{[-\frac{6R}{7}, -\frac{R}{7}]},\,\,\supp_{x}\left\{\mC_{2,R,b}(t, x, x') \right\}\subset Y_{[\frac{R}{7}, \frac{6R}{7}]}.
\end{aligned}\end{align}
If ${\rm d}(x, x')<\frac{R}{7}$, then
\begin{align}\label{e.117}
\mC_{R,b}(t, x, x'), \quad \mC_{1,R,b}(t, x, x')\quad {\rm and }\quad \mC_{2,R,b}(t, x, x') \quad \text{are all vanished.}
\end{align}
\end{lemma}

\begin{lemma}\label{l.6}
There exists $c>0$, such that for $i=1, 2$ and any $k\in \mathbb{N}$
there exists $C_{k}>0$ such that for all $t>0$, $R\geq 1$, $b\in U$, $(x, x')\in \supp \mC_{R,b}(t, \cdot, \cdot)$
 {\rm(resp.} $(x, x')\in \supp \mC_{i,R,b}(t, \cdot, \cdot)${\rm )},
\begin{align}\begin{aligned}\label{e.118}
 &\left|  e^{-t\mathscr{F}_{R,b}}(x, x')\right|_{\mathscr{C}^{k}}\leq C_{k}e^{-c\frac{R^{2}}{t}}, &\left|  e^{-t\mathscr{F}_{i,R,b}}(x, x')\right|_{\mathscr{C}^{k}}\leq C_{k}e^{-c\frac{R^{2}}{t}}, \\
 & \left|  \mC_{R,b}(t, x, x')\right|_{\mathscr{C}^{k}}\leq C_{k}e^{-c\frac{R^{2}}{t}}, &\quad \left|  \mC_{i,R,b}(t, x, x')\right|_{\mathscr{C}^{k}}\leq C_{k}e^{-c\frac{R^{2}}{t}}
 .
\end{aligned}\end{align}
Here $|\cdot|_{\mathscr{C}^{k}}$ denotes the $\mathscr{C}^{k}-$norms.
\end{lemma}
\begin{proof}

Let $f:\mathbb{R}\rightarrow [0,1]$ be a smooth even cut-off function such that
\begin{align}\begin{aligned}\label{e.450}
f(v):=\left\{
\begin{array}{ll}
 1 , & \text{ for }|v|\leq \frac{1}{14}, \\
 0 , & \text{ for }|v|\geq \frac{1}{7}.
\end{array}
\right.\end{aligned}\end{align}
For $a\in \mathbb{C}$, $u>0$, the functions $F_{u}(a),\,G_{u}(a)$ are introduced in (\ref{e.437}), and we have
\begin{align}\begin{aligned}\label{e.453}
e^{ -ta^{2}}=F_{2t/R^{2}}(\sqrt{2t}a)+G_{2t/R^{2}}(\sqrt{2t}a).
\end{aligned}\end{align}

The functions $F_{2t/R^{2}}(a), G_{2t/R^{2}}(a)$ are even holomorphic functions, therefore there exist holomorphic functions
 $\widetilde{F}_{2t/R^{2}}(a), \widetilde{G}_{2t/R^{2}}(a)$ such that
\begin{align}\begin{aligned}\label{e.454}
F_{2t/R^{2}}(a)=\widetilde{F}_{2t/R^{2}}(a^{2}), \quad G_{2t/R^{2}}(a)=\widetilde{G}_{2t/R^{2}}(a^{2}).
\end{aligned}\end{align}
From (\ref{e.453}) and (\ref{e.454}), we get for $t>0$
\begin{align}\begin{aligned}\label{e.455}
e^{-ta}=\widetilde{F}_{2t/R^{2}}(2ta)+\widetilde{G}_{2t/R^{2}}(2ta).
\end{aligned}\end{align}
The operator $\mathscr{F}_{R,b}=(D^{Z_{R,b}})^{2}+\mathscr{F}_{R,b}^{[+]}$ is a fiberwise second order elliptic operator whose principal symbol is $|\xi|^{2}$, hence we have
\begin{align}\begin{aligned}\label{e.456}
e^{-t\mathscr{F}_{R,b}}=\widetilde{F}_{2t/R^{2}}(2t\mathscr{F}_{R,b})+\widetilde{G}_{2t/R^{2}}(2t\mathscr{F}_{R,b}).
\end{aligned}\end{align}
Using the finite propagation speed of the wave operator (cf. \cite[Appendix D.2]{MaMa07}), for $x,x'\in Z_{R},\, {\rm d}(x, x')\geq \frac{R}{7}$ we have
\begin{align}\begin{aligned}\label{e.457}
\widetilde{F}_{2t/R^{2}}(2t\mathscr{F}_{R,b})(x, x')=0.
\end{aligned}\end{align}

Using the integration by parts (see (\ref{e.440})), there exists $c>0$, such that for any $m\in \mathbb{N}$ there
exists $C_{m}>0$ such that for any $R\geq 1$ and $t>0$, we have (cf. \cite[(1.6.16)]{MaMa07})
\begin{align}\begin{aligned}\label{e.458}
\sup_{a\in \mathbb{R}}\left|a^{m}\right|\cdot\big|\widetilde{G}_{2t/R^{2}}(2ta)\big|\leq C_{m}e^{-c\frac{R^{2}}{t}}.
\end{aligned}\end{align}
By the spectral theorem and (\ref{e.458}) we could get the following estimates:
\begin{align}\begin{aligned}\label{e.459}
\left\| \mathscr{F}_{R,b}^{m_{1}}\widetilde{G}_{2t/R^{2}}(2t\mathscr{F}_{R,b})\mathscr{F}_{R,b}^{m_{2}}\right\|\leq C_{m_{1},m_{2}}e^{-c\frac{R^{2}}{t}},
\end{aligned}\end{align}
where the constants $C_{m_{1},m_{2}}>0$ depend only on $m_{1}$ and $m_{2}$.
Apply a proof similar to the equations (\ref{e.741})-(\ref{e.443}), for $s\in \pi^{*}(\Lambda T^{*}S)\widehat{\otimes} \Omega(Z_{R}, F_{R}|_{Z_{R}})$ and $m_{1}+m_{2}\geq m+l$, we get by Sobolev inequality and elliptic estimates
\begin{align}\begin{aligned}\label{e.461}
\left| \widetilde{G}_{2t/R^{2}}(2t\mathscr{F}_{R,b})(x, x') \right|_{\mathscr{C}^{l}}\leq C_{l}e^{-c\frac{R^{2}}{t}}.
\end{aligned}\end{align}
By Lemma \ref{l.5}, (\ref{e.456}), (\ref{e.457}) and (\ref{e.461}), we get the estimates in the first line of (\ref{e.118}).

For the error terms in the second line of (\ref{e.118}), by (\ref{e.114}), we need to deal with the heat kernel $\mE_{c,b}(t, x, x')$ restricted on the cylinder part $Y_{[-R, R]}$.
By Definition \ref{d.36}, we get for $ (y, x_{m}), (y', x'_{m})\in Y_{[-R, R]}$
\begin{align}\begin{aligned}\label{e.463}
\frac{\partial \mE_{c,b}}{\partial x_{m}}\big(t, (y, x_{m}), (y', x'_{m})\big)=\frac{-(x_{m}-x'_{m})}{2t}e^{-t\widetilde{\mathscr{F}}_{b}}(y, y')\otimes e^{-t\frac{\partial^{2}}{\partial x_{m}^{2}}}(x_{m}, x'_{m}).
\end{aligned}\end{align}
We estimate the kernel of $e^{-t\widetilde{\mathscr{F}}_{b}}$ by two cases: For any $l\in \mathbb{N}$, there exists $C_{l}>0$ such that for any
$R\geq 1$, $b\in U$,
\begin{align}\begin{aligned}\label{e.464}
\left\{
\begin{array}{cc}
 \big| e^{-t\widetilde{\mathscr{F}}_{b}}(y, y') \big|_{\mathscr{C}^{l}}\leq C_{l}t^{-\frac{l}{2}} , & {\rm for }\,\, t\geq 1, \\
 \big| e^{-t\widetilde{\mathscr{F}}_{b}}(y, y') \big|_{\mathscr{C}^{l}}\leq C_{l}t^{-\frac{m+l-1}{2}} , & {\rm for }\,\, 0<t<1.
\end{array}
\right.
\end{aligned}\end{align}
If $\big((y, x_{m}),(y', x'_{m})\big)\in
\supp \mC_{R,b}(t, \cdot, \cdot)$, then by Lemma \ref{l.5} we have
\begin{align}\begin{aligned}\label{e.655}
\frac{R}{7} \leq|x_{m}-x'_{m}|\leq 2R.
\end{aligned}\end{align}
By (\ref{e.139}), (\ref{e.664}), (\ref{e.463}), (\ref{e.464}) and (\ref{e.655}), for $t\geq 1$ we have
\begin{align}\begin{aligned}\label{e.466}
&\left|  \mE_{c,b}(t, x, x') \right|_{\mathscr{C}^{l}}\leq C_{l}e^{-c\frac{R^{2}}{t}}t^{-\frac{l+1}{2}}
\leq C_{l}e^{-c\frac{R^{2}}{2t}},\\
&\left|\frac{\partial \mE_{c,b}}{\partial x_{m}}(t, x, x')\right|_{\mathscr{C}^{l}}\leq C_{l}Re^{-c\frac{R^{2}}{t}}t^{-\frac{l+3}{2}}
\leq C_{l}e^{-c\frac{R^{2}}{2t}}.
\end{aligned}\end{align}
And for $0<t<1$, we have
\begin{align}\begin{aligned}\label{e.468}
&\left|  \mE_{c,b}(t, x, x') \right|_{\mathscr{C}^{l}}\leq C_{l}e^{-c\frac{R^{2}}{t}}t^{-\frac{m+l}{2}}\leq C_{l}e^{-c\frac{R^{2}}{2t}},\\
&\left|\frac{\partial \mE_{c,b}}{\partial x_{m}}(t, x, x')\right|_{\mathscr{C}^{l}}\leq C_{l}Re^{-c\frac{R^{2}}{t}}t^{-\frac{m+l+2}{2}}
\leq C_{l}e^{-c\frac{R^{2}}{2t}}.
\end{aligned}\end{align}
Finally by (\ref{e.466}), (\ref{e.468}), we can get the estimates: for $x, x' \in Y_{[-R,R]} $
\begin{align}\begin{aligned}\label{e.470}
\left|  \mE_{c,b}(t, x, x') \right|_{\mathscr{C}^{l}}\leq C_{l}e^{-c\frac{R^{2}}{2t}}
 \quad \text{and} \quad
\left|\frac{\partial \mE_{c,b}}{\partial x_{m}}(t, x, x')\right|_{\mathscr{C}^{l}}\leq C_{l}e^{-c\frac{R^{2}}{2t}}.
\end{aligned}\end{align}
By the estimates in the first line of (\ref{e.118}) obtained above, (\ref{e.114}) and (\ref{e.470}), we get the estimates for the error terms
in the second line of (\ref{e.118}). We follow the same way to get estimates for $i=1,2$ in (\ref{e.118}).
The proof is completed.
\end{proof}

\begin{lemma}\label{l.8} For $t>0$ and $i=1, 2$, the following identities hold
\begin{align}\begin{aligned}\label{e.123}
\exp(D^{2}_{R,t})=\psi^{-1}_{t}e^{-\frac{t}{4}\mathscr{F}_{R}}\psi_{t}, \quad
\exp(D^{2}_{i, R,t})=\psi^{-1}_{t}e^{-\frac{t}{4}\mathscr{F}_{i,R}}\psi_{t}.
\end{aligned}\end{align}
\end{lemma}
\begin{proof}
By (\ref{e.101}) and uniqueness of the heat kernel, we get
\begin{align}\begin{aligned}\label{e.124}
\exp(D^{2}_{R, t})=\psi^{-1}_{t}\exp(tB^{2}_{R})\psi_{t}, \quad \exp(D^{2}_{i, R, t})=\psi^{-1}_{t}\exp(tB^{2}_{i, R})\psi_{t}.
\end{aligned}\end{align}
From (\ref{e.122}) and (\ref{e.102}), we have
\begin{align}\begin{aligned}\label{e.121}
&\exp(D^{2}_{R, t})=\psi^{-1}_{t}\exp(-tA^{2}_{R})\psi_{t}=\psi^{-1}_{t}e^{-\frac{t}{4}\mathscr{F}_{R}}\psi_{t}, \\
 &\exp(D^{2}_{i, R, t})=\psi^{-1}_{t}\exp(-tA^{2}_{i, R})\psi_{t}=\psi^{-1}_{t}e^{-\frac{t}{4}\mathscr{F}_{i, R}}\psi_{t}.
\end{aligned}\end{align}
The proof is completed.
\end{proof}
For $i=1, 2$, let $\chi_{i, R}$ be the characteristic function of $Z_{i, R}$ in $Z_{R}$ such that
\begin{align}\label{e.127}
\chi_{i, R}(x)=\left\{
\begin{array}{ccc}
 1 &, & x\in Z_{i, R}, \\
 0 &, &\text{otherwise.}
\end{array}
\right.\end{align}\index{$\chi_{i, R}(x)$}
We set
\begin{align}\begin{aligned}\label{e.128}
\text{I}_{R}(t):=\varphi\psi^{-1}_{t}&\int_{Z_{R}}\tr_{s}\Big[\frac{N}{2}\big[(1-\frac{1}{2}\mathscr{F}_{R, x})\mQ_{R}(\frac{t}{4}, x, x')\\
&-\sum_{i=1}^{2}\chi_{i, R}(x)(1-\frac{1}{2}\mathscr{F}_{i, R, x})\mQ_{i, R}(\frac{t}{4}, x, x')\big]_{x=x'}\Big]dv_{x},\\
\text{II}_{R}(t):=\varphi\psi^{-1}_{t}&\int_{Z_{R}}\tr_{s}\Big[\frac{N}{2}\big[(1-\frac{1}{2}\mathscr{F}_{R, x})(e^{-\frac{t}{4}\mathscr{F}}*\mC_{R})(\frac{t}{4}, x, x')\\
&-\sum_{i=1}^{2}\chi_{i, R}(x)(1-\frac{1}{2}\mathscr{F}_{i, R, x})(e^{-\frac{t}{4}\mathscr{F}_{i,R}}*\mC_{i,R})(\frac{t}{4}, x, x')\big]_{x=x'}\Big]dv_{
x}.
\end{aligned}\end{align}
Here we use some simple notations $W_{R}=\Omega^{\bullet}(Z_{R},F|_{Z_{R}})$ (resp. $W_{1,R}=\Omega^{\bullet}(Z_{1,R},F|_{Z_{1,R}})$, $W_{2,R}=\Omega^{\bullet}(Z_{2,R},Y,F|_{Z_{2,R}})$) to denote the infinite dimensional vector bundle over $S$. Then by (\ref{e.568}), (\ref{e.102}), (\ref{e.101}), (\ref{e.113}), (\ref{e.123}) and (\ref{e.127}), we get
\begin{align}\begin{aligned}\label{e.617}
&f^{\wedge}(C'_{R,t}, h^{W_{R}})-\sum_{i=1}^{2}f^{\wedge}(C'_{i, R, t}, h^{W_{i, R}})\\
&=\varphi\tr_{s}\left[\psi^{-1}_{t}\frac{N}{2}(1-\frac{1}{2}\mathscr{F}_{R})e^{-\frac{t}{4}\mathscr{F}_{R}}\psi_{t}\right]
-\sum_{i=1}^{2}\varphi\tr_{s}\left[\psi^{-1}_{t}\frac{N}{2}(1-\frac{1}{2}\mathscr{F}_{i, R})e^{-\frac{t}{4}\mathscr{F}_{i, R}}\psi_{t}\right]\\
&=\varphi\psi^{-1}_{t}\left\{\tr_{s}\left[\frac{N}{2}(1-\frac{1}{2}\mathscr{F}_{R})e^{-\frac{t}{4}\mathscr{F}_{R}}\right]
-\sum_{i=1}^{2}\varphi\tr_{s}\left[\frac{N}{2}(1-\frac{1}{2}\mathscr{F}_{i, R})e^{-\frac{t}{4}\mathscr{F}_{i, R}}\right] \right\}\\
&=\varphi\psi^{-1}_{t}\int_{Z_{R}}\tr_{s}\left[\Big(\frac{N}{2}(1-\frac{1}{2}\mathscr{F}_{R})e^{-\frac{t}{4}\mathscr{F}_{R}}\Big)(x, x)\right.\\
&\left.\quad\qquad\quad\qquad-\sum_{i=1}^{2}\chi_{i, R}(x)\Big(\frac{N}{2}(1-\frac{1}{2}\mathscr{F}_{i, R})e^{-\frac{t}{4}\mathscr{F}_{i, R}}\Big)(x, x)\right]dv_{x}\\
&=\text{I}_{R}(t)+\text{II}_{R}(t).
\end{aligned}\end{align}

Let
\begin{align}\begin{aligned}\label{e.749}
\mathcal{E}_{\rm dif}(t,x;b):=\big[\mE_{c,b}(\frac{t}{4}, x, x)-\sum_{i=1}^{2}\chi_{i, R}(x)\mE_{c,i,b}(\frac{t}{4}, x, x)\big],
\end{aligned}\end{align}
then we have the following lemma.
\begin{lemma}\label{l.9} For $t>0,\, R>0,\,b\in U$, we have
\begin{align}\begin{aligned}\label{e.131}
\text{I}_{R,b}(t)=\varphi\psi^{-1}_{t}(1+2\partial_{t})\int_{Z_{R}}\psi_{1, R}(x)\tr_{s}\Big[\frac{N}{2}\mathcal{E}_{\rm dif}(t,x;b)\Big]dv_{x}.
\end{aligned}\end{align}
\end{lemma}
\begin{proof}
Using the heat equations and by (\ref{e.110}), (\ref{e.114}), we get
\begin{align}\begin{aligned}\label{e.133}
(1-\frac{1}{2}&\mathscr{F}_{R,b,x})\mQ_{R,b}(\frac{t}{4}, x, x')\\
&=-\frac{1}{2}\mC_{R,b}(\frac{t}{4}, x, x')+\phi_{1, R}(x)(1-\frac{1}{2}\mathscr{F}_{R,b,x})\mE_{c,b}(\frac{t}{4}, x,x')\psi_{1, R}(x')\\
&\qquad\qquad\qquad\qquad+\phi_{2, R}(x)(1-\frac{1}{2}\mathscr{F}_{R,b,x})e^{-\frac{t}{4}\mathscr{F}_{R,b}}(x,x')\psi_{2, R}(x')\\
&=-\frac{1}{2}\mC_{R,b}(\frac{t}{4}, x, x')+\phi_{1, R}(x)(1+2\partial_{t})\mE_{c,b}(\frac{t}{4}, x, x')\psi_{1, R}(x')\\
&\qquad\qquad\qquad\qquad+\phi_{2, R}(x)(1+2\partial_{t})e^{-\frac{t}{4}\mathscr{F}_{R,b}}(x,x')\psi_{2, R}(x'),
\end{aligned}\end{align}
and similarly
\begin{align}\begin{aligned}\label{e.134}
(1-\frac{1}{2}&\mathscr{F}_{i,R,b,x})\mQ_{i,R,b}(\frac{t}{4}, x,x')\\
&=-\frac{1}{2}\mC_{i,R,b}(\frac{t}{4}, x, x')+\phi_{1, R}^{(i)}(x)(1+2\partial_{t})\mE_{c,i,b}(\frac{t}{4}, x,x')\psi^{(i)}_{1, R}(x')\\
&\qquad\qquad\qquad\qquad+\phi^{(i)}_{2, R}(x)(1+2\partial_{t})e^{-\frac{t}{4}\mathscr{F}_{R,b}}(x,x')\psi^{(i)}_{2, R}(x').\end{aligned}\end{align}
By (\ref{e.117}), we have
\begin{align}\begin{aligned}\label{e.135}
\mC_{R,b}(\frac{t}{4}, x, x)=0, \quad \mC_{i,R,b}(\frac{t}{4}, x, x)=0, \quad i=1, 2.
\end{aligned}\end{align}
Then it follows from (\ref{e.109}), (\ref{e.120}), (\ref{e.133}), (\ref{e.134}) and (\ref{e.135}) that
\begin{align}\begin{aligned}\label{e.136}
&\left\{(1-\frac{1}{2}\mathscr{F}_{R,b, x})\mQ_{R,b}(\frac{t}{4}, x,x')-\sum_{i=1}^{2}(1-\frac{1}{2}\mathscr{F}_{i, R,b, x})\mQ_{i, R,b}(\frac{t}{4}, x,x')\right\}_{x=x'}\\
&=(1+2\partial_{t})\phi_{1, R}(x)\mathcal{E}_{\rm dif}(t,x;b)\psi_{1, R}(x)
+(1+2\partial_{t})\phi_{2, R}(x)\Big[e^{-\frac{t}{4}\mathscr{F}_{R,b}}(x,x)\\
&\qquad\qquad\qquad-\sum_{i=1}^{2}\chi_{i, R}(x)e^{-\frac{t}{4}\mathscr{F}_{R,b}}(x,x)\Big]\psi_{2, R}(x)\\
&=(1+2\partial_{t})\psi_{1, R}(x)\mathcal{E}_{\rm dif}(t,x;b).
\end{aligned}\end{align}
Finally, (\ref{e.131}) follows from (\ref{e.128}) and (\ref{e.136}). The proof is completed.
\end{proof}

\begin{lemma}\label{l.10}The following identity hold, for any $t>0,\,R>0,\,b\in U,$
\begin{align}\begin{aligned}\label{e.137}
\int_{Z_{R}}\psi_{1, R}(x)\tr_{s}\Big[\frac{N}{2}\mathcal{E}_{\rm dif}(t,x;b)\Big]dv_{x}=0.
\end{aligned}\end{align}
\end{lemma}
\begin{proof}
We set $\eta_{-}, \eta_{+}$ the characteristic functions of $\mathbb{R}_{\leq 0}, \mathbb{R}_{\geq 0}$.
Let $x=(y, x_{m})\in Y_{[-R,R]}$ and
 \begin{align}\begin{aligned}\label{e.138}
e_{\rm dif}(t,x_{m}):=e^{-t\frac{\partial^{2}}{\partial x_{m}^{2}}}(x_{m}, x_{m})
&-\eta_{-}(x_{m})e_{{\rm abs}}^{-t\frac{\partial^{2}}{\partial x_{m}^{2}}}(x_{m}, x_{m})\\
&-\eta_{+}(x_{m})e_{{\rm rel}}^{-t\frac{\partial^{2}}{\partial x_{m}^{2}}}(x_{m}, x_{m}).
\end{aligned}\end{align}
 By (\ref{e.664}), (\ref{e.665}), (\ref{e.749}) and $\text{supp}(\psi_{1, R})\subset Y_{[-R, R]}$, we get
 \begin{align}\begin{aligned}\label{e.138}
 &\psi_{1, R}(y, x_{m})\mathcal{E}_{\rm dif}\big(t,(y,x_{m});b\big)
  =e^{-t\widetilde{\mathscr{F}}_{b}}(y, y)\otimes \psi_{1, R}(x_{m}) e_{\rm dif}(t,x_{m}).
\end{aligned}\end{align}
Let ${\rm sign}(x_{m})$ be the sign function defined as
\begin{align}\begin{aligned}\label{e.781}
{\rm sign}(x_{m})=\left\{\begin{array}{cc}
                     1 & x_{m}\geq 0, \\
                     -1& x_{m}<0.
                   \end{array}
\right.
\end{aligned}\end{align}
 By (\ref{e.668}) and (\ref{e.139}), we find
\begin{align}\begin{aligned}\label{e.656}
e_{\rm dif}(t,x_{m})
=
-{\rm sign}(x_{m})\frac{e^{-\frac{x_{m}^{2}}{t}}}{\sqrt{4\pi t}}dx_{m}\otimes(dx_{m})^{*}+{\rm sign}(x_{m})\frac{e^{-\frac{x_{m}^{2}}{t}}}{\sqrt{4\pi t}}1\otimes1^{*}.
\end{aligned}\end{align}
Recall that $\{1, dx_{m}\}$ form a basis of $\Lambda(T^{*}\mathbb{R})$, we use $\tr_{|_{dx_{m}}}$ (resp. $\tr_{|_{1}}$) to denote the point-wise trace restricted on the subbundle $\mathbb{R}\cdot dx_{m}$ (resp. $\mathbb{R}\cdot 1$).
As $\psi_{1, R}(x_{m})$ is an even function on $x_{m}$, by (\ref{e.656}) we get
\begin{align}\begin{aligned}\label{e.140}
\int_{-R}^{R}\psi_{1, R}(x_{m})\tr|_{dx_{m}}\big[e_{\rm dif}(t,x_{m})\big]dx_{m}=-\int_{-R}^{R}\psi_{1, R}(x_{m}){\rm sign}(x_{m})\frac{e^{-\frac{x_{m}^{2}}{t}}}{\sqrt{4\pi t}}dx_{m}=0,
\end{aligned}\end{align}
and similarly we get
\begin{align}\begin{aligned}\label{e.141}
\int_{-R}^{R}\psi_{1, R}(x_{m})\tr|_{1}\big[e_{\rm dif}(t,x_{m})\big]dx_{m}=0.
\end{aligned}\end{align}
By (\ref{e.652}), (\ref{e.138}), (\ref{e.140}) and (\ref{e.141}), we get
\begin{align}\begin{aligned}\label{e.142}
&\int_{Z_{R}}\psi_{1, R}(x)\tr_{s}\Big[\frac{N}{2}\mathcal{E}_{\rm dif}(t,x;b)\Big]dv_{Z_{R}}\\
=&\sum_{p=0}^{m}(-1)^{p}\frac{p}{2}\Big\{\tr|_{\Omega^{p}(Y,F)}\big[e^{-t\widetilde{\mathscr{F}}_{b}}\big]\cdot
\int_{-R}^{R}\psi_{1, R}(x_{m})\tr_{|_{1}}\big[e_{\rm dif}(t,x_{m})\big]dx_{m}\\
&+\tr|_{\Omega^{p-1}(Y,F)}\big[e^{-t\widetilde{\mathscr{F}}_{b}}\big]\cdot
\int_{-R}^{R}\psi_{1, R}(x_{m})\tr_{|_{dx_{m}}}\big[e_{\rm dif}(t,x_{m})\big]dx_{m}\Big\}
=0.
\end{aligned}\end{align}
From (\ref{e.142}) we have proved (\ref{e.137}). The proof is completed.
\end{proof}
By Lemma \ref{l.9} and Lemma \ref{l.10}, we have
\begin{lemma}\label{l.11} For all $t>0,\,R>0,\, b\in U$, we have
\begin{align}\begin{aligned}\label{e.143}
\text{I}_{R,b}(t)=0.\end{aligned}\end{align}
\end{lemma}
Now we start to treat $\text{II}_{R,b}(t)$ appearing in (\ref{e.128}).
Let
\begin{align}\begin{aligned}\label{e.750}
g_{b}(t,x;R)&:=\big[(1-\frac{1}{2}\mathscr{F}_{R,b, x})(e^{-\frac{t}{4}\mathscr{F}_{R,b}}*\mC_{R,b})(\frac{t}{4}, x, x')\big]_{x=x'},\\
g_{i,b}(t,x;R)&:=\big[(1-\frac{1}{2}\mathscr{F}_{i, R,b, x})(e^{-\frac{t}{4}\mathscr{F}_{i,R,b}}*\mC_{i,R,b})(\frac{t}{4}, x, x')\big]_{x=x'},
\end{aligned}\end{align}
then we have the following lemma.
\begin{lemma}\label{l.7}For $\varepsilon>0$ sufficiently small, $b\in U$, we have
\begin{align}\begin{aligned}\label{e.119}
\lim_{R\rightarrow \infty}&\int_{0}^{R^{2-\varepsilon}}\varphi\psi^{-1}_{t}\int_{Z_{R}}
\tr_{s}\Big[\frac{N}{2}g_{b}(t,x;R)\Big]dv_{Z_{R}}(x)\frac{dt}{t}=0,
\end{aligned}\end{align}
and for $i=1, 2$,
\begin{align}\begin{aligned}\label{e.144}
\lim_{R\rightarrow \infty}&\int_{0}^{R^{2-\varepsilon}}\varphi\psi^{-1}_{t}\int_{Z_{i, R}}\tr_{s}\Big[
\frac{N}{2}g_{i,b}(t,x;R)\Big]dv_{Z_{i,R}}(x)\frac{dt}{t}=0.
\end{aligned}\end{align}
\end{lemma}
\begin{proof}
By Lemmas \ref{l.5}, \ref{l.6} and (\ref{e.112}), we have that for any $R\geq 1$, $t>0$ and $b\in U$
\begin{align}\begin{aligned}\label{e.145}
 &\left|\int_{Z_{R}}\tr_{s}\Big[\frac{N}{2}g_{b}(t,x;R)\Big]dv_{Z_{R}}(x)
\right|_{\mathscr{C}^{0}(U)}
 \leq C_{m}\text{Vol}(Z_{R})\left|g_{b}(t,x;R)\right|_{\mathscr{C}^{0}(Z_{R})}\\
  \leq & C_{m}R\int_{0}^{t/4}ds\int_{Y\times[-\frac{6R}{7}, \frac{6R}{7}]}\big|e^{-(\frac{t}{4}-s)\mathscr{F}_{R,b}}(x,z)\big|_{\mathscr{C}^{2}}
\big|\mC_{R,b}(s, z, x)\big|_{\mathscr{C}^{0}}dv_{Z_{R}}(z) \\
  \leq &C_{m}R\int_{0}^{t/4}ds\int_{Y\times[-\frac{6R}{7}, \frac{6R}{7}]}\exp(-c\frac{4R^{2}}{t-4s})\exp(-c\frac{R^{2}}{s})dv_{Z_{R}}(z) \\
  \leq &C_{m}R^{2}\text{Vol}(Y)\int_{0}^{t/4}\exp(-c\frac{tR^{2}}{s(t-4s)})ds
\leq C_{m}R^{2}\int_{0}^{t/4}\exp(-c\frac{R^{2}}{s})ds\\
\leq &C_{m}R^{2}\int_{0}^{t}\exp(-c\frac{R^{2}}{s})ds
\leq C_{m}R^{2}t\exp(-c\frac{R^{2}}{t}).
\end{aligned}\end{align}
 Recall that $n=\dim(S)$, for any $\alpha\in \Omega(S)$ and $t>0$, we have
\begin{align}\begin{aligned}\label{e.758}
|\psi^{-1}_{t}\alpha|_{\mathscr{C}^{0}}\leq C(1+t^{-\frac{n}{2}})|\alpha|_{\mathscr{C}^{0}}.
\end{aligned}\end{align}
 Then for $R\geq 1$ sufficiently large we get
\begin{align}\begin{aligned}\label{e.146}
&\Big|\int_{0}^{R^{2-\varepsilon}}\frac{dt}{t}\varphi\psi^{-1}_{t}\int_{Z_{R}}\tr_{s}\Big[\frac{N}{2}g_{b}(t,x;R)\Big]
dv_{Z_{R}}\Big|_{\mathscr{C}^{0}(U)}\\
&\leq C_{m}\int_{0}^{R^{2-\varepsilon}}\frac{dt}{t}\|\psi^{-1}_{t}\|\Big|\int_{Z_{R}}\tr_{s}\Big[\frac{N}{2}
g_{b}(t,x;R)\Big]
dv_{Z_{R}}(x)\Big|_{\mathscr{C}^{0}(U)}\\
 &\leq C_{m}\int_{0}^{R^{2-\varepsilon}}R^{2}(1+t^{-\frac{n}{2}}) e^{-c\frac{R^{2}}{t}}dt
 \leq C_{m} R^{2} \int_{0}^{R^{2-\varepsilon}}\big(1+R^{n}t^{-n/2}\big)e^{-cR^{2}/t}dt\\
& \leq C_{m}R^{4}\int_{R^{\varepsilon}}^{\infty}\big(1+u^{n/2}\big)e^{-cu}\frac{du}{u^{2}}
\leq C_{m}R^{4-2\varepsilon}\int_{R^{\varepsilon}}^{\infty}e^{-\frac{cu}{2}}du
 \leq C_{m}R^{4-2\varepsilon}e^{-cR^{\varepsilon}/4}.
\end{aligned}\end{align}
By (\ref{e.146}), we get (\ref{e.119}). In the same way we prove (\ref{e.144}). The proof is completed.
\end{proof}

By (\ref{e.128}) and Lemma \ref{l.7}, we get
\begin{align}\begin{aligned}\label{e.147}
\lim_{R\rightarrow \infty}\int_{0}^{R^{2-\varepsilon}}\text{II}_{R}(t)\frac{dt}{t}=0.\end{aligned}\end{align}
By Lemma \ref{l.9}, (\ref{e.92}), (\ref{e.617}) and (\ref{e.147}), we have proved Theorem \ref{t.6}.

\section{Large time contributions in the adiabatic limit}\label{s.2}

In this section, we will first study the spectral properties of the Hodge-Laplacian $(D^{Z_{R, b}})^{2}$ (resp. $(D^{Z_{i, R, b}})^{2},\,i=1,2$) parameterized by $b\in U\subset S$
  on $Z_{R}$ (resp. $Z_{i, R}$) under the assumption (\ref{e.356}) when the length of the cylinder $R$ extends to infinity. As we will see, in this process the spectral of these operators is divided into two groups: one is the $0-$spectrum, the other one is the
 collection of spectrum uniformly bounded away from $0$ with respect to $ b\in U$ and $R$ sufficiently large.
Then by using the existence of spectral gap we show that the large time contribution is null.

In Section \ref{ss3.3}, we state the main result, Theorem \ref{t.10}, on the properties of the spectral of the fiberwise Hodge-Laplacians.
 We establish the $L^{2}(Y)-$norm estimates of the $\lambda-$eigensections of the Hodge-Laplacians
 on the cylinder part along
the fibers $Z$.
In Section \ref{ss3.4}, we show that the eigensections lying in $\mathscr{W}_{R}$ (resp. $\mathscr{W}_{1, R}$, $\mathscr{W}_{2, R}$) correspond to the
eigenvalues equal to $0$ or decaying exponentially with respect to $R\rightarrow \infty$.
In Section \ref{ss3.5}, we show that the eigensections orthogonal to $\mathscr{W}_{R}$ (resp. $\mathscr{W}_{1, R}$, $\mathscr{W}_{2, R}$) own
the eigenvalue uniformly bounded away from $0$ by a uniform positive constant when $R$ goes to infinity. Then by a result in \cite{APS}
 we show that in fact there does not exist any eigenvalues decaying exponentially.
In Section \ref{ss3.6}, we show that the large time contribution is null.

\subsection{Spectral gaps uniform with respect to $R\rightarrow \infty$}\label{ss3.3}
In this section, we adopt the notation of Section \ref{ss3.2}. For $b\in U\subset S$, $(D^{Z_{R,b}})^{2}$ (resp. $(D^{Z_{i, R,b}})^{2},\, i=1,2,$)
is the $0$-degree part in $\Lambda (T_{b}^{*}S)$
of $\mathscr{F}_{R,b}$ (resp. $\mathscr{F}_{i, R,b}$) (see Definition \ref{d.9}).
Under the local trivialization (\ref{e.663}), they are smooth families of generalized Laplacians
along the fibers $Z_{R}$ (resp. $Z_{i, R}$) parameterized by $b\in U$.
The $0-$degree component $(D^{Y_{b}})^{2}$ of $\widetilde{\mathscr{F}}_{b}$ is a smooth family of generalized Laplacians along the fibers $Y$
parameterized by $b\in U$.

 We will omit the sub-script $b\in U$ indicating on which fiber we work, and only mention it when it is necessary. First,
 we announce our main theorem of this section.

\begin{thm}\label{t.10} Under the assumption (\ref{e.356}), there exist $R_{0}>0$ and $c>0$, such that for any $ R>R_{0},\, b\in U$, $i=1, 2$,
the eigenvalue $\mu$ of the operator $(D^{Z_{R,b}})^{2}$ (resp. $(D^{Z_{i, R,b}})^{2}$) is either bounded away from $0$ with $\mu>c$,
or it is equal to $0$. In other words, we have
\begin{align}\begin{aligned}\label{e.148}
{\rm Spec}((D^{Z_{R,b}})^{2})\subset\{0\}\cup[c, +\infty) \quad\big(\text{resp.}\,\,{\rm Spec}((D^{Z_{i,R,b}})^{2})\subset\{0\}\cup[c, +\infty)\big).
\end{aligned}\end{align}
\end{thm}

Set
\begin{align}\begin{aligned}\label{e.669}
\delta=\inf_{b\in U}\min\{\mu>0|\,\,\mu \in \text{Spec}((D^{Y_{b}})^{2})\},
\end{aligned}\end{align}
from our assumption (\ref{e.356}) we can assume that $\delta>0$, since $b$ varies in the compact subset $U$.
 Let
$
\{\phi_{i}\}_{i=1}^{\infty}
$
be an orthonormal basis
 of $L^{2}(Y, \Lambda (T^{*}Y)\otimes F)$ consisting of smooth eigensections of $(D^{Y})^{2}$ such that
\begin{align}\begin{aligned}\label{e.271}
 (D^{Y})^{2}\phi_{i}=\mu_{i}\phi_{i}, \quad 0<\delta\leq \mu_{1}\leq \mu_{2}\leq \cdots \leq \mu_{i}\leq \cdots \rightarrow + \infty.
\end{aligned}\end{align}
Let $\psi$ be a smooth eigensection of $(D^{Z_{R}})^{2}$ such that
\begin{align}\begin{aligned}\label{e.272}
 (D^{Z_{R}})^{2}\psi=\lambda \psi, \quad 0\leq\lambda<\frac{3\delta}{4}, \quad \|\psi\|_{L^{2}(Z_{R})}=1.
\end{aligned}\end{align}
On the cylinder $Y_{[-R,R]}$ we expand $\psi$ in term of the basis (\ref{e.271})
\begin{align}\begin{aligned}\label{e.273}
\psi(y, x_{m})=\sum_{k=1}^{\infty}f_{k}(x_{m})\phi_{k}(y)+\sum_{k=1}^{\infty}g_{k}(x_{m})dx_{m}\wedge\phi_{k}(y), \quad (y, x_{m})\in Y_{[-R,R]}.
\end{aligned}\end{align}

Using the product structures, on $Y_{[-R,R]}$ we have
\begin{align}\begin{aligned}\label{e.274}
(D^{Z_{R}})^{2}=-\frac{\partial^{2}}{\partial x^{2}_{m}}+(D^{Y})^{2}.\end{aligned}\end{align}
By (\ref{e.272}), (\ref{e.273}) and (\ref{e.274}), we get
\begin{align}\begin{aligned}\label{e.275}
     \frac{\partial^{2}f_{k}}{\partial x^{2}_{m}}  =& (\mu_{k}-\lambda)f_{k},\quad
     \frac{\partial^{2}g_{k}}{\partial x^{2}_{m}} =&  (\mu_{k}-\lambda)g_{k}.
\end{aligned}\end{align}
By (\ref{e.271}), (\ref{e.272}), (\ref{e.275}), for any $k\in \mathbb{N}$ and $\lambda<u_{k}$ we find
\begin{align}\begin{aligned}\label{e.276}
\left\{\begin{array}{ll}
     f_{k}(y, x_{m})  = &a_{k}e^{-\sqrt{\mu_{k}-\lambda}\,x_{m}}+b_{k}e^{\sqrt{\mu_{k}-\lambda}\,x_{m}}, \\
     g_{k}(y, x_{m})  = &c_{k}e^{-\sqrt{\mu_{k}-\lambda}\,x_{m}}+d_{k}e^{\sqrt{\mu_{k}-\lambda}\,x_{m}},
\end{array}
\right.
\end{aligned}\end{align}
where $a_{k}, b_{k}, c_{k}, d_{k}$ are some constants. Let
\begin{align}\begin{aligned}\label{e.278}
\psi^{+}(y, x_{m})&:=\sum_{k}e^{-\sqrt{\mu_{k}-\lambda}\,x_{m}}\big(a_{k}\phi_{k}(y)+c_{k}dx_{m}\wedge\phi_{k}(y)\big),\\
\psi^{-}(y, x_{m})&:=\sum_{k}e^{\sqrt{\mu_{k}-\lambda}\,x_{m}}\big(b_{k}\phi_{k}(y)+d_{k}dx_{m}\wedge\phi_{k}(y)\big).
\end{aligned}\end{align}
Then substitute (\ref{e.276}) into (\ref{e.273}), it produces
\begin{align}\begin{aligned}\label{e.277}
&\psi(y, x_{m})=\psi^{+}(y, x_{m})+\psi^{-}(y, x_{m}).
\end{aligned}\end{align}

\begin{lemma}\label{l.28}
 There exist constants $C>0$ and $R_{0}>0$ such that for any $R>R_{0}$, $-\frac{3R}{4}\leq x_{m}\leq \frac{3R}{4}$
and $\psi$ a smooth eigensection of $(D^{Z_{R}})^{2}$ satisfying (\ref{e.272}), we have
\begin{align}\begin{aligned}\label{e.290}
\left\|\psi  \right\|_{L^{2}(Y\times\{x_{m}\})}\leq C e^{-\frac{\sqrt{\delta}}{16}R}.
\end{aligned}\end{align}
As its consequence, we also have $\left\|D^{Z_{R}}\psi  \right\|_{L^{2}(Y\times\{x_{m}\})}\leq C \lambda^{\frac{1}{2}}e^{-\frac{\sqrt{\delta}}{16}R}$.
\end{lemma}
\begin{proof}
We denote
\begin{align}\begin{aligned}\label{e.280}
T_{\text{I}}&:=\int_{-R}^{+R}\left\|\psi^{+}\right\|_{L^{2}(Y \times \{x_{m}\})}^{2}+\left\|\psi^{-}\right\|_{L^{2}(Y \times \{x_{m}\})}^{2}dx_{m},\\
T_{\text{II}}&:=2\int_{-R}^{+R}\Re\left\langle\psi^{+}, \psi^{-} \right\rangle_{L^{2}(Y \times \{x_{m}\})}dx_{m}.
\end{aligned}\end{align}
In (\ref{e.272}), we have supposed that $\|\psi\|_{L^{2}(Z_{R})}=1$, then by (\ref{e.277}) and (\ref{e.280})
\begin{align}\begin{aligned}\label{e.279}
1&=\|\psi\|^{2}_{L^{2}(Z_{R})} \geq
\int_{-R}^{+R}\left\|\psi^{+}+\psi^{-}\right\|_{L^{2}(Y \times \{x_{m}\})}^{2}dx_{m}=T_{\text{I}}+T_{\text{II}}.
\end{aligned}\end{align}

\textbf{For the first term $T_{\text{I}}$}. We use a simpler notation
\begin{align}\begin{aligned}\label{e.283}
|\sigma_{k}|:=\sqrt{|a_{k}|^{2}+|c_{k}|^{2}+|b_{k}|^{2}+|d_{k}|^{2}}.
\end{aligned}\end{align}\index{$\mid\sigma_{k}\mid$}
By (\ref{e.283}), (\ref{e.278}) and (\ref{e.280}), we get
\begin{align}\begin{aligned}\label{e.281}
T_{\text{I}}&\geq \int_{-R}^{+R}\sum_{k}e^{-2x_{m}\sqrt{\mu_{k}-\lambda}}\left(|a_{k}|^{2}+|c_{k}|^{2}\right)dx_{m}\\
&\qquad\qquad+\int_{-R}^{+R}\sum_{k}e^{2x_{m}\sqrt{\mu_{k}-\lambda}}\left(|b_{k}|^{2}+|d_{k}|^{2}\right)dx_{m} \\
&=\sum_{k}\frac{e^{2R\sqrt{\mu_{k}-\lambda}}-e^{-2R\sqrt{\mu_{k}-\lambda}}}{2\sqrt{\mu_{k}-\lambda}}|\sigma_{k}|^{2}.
\end{aligned}\end{align}
For $x >0$, there exist a constant $C_{0}>0$ such that
\begin{align}\begin{aligned}\label{e.282}
\frac{e^{x}-e^{-x}}{x}\geq C_{0}e^{\frac{7}{8}x}.
\end{aligned}\end{align}
By (\ref{e.272}), (\ref{e.283}), (\ref{e.281}) and (\ref{e.282}), we find
\begin{align}\begin{aligned}\label{e.284}
T_{\text{I}} \geq& \sum_{k}C_{0}R e^{\frac{7R}{4}\sqrt{\mu_{k}-\lambda}}|\sigma_{k}|^{2}
\geq &C_{0}R e^{\frac{R}{8}\sqrt{\delta}}\sum_{k}e^{\frac{3R}{2}\sqrt{\mu_{k}-\lambda}}
|\sigma_{k}|^{2}.
\end{aligned}\end{align}

\textbf{For the second term $T_{\text{II}}$}. By (\ref{e.278}) and (\ref{e.280}), we get
\begin{align}\begin{aligned}\label{e.285}
T_{\text{II}}&=\int_{-R}^{+R}\left\{\sum_{k}(a_{k}\bar{b}_{k}+c_{k}\bar{d}_{k}+\bar{a}_{k}b_{k}+\bar{c}_{k}d_{k})\right\}dx_{m}\\
&=2R\sum_{k}(a_{k}\bar{b}_{k}+c_{k}\bar{d}_{k}+\bar{a}_{k}b_{k}+\bar{c}_{k}d_{k}),
\end{aligned}\end{align}
hence we have by (\ref{e.283}) and (\ref{e.285})
\begin{align}\begin{aligned}\label{e.286}
|T_{\text{II}}|\leq 2R\sum_{k}|\sigma_{k}|^{2}\leq 2R\sum_{k}e^{\frac{3R}{2}\sqrt{\mu_{k}-\lambda}}
|\sigma_{k}|^{2}.
\end{aligned}\end{align}
By (\ref{e.280}), (\ref{e.279}), (\ref{e.284}) and (\ref{e.286}), we find for $R\geq 1$ large enough
\begin{align}\begin{aligned}\label{e.287}
1&=\|\psi\|^{2}\geq T_{\text{I}}+T_{\text{II}}\geq C_{0}R e^{\frac{R}{8}\sqrt{\delta}}\sum_{k}e^{\frac{3R}{2}\sqrt{\mu_{k}-\lambda}}
|\sigma_{k}|^{2}+T_{\text{II}}\\
&\geq \frac{1}{2}C_{0}R e^{\frac{R}{8}\sqrt{\delta}}\sum_{k}e^{\frac{3R}{2}\sqrt{\mu_{k}-\lambda}}
|\sigma_{k}|^{2}+\Big\{\frac{1}{2}C_{0}R e^{\frac{R}{8}\sqrt{\delta}}-2R\Big\}\sum_{k}e^{\frac{3R}{2}\sqrt{\mu_{k}-\lambda}}
|\sigma_{k}|^{2}\\
&\geq \frac{1}{2}C_{0}R e^{\frac{R}{8}\sqrt{\delta}}\sum_{k}e^{\frac{3R}{2}\sqrt{\mu_{k}-\lambda}}
|\sigma_{k}|^{2}.\end{aligned}\end{align}
In other words, for $R\geq1$ large enough we have
\begin{align}\begin{aligned}\label{e.288}
\sum_{k}e^{\frac{3R}{2}\sqrt{\mu_{k}-\lambda}}
|\sigma_{k}|^{2}\leq \frac{2}{C_{0}R}e^{-\frac{R}{8}\sqrt{\delta}}\leq 2C^{-1}_{0}e^{-\frac{R}{8}\sqrt{\delta}}.
\end{aligned}\end{align}
By (\ref{e.283}), (\ref{e.277}) and (\ref{e.288}), we obtain the $L^{2}-$norm estimate of $\psi$ in the $Y$-direction for any $-\frac{3R}{4}\leq x_{m}\leq \frac{3R}{4}$
\begin{align}\begin{aligned}\label{e.289}
&\left\|\psi\right\|_{L^{2}(Y\times\{x_{m}\} )}^{2}
\leq 2\left(\big\|\psi^{+}\big\|_{L^{2}(Y\times\{x_{m}\} )}^{2}+\big\|\psi^{-}\big\|_{L^{2}(Y\times\{x_{m}\} )}^{2}\right)\\
&=\sum_{k}2e^{-2x_{m}\sqrt{\mu_{k}-\lambda}}\left(|a_{k}|^{2}+|c_{k}|^{2}\right)+\sum_{k}2e^{2x_{m}\sqrt{\mu_{k}-\lambda}}\left(|b_{k}|^{2}+|d_{k}|^{2}\right)\\
&\leq 2\sum_{k}e^{\frac{3R}{2}\sqrt{\mu_{k}-\lambda}}|\sigma_{k}|^{2}
\leq 4C^{-1}_{0}e^{-\frac{\sqrt{\delta}}{8}R}.
\end{aligned}\end{align}
Now we have finished the proof.
\end{proof}

Next we try to get similar lemmas for the eigensections of $(D^{Z_{1, R}})^{2}$ and $(D^{Z_{2, R}})^{2}$. Let
\begin{align}\begin{aligned}\label{e.291}
\left\{ \begin{array}{c}
     (D^{Z_{1, R}})^{2}\psi_{1}=\lambda\cdot \psi_{1},\, 0\leq\lambda<\frac{3\delta}{4}, \quad \|\psi_{1}\|_{L^{2}(Z_{1, R})}=1, \\
     \left.
\left(i(e_{\mathbf{n}})\psi_{1}\right) \right |_{Y}=\left.\left(i(e_{\mathbf{n}})d^{Z_{1, R}}\psi_{1} \right)\right |_{Y}=0,
 \end{array}
 \right.\end{aligned}\end{align}
where $e_{\mathbf{n}}$ denotes the inward-pointing normal vector along the boundary. On $Y_{[-R,0]}$ we expand $\psi_{1}$ in term of basis (\ref{e.271})
\begin{align}\begin{aligned}\label{e.292}
\psi_{1}(y, x_{m})=\sum_{k=1}^{\infty}f_{k}(x_{m})\phi_{k}(y)+\sum_{k=1}^{\infty}g_{k}(x_{m})dx_{m}\wedge\phi_{k}(y),
\quad (y, x_{m})\in Y_{(-\infty,0]}.
\end{aligned}\end{align}
Similar to (\ref{e.274})-(\ref{e.276}), and using (\ref{e.291}), (\ref{e.292}), for absolute boundary conditions, we find
$\frac{\partial f_{k}}{\partial x_{m}}(0) = 0,\, g_{k}(0) = 0$, then it follows that
\begin{align}\begin{aligned}\label{e.293}
     a_{k}=b_{k},\quad
     c_{k}=-d_{k}.
\end{aligned}\end{align}
We set
\begin{align}\begin{aligned}\label{e.295}
\psi^{+}_{1}(y, x_{m})&:=\sum_{k}e^{-\sqrt{\mu_{k}-\lambda}\,x_{m}}\big(a_{k}\phi_{k}(y)+c_{k}dx_{m}\wedge\phi_{k}(y)\big), \\
\psi^{-}_{1}(y, x_{m})&:=\sum_{k}e^{\sqrt{\mu_{k}-\lambda}\,x_{m}}\big(a_{k}\phi_{k}(y)-c_{k}dx_{m}\wedge\phi_{k}(y)\big),
\end{aligned}\end{align}
then substitute (\ref{e.293}) into (\ref{e.292}), we have on $Y_{[-R,0]}$
\begin{align}\begin{aligned}\label{e.294}
\psi_{1}(y, x_{m})=\psi^{+}_{1}(y, x_{m})+\psi^{-}_{1}(y, x_{m}).
\end{aligned}\end{align}

\begin{lemma}\label{l.29}
There exist constants $C>0$ and $R_{0}>0$ such that for any $R>R_{0}$, $-\frac{3R}{4}\leq x_{m}\leq 0$
and $\psi_{1}$ a smooth eigensection of $(D^{Z_{1, R}})^{2}$ satisfying (\ref{e.291}), we have
\begin{align}\begin{aligned}\label{e.305}
\left\|\psi_{1}\right\|_{L^{2}(Y\times \{x_{m}\})}\leq C e^{-\frac{\sqrt{\delta}}{16}R}.
\end{aligned}\end{align}
As its consequence, we also have $\left\|D^{Z_{1,R}}\psi_{1}\right\|_{L^{2}(Y\times\{x_{m}\})}\leq C \lambda^{\frac{1}{2}}e^{-\frac{\sqrt{\delta}}{16}R}$.
\end{lemma}
\begin{proof}
We denote
\begin{align}\begin{aligned}\label{e.297}
S_{\text{I}}&:=\int_{-R}^{0}\left\|\psi^{+}_{1}\right\|_{L^{2}(Y\times \{x_{m}\})}^{2}+\left\|\psi^{-}_{1}\right\|_{L^{2}(Y\times \{x_{m}\})}^{2}dx_{m}, \\
S_{\text{II}}&:=2\int_{-R}^{0}\Re\left\langle\psi^{+}_{1}, \psi^{-}_{1}\right\rangle_{L^{2}(Y\times \{x_{m}\})} dx_{m}.\end{aligned}\end{align}
By (\ref{e.291}), (\ref{e.292}), (\ref{e.294}) and (\ref{e.297}), we find
\begin{align}\begin{aligned}\label{e.296}
1&=\|\psi_{1}\|_{L^{2}(Z_{1, R})}^{2}\geq \int_{-R}^{0}\left\|\psi^{+}_{1}+\psi^{-}_{1}\right\|_{L^{2}(Y\times \{x_{m}\})}^{2}dx_{m}=S_{\text{I}}+S_{\text{II}}.
\end{aligned}\end{align}

\textbf{For the first term $S_{\text{I}}$}. By (\ref{e.295}) and (\ref{e.297}), we get
\begin{align}\begin{aligned}\label{e.298}
S_{\text{I}}&\geq \int_{-R}^{0}\sum_{k}e^{-2x_{m}\sqrt{\mu_{k}-\lambda}}\left(|a_{k}|^{2}+|c_{k}|^{2}\right)dx_{m}\\
&\qquad\qquad+\int_{-R}^{0}\sum_{k}e^{2x_{m}\sqrt{\mu_{k}-\lambda}}\left(|a_{k}|^{2}+|c_{k}|^{2}\right)dx_{m} \\
&=\sum_{k}\frac{e^{2R\sqrt{\mu_{k}-\lambda}}-e^{-2R\sqrt{\mu_{k}-\lambda}}}{2\sqrt{\mu_{k}-\lambda}}(|a_{k}|^{2}+|c_{k}|^{2}).
\end{aligned}\end{align}
By(\ref{e.282}), (\ref{e.291}) and (\ref{e.298}), we get
\begin{align}\begin{aligned}\label{e.299}
S_{\text{I}}&\geq C_{0} R \sum_{k}e^{\frac{7R}{4}\sqrt{\mu_{k}-\lambda}}
(|a_{k}|^{2}+|c_{k}|^{2})
\geq C_{0} R e^{\frac{R}{8}\sqrt{\mu_{k}}} \sum_{k}e^{\frac{3R}{2}\sqrt{\mu_{k}-\lambda}}
(|a_{k}|^{2}+|c_{k}|^{2}).
\end{aligned}\end{align}
\textbf{For the second term $S_{\text{II}}$}. By (\ref{e.295}) and (\ref{e.297}), we get
\begin{align}\begin{aligned}\label{e.300}
S_{\text{II}}&:=2\int_{-R}^{0}\Re\left\langle\psi^{+}_{1}, \psi^{-}_{1}\right\rangle_{L^{2}(Y\times \{x_{m}\})} dx_{m}\\
&=\int_{-R}^{0}\left\{\sum_{k}a_{k}\bar{a}_{k}-c_{k}\bar{c}_{k}+\bar{a}_{k}a_{k}-\bar{c}_{k}c_{k}\right\}
dx_{m}=2R\sum_{k}(|a_{k}|^{2}-|c_{k}|^{2}).
\end{aligned}\end{align}

 When $R$ is sufficiently large, we have $$
\Big|2R\sum_{k}(|a_{k}|^{2}-|c_{k}|^{2})\Big|
\leq \frac{C_{0}}{2} R \,e^{\frac{R}{8}\sqrt{\mu_{k}}} \sum_{k}e^{\frac{3R}{2}\sqrt{\mu_{k}-\lambda}}
(|a_{k}|^{2}
+|c_{k}|^{2}),$$ hence by (\ref{e.296}), (\ref{e.299}) and (\ref{e.300}), we get for $R\geq 1$ large enough
\begin{align}\begin{aligned}\label{e.303}
&\frac{C_{0}R}{2}e^{\frac{R}{8}\sqrt{\delta}}\sum_{k}e^{\frac{3R}{2}\sqrt{\mu_{k}-\lambda}}
(|a_{k}|^{2}+|c_{k}|^{2})\\
&\leq  C_{0} R e^{\frac{R}{8}\sqrt{\mu_{k}}} \sum_{k}e^{\frac{3R}{2}\sqrt{\mu_{k}-\lambda}}
(|a_{k}|^{2}
+|c_{k}|^{2})
+2R\sum_{k}(|a_{k}|^{2}-|c_{k}|^{2})\\
&\leq S_{\text{I}}+S_{\text{II}}\leq 1.
\end{aligned}\end{align}
By (\ref{e.295}), (\ref{e.294}) and (\ref{e.303}), we obtain the $L^{2}-$norm estimate of $\psi_{1}$ in $Y$-direction for any $-\frac{3R}{4}\leq x_{m}\leq 0$ and $R\geq 1$ large enough,
\begin{align}\begin{aligned}\label{e.304}
&\left\|\psi_{1}\right\|^{2}_{L^{2}(Y\times \{x_{m}\})}
\leq 2\left(\big\|\psi_{1}^{+}\big\|^{2}_{L^{2}(Y\times \{x_{m}\})}+\big\|\psi_{1}^{-}\big\|^{2}_{L^{2}(Y\times \{x_{m}\})}\right)\\
&=\sum_{k}2e^{-2x_{m}\sqrt{\mu_{k}-\lambda}}\left(|a_{k}|^{2}+|c_{k}|^{2}\right)+\sum_{k}2e^{2x_{m}\sqrt{\mu_{k}-\lambda}}\left(|a_{k}|^{2}+|c_{k}|^{2}\right)\\
&\leq 4\sum_{k}e^{\frac{3R}{2}\sqrt{\mu_{k}-\lambda}}\left(|a_{k}|^{2}+|c_{k}|^{2}\right)
\leq 8C^{-1}_{0}e^{-\frac{\sqrt{\delta}}{8}R}.
\end{aligned}\end{align}
Now we have finished the proof of the lemma.
\end{proof}

 We set
\begin{align}\begin{aligned}\label{e.306}
\left\{ \begin{array}{c}
     (D^{Z_{2, R}})^{2}\psi_{2}=\lambda\cdot \psi_{2}, \quad 0\leq\lambda<\frac{3\delta}{4}, \quad \|\psi_{2}\|_{L^{2}(Z_{2, R})}=1, \\
     \left.
\left(e^{\mathbf{n}}\wedge\psi_{2}\right) \right |_{Y}=\left.
\big(e^{\mathbf{n}}\wedge\big(d^{Z_{2, R}}\big)^{*}\psi_{2} \big)\right |_{Y}=0,    \end{array}
 \right.\end{aligned}\end{align}
where we use $e^{\mathbf{n}}$ to denote the dual vector of the inward-pointing normal vector along the boundary.
Similar to Lemma \ref{l.29}, we have the following lemma.

\begin{lemma}\label{l.30}
There exist $C>0$ and $R>R_{0}$ such that for any $R>R_{0}$, $0\leq x_{m}\leq \frac{3R}{4}$ and
 $\psi_{2}$ a smooth eigensection of $(D^{Z_{2, R}})^{2}$ satisfying (\ref{e.306}), we have
\begin{align}\begin{aligned}\label{e.310}
\left\|\psi_{2}\right\|_{L^{2}( Y\times\{x_{m}\})}\leq Ce^{-\frac{\sqrt{\delta}}{16}R}.
\end{aligned}\end{align}
As its consequence, we also have $\left\|D^{Z_{2,R}}\psi_{2}  \right\|_{L^{2}(Y\times\{x_{m}\})}\leq C \lambda^{\frac{1}{2}}e^{-\frac{\sqrt{\delta}}{16}R}$.
\end{lemma}
\begin{proof}
The proof is essentially the same as that of Lemma \ref{l.29}.
\end{proof}

 An estimate similar to Lemmas \ref{l.28}, \ref{l.29} and \ref{l.30} has been done by Wojciechowski \cite{Wojcie} for the APS boundary conditions.

\subsection{Eigenvalues decaying exponentially as $R\rightarrow \infty $}\label{ss3.4}

Let $f:[-1, 1]\rightarrow [0, 1]$ be a smooth even cut-off function such that
\begin{align}\begin{aligned}\label{e.311}
f(x)=\left\{\begin{array}{ll}
     0 \, \text{ for }&-\frac{1}{4} \leq x\leq \frac{1}{4};\\
     1 \, \text{ for }&\frac{1}{2} \leq |x|\leq 1,
  \end{array}
\right.
\end{aligned}\end{align}
then we define $f_{R}(y, x_{m}):=f(\frac{x_{m}}{R})$ on $Y_{[-R,R]}$. We extend $f_{R}$ by $1$ from $Y_{[-R,R]}$ to the whole fiber $Z_{R}$. We set
\begin{align}\begin{aligned}\label{e.312}
 f_{1,R}(x):=\left\{\begin{array}{ccc}
            f_{R}(x) ,& & x\in Z_{1, R}, \\
             & & \\
            0 ,& & x\in Z_{2, R};           \end{array}
 \right. \quad
 f_{2,R}(x):=\left\{\begin{array}{ccc}
            0 ,& & x\in Z_{1, R}, \\
             & & \\
            f_{R}(x) ,& & x\in Z_{2, R}.     \end{array}
 \right.\end{aligned}\end{align}

\begin{defn}\label{d.19}
Let $Z_{1, \infty}\cong Z_{1, R}\cup Y_{[0,+\infty)}$
(resp. $Z_{2, \infty}\cong Y_{(-\infty,0]}\cup Z_{2, R}$). We extends all the geometric data from $Z_{1, R}$ (resp. $Z_{2, R}$) to $Z_{1, \infty}$ (resp. $Z_{2, \infty}$) by using the product structures. For $i=1,2$, let $\Ker_{L^{2}}(D^{Z_{i, \infty}})^{2}$ be the $L^{2}-$integrable solutions of the Hodge-Laplacian $(D^{Z_{i, \infty}})^{2}$ on $Z_{i, \infty}$.
\end{defn}

\begin{defn}\label{d.13} We define $
\mathscr{W}_{i, R}:=\text{span}\{f_{i,R}\,s_{i},\, s_{i}\in \Ker_{L^{2}}(D^{Z_{i, \infty}})^{2} \}$
two subspaces of $\Omega(Z_{i,R}, F_{R})$ and let $
\mathscr{W}_{R}:=\mathscr{W}_{1, R}\oplus \mathscr{W}_{2, R}$
 regarded as a subspace of $\Omega(Z_{R}, F_{R})$.
\end{defn}

\begin{lemma}\label{l.17} There exist $R_{0}>0$, $C>0$ such that for any $R>R_{0}$ and $ s \in \mathscr{W}_{R}$
\begin{align}\begin{aligned}\label{e.190}
\|(D^{Z_{R}})^{2}s\|_{L^{2}(Z_{R})}\leq C e^{-\frac{R\sqrt{\delta}}{8}}\|s\|_{L^{2}(Z_{R})}.
\end{aligned}\end{align}
\end{lemma}

\begin{proof}
 As $ s \in \mathscr{W}_{R}$, there exist $s_{i}\in \Ker_{L^{2}}(D^{Z_{i, \infty}})^{2}, i=1, 2$, such that
 \begin{align}\begin{aligned}\label{e.314}
s=f_{1,R}s_{1}+f_{2,R}s_{2}.
\end{aligned}\end{align}
On the $Y_{[-R,R]}$, we expand $s_{1}, s_{2}$ in term of basis (\ref{e.271}) for $(y,x_{m})\in Y_{[-R,R]}$
\begin{align}\begin{aligned}\label{e.315}
\left\{\begin{array}{lll}
     s_{1}(y, x_{m})=\sum_{k=1}^{\infty}e^{-\sqrt{\mu_{k}}(x_{m}+R)}\left(a_{k}\phi_{k}(y)+c_{k}dx_{m}\wedge\phi_{k}(y)\right), \\
     s_{2}(y, x_{m})=\sum_{k=1}^{\infty}e^{+\sqrt{\mu_{k}}(x_{m}-R)}\left(b_{k}\phi_{k}(y)+d_{k}dx_{m}\wedge\phi_{k}(y)\right).
\end{array}
\right.
\end{aligned}\end{align}
By (\ref{e.274}) and $s_{i}\in \Ker_{L^{2}}(D^{Z_{i, \infty}})^{2}, i=1, 2$, we have
$(D^{Z_{R}})^{2}(f_{i,R}s_{i})=-\frac{\partial^{2}f_{i,R}}{\partial x^{2}_{m}}s_{i}-2\frac{\partial f_{i,R}}{\partial x_{m}}\frac{\partial s_{i} }{\partial x_{m}}.$
Hence from (\ref{e.669}), (\ref{e.312}), (\ref{e.314}) and (\ref{e.315}), we get for $R\geq 1$ by using the notation (\ref{e.283}) that
\begin{align}\begin{aligned}\label{e.316}
&\|(D^{Z_{R}})^{2}s\|^{2}_{L^{2}(Z_{R})}\\
&\leq 2\sum_{i=1}^{2}\Big(\big\|\frac{\partial^{2}f_{i,R}}{\partial x^{2}_{m}}s_{i} \big\|^{2}_{L^{2}(Z_{R})}+2\big\| \frac{\partial f_{i,R}}{\partial x_{m}}\frac{\partial s_{i} }{\partial x_{m}}\big\|^{2}_{L^{2}(Z_{R})}\Big)\\
&\leq 2\sup_{x_{m}\in \mathbb{R}}\left|\frac{\partial^{2}f_{R}}{\partial x^{2}_{m}}\right|^{2}
\left(\|s_{1}\|^{2}_{L^{2}(Y_{[-\frac{R}{2},-\frac{R}{4}]})}+
\|s_{2}\|^{2}_{L^{2}(Y_{[\frac{R}{4},\frac{R}{2}]})}\right)\\
&+4\sup_{x_{m}\in \mathbb{R}}\left|\frac{\partial f_{R}}{\partial x_{m}}\right|^{2}
\left(\Big\|\frac{\partial s_{1} }{\partial x_{m}}\Big\|^{2}_{L^{2}(Y_{[-\frac{R}{4},-\frac{R}{2}]})}
+
\Big\|\frac{\partial s_{2} }{\partial x_{m}}\Big\|^{2}_{L^{2}(Y_{[\frac{R}{4},\frac{R}{2}]})}\right)\\
&\leq\frac{2c_{1}}{R^{4}}\sum_{k=1}^{\infty}\frac{e^{-R\sqrt{\mu_{k}}}(1-e^{-\frac{R}{2}\sqrt{\mu_{k}}})}{2\sqrt{\mu_{k}}}
|\sigma_{k}|^{2}\\
&\quad \quad\quad \quad\quad\quad \quad+\frac{4c_{2}}{R^{2}}\sum_{k=1}^{\infty}\frac{\mu_{k}e^{-R\sqrt{\mu_{k}}}(1-e^{-\frac{R}{2}\sqrt{\mu_{k}}})}{2\sqrt{\mu_{k}}}
|\sigma_{k}|^{2}\\
&\leq c_{3}e^{-\frac{\sqrt{\delta}}{2}R}\sum_{k=1}^{\infty}\frac{|\sigma_{k}|^{2}}{2\sqrt{\mu_{k}}},
\end{aligned}\end{align}
where
\begin{align}\begin{aligned}\label{e.757}
c_{1}=\max_{u\in[-1, 1]}\left|\frac{\partial^{2}f}{\partial^{2}u}(u)\right|^{2},
\, c_{2}=\max_{u\in[-1, 1]}\left|\frac{\partial f}{\partial u}(u)\right|^{2}.
\end{aligned}\end{align}
On the other hand, from (\ref{e.669}), (\ref{e.312}), (\ref{e.314}) and (\ref{e.315}) we have
\begin{align}\begin{aligned}\label{e.318}
\|s\|^{2}_{L^{2}(Z_{R})}&=\|F_{R}s_{1}+F_{R}s_{2}\|^{2}_{L^{2}(Z_{R})}\\
&\geq \int_{-R}^{-\frac{R}{2}}\|s_{1}\|^{2}_{L^{2}(Y\times\{x_{m}\})}dx_{m}+\int_{\frac{R}{2}}^{R}\|s_{2}\|^{2}_{L^{2}(Y\times\{x_{m}\})}dx_{m}\\
&=\sum_{k=1}^{\infty}\frac{1-e^{-R\sqrt{\mu_{k}}}}{2\sqrt{\mu_{k}}}|\sigma_{k}|^{2}
\geq(1-e^{-R\sqrt{\delta}})\sum_{k=1}^{\infty}\frac{|\sigma_{k}|^{2}}{2\sqrt{\mu_{k}}}.
\end{aligned}\end{align}
By (\ref{e.316}) and (\ref{e.318}), there exist $R_{0}>0, c_{4}>0$ such that for any $R>R_{0}$
\begin{align}\begin{aligned}\label{e.319}
\|(D^{Z_{R}})^{2}s\|^{2}_{L^{2}(Z_{R})}\leq \frac{c_{3}e^{-\frac{\sqrt{\delta}}{2}R}}{1-e^{-R\sqrt{\delta}}}\|s\|^{2}_{L^{2}(Z_{R})}\leq c_{4}\cdot e^{-\frac{R}{4}\sqrt{\delta}}\|s\|^{2}_{L^{2}(Z_{R})}.
\end{aligned}\end{align}
The proof is now completed.
\end{proof}

Similar to Lemma \ref{l.17}, we have
\begin{lemma}\label{l.18} For $i=1,2$, there exist $R_{0}>0, C>0$ such that for any $R>R_{0}$ and $s \in \mathscr{W}_{i, R}$
\begin{align}\begin{aligned}\label{e.191}
\|(D^{Z_{i, R}})^{2}s\|_{L^{2}(Z_{i, R})}\leq C e^{-\frac{\sqrt{\delta}R}{8}}\|s\|_{L^{2}(Z_{i, R})}.
\end{aligned}\end{align}
\end{lemma}

\begin{proof}
We first establish (\ref{e.191}) for $i=1$. There exists $s_{1}\in \Ker_{L^{2}} (D^{Z_{1, \infty}})^{2}$ such that $s=f_{1,R}s_{1}$. We expand $s_{1}$ for
$(y, x_{m})\in Y _{[-R, 0]}$ in term of basis
(\ref{e.271}),
\begin{align}\begin{aligned}\label{e.320}
s_{1}(y, x_{m})=\sum_{k=1}^{\infty}e^{-\sqrt{\mu_{k}}(x_{m}+R)}\left(a_{k}\phi_{k}(y)+c_{k}dx_{m}\wedge\phi_{k}(y)\right),
\end{aligned}\end{align}
then, for $R\geq 1$, as (\ref{e.316}) we have by using $s_{1}\in \Ker_{L^{2}} (D^{Z_{1, \infty}})^{2}$, (\ref{e.757}) and (\ref{e.320})
\begin{align}\begin{aligned}\label{e.321}
&\|(D^{Z_{1, R}})^{2}s\|^{2}_{L^{2}(Z_{1, R})}\\
&\leq \frac{2c_{1}}{R^{4}}
\|s_{1}\|^{2}_{L^{2}(Y_{[-\frac{R}{2},-\frac{R}{4}]})}+\frac{4c_{2}}{R^{2}}
\Big\|\frac{\partial s_{1}}{\partial x_{m}}\Big\|^{2}_{L^{2}(Y_{[-\frac{R}{2},-\frac{R}{4}]})}\\
&\leq\frac{2c_{1}}{R^{4}}\sum_{k=1}^{\infty}\frac{e^{-R\sqrt{\mu_{k}}}(1-e^{-\frac{R}{2}\sqrt{\mu_{k}}})}{2\sqrt{\mu_{k}}}
\left(|a_{k}|^{2}+|c_{k}|^{2}\right)\\
&\qquad\qquad+\frac{4c_{2}}{R^{2}}\sum_{k=1}^{\infty}\frac{\mu_{k}e^{-R\sqrt{\mu_{k}}}(1-e^{-\frac{R}{2}\sqrt{\mu_{k}}})}{2\sqrt{\mu_{k}}}
\left(|a_{k}|^{2}+|c_{k}|^{2}\right)\\
&\leq c_{3}e^{-\frac{\sqrt{\delta}}{2}R}\sum_{k=1}^{\infty}\frac{|a_{k}|^{2}+|c_{k}|^{2}}{2\sqrt{\mu_{k}}}.
\end{aligned}\end{align}
 On the other hand as (\ref{e.318}) we have
\begin{align}\begin{aligned}\label{e.322}
\|s\|^{2}_{L^{2}(Z_{1, R})}&\geq \int_{-R}^{-\frac{R}{2}}\|s_{1}\|^{2}_{L^{2}(Y\times \{x_{m}\})}dx_{m}
=\sum_{k=1}^{\infty}\frac{1-e^{-R\sqrt{\mu_{k}}}}{2\sqrt{\mu_{k}}}(|a_{k}|^{2}+|c_{k}|^{2})\\
&\geq(1-e^{-R\sqrt{\delta}})\sum_{k=1}^{\infty}\frac{|a_{k}|^{2}+|c_{k}|^{2}}{2\sqrt{\mu_{k}}}.
\end{aligned}\end{align}
By (\ref{e.321}) and (\ref{e.322}), there exist $C_{1}>0$ such that for $R\geq 1$ large enough
\begin{align}\begin{aligned}\label{e.323}
\|(D^{Z_{1, R}})^{2}s\|^{2}_{L^{2}(Z_{1, R})}\leq \frac{c_{3}e^{-\frac{\sqrt{\delta}}{2}R}}{1-e^{-R\sqrt{\delta}}}\|s\|^{2}_{L^{2}(Z_{1, R})}\leq C_{1}e^{-\frac{\sqrt{\delta}}{4}R}\|s\|^{2}_{L^{2}(Z_{1, R})}.
\end{aligned}\end{align}
Thus we get (\ref{e.191}) for $i=1$. For $i=2$, the estimate (\ref{e.191}) is true as that the Hodge star operator exchanges the relative and absolute boundary conditions or we can follow the same proof as for $i=1$. The proof is completed.
\end{proof}

\begin{defn}\label{d.18} Let $\mathbb{I}$ be a subset of $\mathbb{R}$. For $i=1, 2$, let $\mathbb{F}_{R}^{\mathbb{I}}$ (resp. $\mathbb{F}_{i, R}^{\mathbb{I}}$)
be the direct sum of the eigenspaces of $(D^{Z_{R}})^{2}$ (resp. $(D^{Z_{i, R}})^{2}$) associated to eigenvalues $\lambda \in \mathbb{I}$. Let $P_{R}^{\mathbb{I}}$ (resp. $P_{i, R}^{\mathbb{I}}$) be the orthogonal projection operator from
 $L^{2}(Z_{R}, \Lambda (T^{*}Z_{R})\otimes F_{R})$ (resp. $L^{2}(Z_{i, R}, \Lambda (T^{*}Z_{i, R})\otimes F_{R})$) onto $\mathbb{F}_{R}^{\mathbb{I}}$ (resp. $\mathbb{F}_{i, R}^{\mathbb{I}}$).
 In our application, for $c>0$, $R\geq 0$, the subset $\mathbb{I}\subset \mathbb{R}$ will be taken into $[0, c],\,(0, c],\,\{0\}$.
\end{defn}

\begin{lemma}\label{l.20} For $i=1,2$, there exist $C>0$ and $R_{0}>0$ such that for any $R>R_{0}$ and $s\in \mathscr{W}_{R}$ (resp. $\mathscr{W}_{i,R}$)
\begin{align}\begin{aligned}\label{e.193}
&\big\|\big({\rm Id}-P_{R}^{[0, e^{-\frac{R\sqrt{\delta}}{16}}]}\big)s \big\|_{L^{2}(Z_{R})} \leq Ce^{-\frac{R\sqrt{\delta}}{16}} \cdot \|s\|_{L^{2}(Z_{R})}\\
\big(\text{resp.}\,\,&\big\|\big({\rm Id}-P_{i, R}^{[0,\, e^{-\frac{R\sqrt{\delta}}{16}}]}\big)s\big\|_{L^{2}(Z_{i, R})} \leq Ce^{-\frac{R\sqrt{\delta}}{16}} \cdot \|s\|_{L^{2}(Z_{i, R})}\,\,\big).
\end{aligned}\end{align}
\end{lemma}
\begin{proof}
Let $\{\psi_{i}\}_{i=1}^{\infty}$ be an orthonormal basis of $L^{2}(Z_{R}, \Lambda (T^{*}Z_{R})\otimes F_{R})$ consisting of smooth eigensections of $(D^{Z_{R}})^{2}$ such that
\begin{align}\begin{aligned}\label{e.324}
 (D^{Z_{R}})^{2}\psi_{i}=\rho_{i}\psi_{i}, \quad 0<\rho_{1}\leq \rho_{2}\leq \cdots \leq \rho_{i}\leq \cdots \rightarrow + \infty.
\end{aligned}\end{align}
We expand $s$ in term of this basis as $s=\sum_{k}a_{k}\psi_{k}$, then by Lemma \ref{l.17} we have
\begin{align}\begin{aligned}\label{e.325}
\big\|\big( {\rm Id}-P^{[0, \exp(-\frac{R\sqrt{\delta}}{16})]}\big)s \big\|^{2}_{L^{2}(Z_{R})}
&=\sum_{\rho_{k}^{2}>\exp(-\frac{R\sqrt{\delta}}{8})}a^{2}_{k}\leq \sum_{\rho_{k}^{2}>\exp(-\frac{R\sqrt{\delta}}{8})}e^{\frac{R\sqrt{\delta}}{8}}\rho_{k}^{2}a^{2}_{k}\\
&\leq e^{\frac{R\sqrt{\delta}}{8}}\|(D^{Z_{R}})^{2}s\|^{2}_{L^{2}(Z_{R})}\leq C{e}^{-\frac{R\sqrt{\delta}}{8}}\|s\|^{2}_{L^{2}(Z_{R})}.
\end{aligned}\end{align}
Thus we get the first line in (\ref{e.193}). For $i=1,2$, the estimates are obtained in the same way. The proof is completed.
\end{proof}

\begin{prop}\label{p.4} For $i=1, 2$, there exist $R_{0}>0$ such that for all $R>R_{0}$, \\
{\rm(a)}\,the projection $P_{R}^{[0, \,e^{-\frac{R\sqrt{\delta}}{16}}]}$
restricted on $\mathscr{W}_{R}$ is injective.
In particular, $(D^{Z_{R}})^{2}$ has at least $\dim \mathscr{W}_{R}$ {\rm(see Def. \ref{d.13})} eigenvalues laying in
$[0, e^{-\frac{R\sqrt{\delta}}{16}}]$;\\
{\rm(b)}\,the projection $P_{i, R}^{[0, \,e^{-\frac{R\sqrt{\delta}}{16}}]}$ restricted to $\mathscr{W}_{i, R}$ is injective.
In particular, $(D^{Z_{i, R}})^{2}$ has at least $\dim \mathscr{W}_{i, R}$ eigenvalues laying in $[0, e^{-\frac{R\sqrt{\delta}}{16}}]$.
\end{prop}

\begin{proof}
Let $s\in \mathscr{W}_{R}$ and assume that $P^{[0, \exp(-\frac{R\sqrt{\delta}}{16})]}s=0$. By Lemma \ref{l.20}, we have for $R$ sufficiently large
\begin{align}\begin{aligned}\label{e.326}
\|s\|_{L^{2}(Z_{R})}=\|\big({\rm Id}-P^{[0, e^{-\frac{R\sqrt{\delta}}{16}}]}\big)s\|_{L^{2}(Z_{R})}\leq Ce^{-\frac{R\sqrt{\delta}}{16}}\|s\|_{L^{2}(Z_{R})}\leq \frac{1}{2}\|s\|_{L^{2}(Z_{R})}.
\end{aligned}\end{align}
Thus we get part (a). The same proof gives part (b).
\end{proof}

\subsection{Eigenvalues bounded away from 0}\label{ss3.5}
To prove Theorem \ref{t.10}, we establish first the following proposition:

\begin{prop}\label{p.13} There exist $R_{0}>0$ and $c>0$ such that for any $R>R_{0}$ and $\psi\in \Omega(Z_{R}, F_{R})$
such that
\begin{align}\begin{aligned}\label{e.327}
(D^{Z_{R}})^{2}\psi=\lambda \psi, \quad 0\leq\lambda<\frac{3\delta}{4}, \quad \|\psi\|_{L^{2}(Z_{R})}=1,
\end{aligned}\end{align}
and $\psi$ lays in the orthogonal complement of $P_{R}^{[0,\, e^{-\frac{R\sqrt{\delta}}{16}}]}\mathscr{W}_{R}$ in
$ L^{2}\left(Z_{R}, \Lambda (T^{*}Z_{R})\otimes F_{R}\right)$, i.e.,
$
\psi\in \big\{P_{R}^{[0,\, e^{-\frac{R\sqrt{\delta}}{16}}]}\mathscr{W}_{R}\big\}^{\perp},
$
then we get $\lambda \geq c>0.$
Consequently, the spectral projection
$$P_{R}^{[0,\, e^{-\frac{R\sqrt{\delta}}{16}}]}: L^{2}\left(Z_{R}, \Lambda (T^{*}Z_{R})\otimes F_{R}\right)
\rightarrow \mathbb{F}^{[0, e^{-\frac{R\sqrt{\delta}}{16}}]}_{R} $$
restricted on the subspace $\mathscr{W}_{R}$ is surjective.
\end{prop}
To prove Proposition \ref{p.13}, we need first to establish some lemmas.
Let $h:[-1, +\infty)\rightarrow [0, 1]$ be a smooth cut-off function such that
\begin{align}\begin{aligned}\label{e.332}
h(x)=\left\{\begin{array}{ll}
     1 & \text{ for } x \in [-1, -\frac{3}{4}], \\
     0 & \text{ for } x \in [-\frac{1}{2}, +\infty).
\end{array}
\right.
\end{aligned}\end{align}
We define $h_{1, R}:Z_{1, \infty}\rightarrow [0, 1]$ and $h_{2, R}:Z_{2, \infty}\rightarrow [0, 1]$ by
\begin{align}\begin{aligned}\label{e.333}
h_{1, R}(x)&=\left\{\begin{array}{cc}
     1 & \text{ for }x\in Z_{1,R}\backslash Y_{[-R,0]},  \\
     h(\frac{x_{m}}{R}) & \text{ for }(y, x_{m})\in Y_{[-R, +\infty)};
  \end{array}\right.\\
h_{2, R}(x)&=\left\{\begin{array}{ll}
     1 &  \text{ for } x\in Z_{2,R}\backslash Y_{[0,R]}, \\
    h(-\frac{x_{m}}{R})& \text{ for } (y, x_{m}) \in Y_{(-\infty, R]}.
  \end{array}
\right.\end{aligned}\end{align}
Then we have
\begin{align}\begin{aligned}\label{e.677}
{\rm supp}(h_{1,R})\subset Z_{1, \infty}\backslash Y_{[-\frac{R}{2},+\infty)},\quad
{\rm supp}(h_{2,R})\subset Z_{2, \infty}\backslash Y_{(-\infty,\frac{R}{2}]}.
\end{aligned}\end{align}
Let $\psi$ be given as in Propostion \ref{p.13}, then for $i=1, 2$ we define $\psi^{\infty}_{i, R}$ on $Z_{i, \infty}$ as follows:
\begin{align}\begin{aligned}\label{e.335}
\psi^{\infty}_{i, R}=h_{i, R}(x)\psi(x).
\end{aligned}\end{align}

\begin{lemma}\label{l.34} For $i=1,2$, there exist $C>0$ and $R_{0}>0$ such that for any $R>R_{0}$,\\
{\rm 1)}\, we have
\begin{align}\begin{aligned}\label{e.346}
\|\psi^{\infty}_{1, R}\|_{L^{2}(Z_{1, \infty})}^{2}+\|\psi^{\infty}_{2, R}\|_{L^{2}(Z_{2, \infty})}^{2}\geq \frac{1}{8};
\end{aligned}\end{align}
{\rm 2)}\, for any $s\in \Ker _{L^{2}}(D^{Z_{1, \infty}})^{2}$, we have
\begin{align}\begin{aligned}\label{e.336}
|\langle\psi^{\infty}_{i, R}, s\rangle_{L^{2}(Z_{i, \infty})}|\leq Ce^{-\frac{\sqrt{\delta}}{32}R}\|s\|_{L^{2}(Z_{i, \infty})}.
\end{aligned}\end{align}
\end{lemma}

\begin{proof}
By our definition of $h_{1,R}$ and $h_{2,R}$, we see that
\begin{align}\begin{aligned}\label{e.679}
{\rm supp}(1-h_{1,R})\subset Y_{[-\frac{3R}{4},0]}\quad {\rm and}\quad {\rm supp}(1-h_{2,R})\subset Y_{[0,\frac{3R}{4}]}.
\end{aligned}\end{align}
By  Lemma \ref{l.28}, (\ref{e.327}), (\ref{e.335}) and (\ref{e.679}), we have for $R$ sufficiently large
\begin{align}\begin{aligned}\label{e.680}
\sum_{i=1}^{2}\|\psi^{\infty}_{i, R}\|_{L^{2}(Z_{i, \infty})}^{2}
\geq &\sum_{i=1}^{2}\left(\frac{1}{2}\|\psi\|^{2}_{L^{2}(Z_{i,R})}-\|(1-h_{i,R})\psi\|^{2}_{L^{2}(Z_{i,R})}\right)\\
\geq &\frac{1}{2}-\int_{-\frac{3R}{4}}^{\frac{3R}{4}}\big\|\psi\big\|^{2}_{L^{2}(Y\times \{x_{m}\})}dx_{m}
\geq \frac{1}{2}-CR\,e^{-\frac{\sqrt{\delta}}{8}R}\geq\frac{1}{4}.
\end{aligned}\end{align}
By (\ref{e.312}), (\ref{e.333}) and (\ref{e.335}), we have
\begin{align}\begin{aligned}\label{e.337}
\big|\langle\psi^{\infty}_{1, R}, s\rangle_{L^{2}(Z_{1, \infty})}\big|&=\big| \big\langle h_{1, R}\psi, f_{1,R}s\big\rangle_{L^{2}(Z_{1, R})}\big|\\
&\leq \big| \big\langle \psi, f_{1,R}s\big\rangle_{L^{2}(Z_{1, R})}\big|+\big| \big\langle(1-h_{1, R})\psi, f_{1,R}s\big\rangle_{L^{2}(Z_{1, R})}\big|.
\end{aligned}\end{align}
By Lemma \ref{l.20}, (\ref{e.327}), the fact $\psi\in \big\{P_{R}^{[0,\, e^{-\frac{R\sqrt{\delta}}{16}}]}\mathscr{W}_{R}\big\}^{\perp}$ and Cauchy-Schwartz inequality, we have for any $f_{1,R}s\in \mathscr{W}_{R}$ and $R$ sufficiently large
\begin{align}\begin{aligned}\label{e.331}
|\langle\psi, f_{1,R}s\rangle_{L^{2}(Z_{1,R})}|
&=\big|\big\langle \psi, \big({\rm Id}-P_{1,R}^{[0,\, \exp(-\frac{R\sqrt{\delta}}{16})]}\big)f_{1,R}s\big\rangle_{L^{2}(Z_{1,R})}\big|\\
&\leq Ce^{-\frac{R\sqrt{\delta}}{16}}\|f_{1,R}s\|_{L^{2}(Z_{1,R})}\leq Ce^{-\frac{R\sqrt{\delta}}{16}}\|s\|_{L^{2}(Z_{1, \infty})}.
\end{aligned}\end{align}
For the second summand of (\ref{e.337}), by Lemma \ref{l.28}, (\ref{e.312}) and (\ref{e.333}) we get
\begin{align}\begin{aligned}\label{e.339}
&\big| \big\langle(1-h_{1, R})\psi, f_{1,R}s\big\rangle_{L^{2}(Z_{1, R})}\big|\leq \|\psi\|_{L^{2}(Y_{[-\frac{3R}{4},0]})}\|s\|_{L^{2}(Z_{1, \infty})}\\
&\leq C R^{\frac{1}{2}}e^{-\frac{\sqrt{\delta}}{16}R} \|s\|_{L^{2}(Z_{1, \infty})}\leq C{e}^{-\frac{\sqrt{\delta}}{32}R}
\|s\|_{L^{2}(Z_{1, \infty})}.
\end{aligned}\end{align}
Then for $i=1$ (\ref{e.336}) follows from (\ref{e.337}), (\ref{e.331}) and (\ref{e.339}).
In the same way we get (\ref{e.336}) for $i=2$.
\end{proof}

Now we begin to prove our main result in this subsection.

\textbf{Proof of Proposition \ref{p.13}.}
In \cite[Prop. 4.9]{APS} Atiyah, Patodi and Singer given a topological interpretation of the dimension of the space $\Ker_{L^{2}}(D^{Z_{i, \infty}})^{2, (p)}$
of $L^{2}$-harmonic forms on $Z_{i, \infty}$. We have the following isomorphism: for $ 0\leq p \leq m$
\begin{align}\begin{aligned}\label{e.202}
 \Ker_{L^{2}}(D^{Z_{i, \infty}})^{2, (p)}\cong \Im \left\{H^{p}(Z_{i, R}, Y, F_{R})\rightarrow
H^{p}(Z_{i, R}, F_{R})\right\},
\end{aligned}\end{align}
hence we have for $ 0\leq p\leq m $ under the assumption
(\ref{e.356})
\begin{align}\begin{aligned}\label{e.203}
h^{(p)}_{1, \infty}&:=\dim \Ker_{L^{2}}(D^{Z_{1, \infty}})^{2, (p)}\leq \dim H^{p}(Z_{1}, F)=\dim \mathbb{F}^{\{0\}, (p)}_{1, R}, \\
h^{(p)}_{2, \infty}&:=\dim \Ker_{L^{2}}(D^{Z_{2, \infty}})^{2, (p)}\leq\dim H^{p}(Z_{2}, Y, F)=\dim \mathbb{F}^{\{0\}, (p)}_{2, R}.
\end{aligned}\end{align}
For $i=1, 2$, let $\mathbb{T}_{i}$ be the orthonormal projection operator onto $\Ker_{L^{2}}(D^{Z_{i, \infty}})^{2}$, then
we define
\begin{align}\begin{aligned}\label{e.341}
\widetilde{\psi}_{i, R}:=\psi^{\infty}_{i, R}-\mathbb{T}_{i}\left(\psi^{\infty}_{i, R}\right)\in \big\{\Ker_{L^{2}}(D^{Z_{i, \infty}})^{2}\big\}^{\perp}.
\end{aligned}\end{align}
 It follows form Lemma \ref{l.34} and (\ref{e.341}) that for $R$ sufficiently large
 \begin{align}\begin{aligned}\label{e.342}
\sum_{i=1}^{2}\|\widetilde{\psi}_{i, R}\|_{L^{2}(Z_{i, \infty})}^{2}
\geq \frac{1}{3}\sum_{i=1}^{2}\|\psi^{\infty}_{i, R}
\|_{L^{2}(Z_{i, \infty})}^{2}\geq \frac{1}{12}.
\end{aligned}\end{align}
  For $i=1,2$ and $b\in U\subset S$, as $U$ is compact we set
\begin{align}\begin{aligned}\label{e.671}
\gamma_{i}=\inf_{b\in U}\min\Big\{\lambda>0\big|\,\lambda \in \text{spec}(D^{Z_{i,\infty,b}})^{2}\Big\}>0,\quad \gamma_{0}:=\min\{\gamma_{1}, \gamma_{2}\}.
\end{aligned}\end{align}
 Then it follows from the Min-Max Principle (cf. \cite[Appendix C.3]{MaMa07}) that
\begin{align}\begin{aligned}\label{e.343}
\big\langle(D^{Z_{i, \infty}})^{2}\widetilde{\psi}_{i, R}, \widetilde{\psi}_{i, R}\big\rangle_{L^{2}(Z_{i, \infty})}
\geq \gamma_{i}\|\widetilde{\psi}_{i, R}\|^{2}_{L^{2}(Z_{i, \infty})}.
\end{aligned}\end{align}
By (\ref{e.327}), we have
\begin{align}\begin{aligned}\label{e.344}
\lambda&=\big\langle(D^{Z_{R}})^{2}\psi, \psi\big\rangle_{L^{2}(Z_{R})}
=\|D^{Z_{R}}\psi\|^{2}_{L^{2}(Z_{R})}\\
&=\sum_{i=1}^{2}\left\|D^{Z_{R}}h_{i, R}\psi+D^{Z_{R}}(1-h_{i, R})\psi\right\|^{2}_{L^{2}(Z_{i, R})}\\
&\geq \frac{1}{2}\sum_{i=1}^{2}\left\|D^{Z_{i, \infty}}\psi_{i, R}^{\infty}\right\|_{L^{2}(Z_{i, \infty})}^{2}-\sum_{i=1}^{2}
\left\|D^{Z_{R}}(1-h_{i, R})\psi\right\|_{L^{2}(Z_{i, R})}^{2}.
\end{aligned}\end{align}
 For the first term, by (\ref{e.341}), (\ref{e.342}), (\ref{e.671}) and (\ref{e.343}) we have
\begin{align}\begin{aligned}\label{e.345}
 \sum_{i=1}^{2}\left\|D^{Z_{i, \infty}}\psi_{i, R}^{\infty}\right\|_{L^{2}(Z_{i, \infty})}^{2}
&=\sum_{i=1}^{2}\big\langle D^{Z_{i, \infty}}\widetilde{\psi}_{i, R}, D^{Z_{i, \infty}}\widetilde{\psi}_{i, R}\big\rangle_{L^{2}(Z_{i, \infty})}\\
&=\sum_{i=1}^{2}\big\langle(D^{Z_{i, \infty}})^{2}\widetilde{\psi}_{i, R}, \widetilde{\psi}_{i, R}\big\rangle_{L^{2}(Z_{i, \infty})}
\geq \frac{\gamma_{0}}{12}.
\end{aligned}\end{align}

Using the fact that $D^{Z_{R}}={\bf c}(dx^{m})\frac{\partial}{\partial x_{m}}+D^{Y}$ on $Y_{[-R,R]}$, where ${\bf c}(dx_{m})=dx_{m}\wedge-i(\frac{\partial}{\partial x_{m}})$ denotes the Clifford action, we have
$D^{Z_{R}}(1-h_{i, R})\psi=(1-h_{i, R})\left(D^{Z_{R}}\psi\right)-{\bf c}(dx_{m})\frac{\partial h_{i, R}}{\partial x_{m}}\psi$, hence from Lemma \ref{l.28}, (\ref{e.332}), (\ref{e.679}), we get
\begin{align}\begin{aligned}\label{e.347}
\sum_{i=1}^{2}&
\left\|D^{Z_{R}}(1-h_{i, R})\psi\right\|_{L^{2}(Z_{i, R})}^{2}\\
&\leq 2\sum_{i=1}^{2}
\left\|(1-h_{i, R})\left(D^{Z_{R}}\psi\right)\right\|^{2}_{L^{2}(Z_{i, R})}+2\sum_{i=1}^{2}
\big\|{\bf c}(dx_{m})\frac{\partial h_{i, R}}{\partial x_{m}}\psi\big\|^{2}_{L^{2}(Z_{i, R})}\\
&\leq 2\int_{-\frac{3R}{4}}^{\frac{3R}{4}}\left\|D^{Z_{R}}\psi\right\|_{L^{2}(Y\times\{x_{m}\})}^{2}dx_{m}+
\frac{2b_{1}}{R^{2}}\int_{-\frac{3R}{4}}^{-\frac{R}{2}}\big\|\psi\big\|_{L^{2}(Y\times\{x_{m}\})}^{2}dx_{m}\\ &\quad\quad\quad\quad+\frac{2b_{1}}{R^{2}}\int_{\frac{R}{2}}^{\frac{3R}{4}}\big\|\psi\big\|_{L^{2}(Y\times\{x_{m}\})}^{2}dx_{m}\\
&\leq 3C \lambda R\,e^{-\frac{\sqrt{\delta}}{16}R}+2\times \frac{2b_{1}}{R^{2}}\frac{CR}{4}e^{-\frac{\sqrt{\delta}}{16}R}
=3C R\,e^{-\frac{\sqrt{\delta}}{16}R}\cdot\lambda +\frac{Cb_{1}}{R}e^{-\frac{\sqrt{\delta}}{16}R},
\end{aligned}\end{align}
where we denote
$
b_{1}=\max_{u\in[-1, 1]}\left|\frac{\partial h}{\partial u}(u)\right|^{2}.
$
By (\ref{e.344}), (\ref{e.345}) and (\ref{e.347}), we get
\begin{align}\begin{aligned}\label{e.349}
\lambda\geq \frac{\gamma_{0}}{24}-3C R\,e^{-\frac{\sqrt{\delta}}{16}R}\cdot\lambda -\frac{Cb_{1}}{R}e^{-\frac{\sqrt{\delta}}{16}R}.
\end{aligned}\end{align}
Let $R$ be enough large such that
$
3C R\,e^{-\frac{\sqrt{\delta}}{16}R}\leq \frac{1}{2}, \, \frac{Cb_{1}}{R}e^{-\frac{\sqrt{\delta}}{16}R}\leq \frac{\gamma_{0}}{48},
$
then we get from (\ref{e.347}) that $\lambda\geq \frac{\gamma_{0}}{72}>0$.
Now the proof of Proposition \ref{p.13} is completed.\\

Similar to the proof of Proposition \ref{p.13}, by using Lemmas \ref{l.29}, \ref{l.30} and \ref{l.34}, we have

\begin{prop}\label{p.14a} For $i=1,2$, there exist constants $c>0$ and $R_{0}>0$  such that for any $R>R_{0}$ and
$\psi_{i}\in \Omega_{\rm bd}\left(Z_{i, R}, F_{R}\right)$ such that
\begin{align}\begin{aligned}\label{e.351}
(D^{Z_{i, R}})^{2}\psi_{i}=\lambda_{i} \psi_{i}, \quad 0\leq\lambda_{i}<\frac{3\delta}{4}, \quad \|\psi_{i}\|_{L^{2}(Z_{i, R})}=1,
\end{aligned}\end{align}
and $\psi_{i}$ lays in the orthogonal complement of $P_{i, R}^{[0, e^{-\frac{R\sqrt{\delta}}{16}}]}\mathscr{W}_{i, R}$ in
$L^{2}\left(Z_{i, R}, \Lambda (T^{*}Z_{i, R})\otimes F_{R}\right)$
i.e., $\psi_{i}\in \big\{P_{i, R}^{[0,\, e^{-\frac{R\sqrt{\delta}}{16}}]}\mathscr{W}_{i, R}\big\}^{\perp}$,
then we have $\lambda_{i} \geq c>0$. Consequently, the spectral projection
$$P_{i,R}^{[0,\, e^{-\frac{R\sqrt{\delta}}{16}}]}: L^{2}\left(Z_{i,R}, \Lambda (T^{*}Z_{i,R})\otimes F_{i,R}\right)
\rightarrow \mathbb{F}^{[0, e^{-\frac{R\sqrt{\delta}}{16}}]}_{i,R} $$
restricted on the subspace $\mathscr{W}_{i,R}$ is surjective.
\end{prop}

 In fact we will show in the following paragraph that $(D^{Z_{i, R}})^{2}, i=1, 2$, does not have the non-trivial exponentially decreasing
small eigenvalues when $R$ goes to infinity.

\begin{prop}\label{p.7} There exist $R_{0}>0$ such that for $0\leq p \leq m$, $i=1, 2$ and $R>R_{0}$, under the assumption (\ref{e.356}), we have
\begin{align}\begin{aligned}\label{e.200}
\dim \mathbb{F}^{(0, e^{-\frac{R\sqrt{\delta}}{16}}], (p)}_{i, R}=0, \quad h^{(p)}_{i, \infty}=h_{i}^{(p)}, \quad 0\leq p \leq m.
\end{aligned}\end{align}
Moreover
\begin{align}\begin{aligned}\label{e.201}
\dim \mathbb{F}^{(0,e^{-\frac{R\sqrt{\delta}}{16}}], (p)}_{R}=0, \quad 0\leq p \leq m.\end{aligned}\end{align}
\end{prop}

\begin{proof}
In fact, by Propositions \ref{p.4}, \ref{p.13}, \ref{p.14a}, we have proved that
\begin{align}\begin{aligned}\label{e.197}
\dim \mathbb{F}^{[0, e^{-\frac{R\sqrt{\delta}}{16}}], (p)}_{i, R}=h^{(p)}_{i, \infty},\quad i=1,2.
\end{aligned}\end{align}
By (\ref{e.203}) and (\ref{e.197}), we get (\ref{e.200}).
It means that $(D^{Z_{i, R}})^{2}$ don't have the non-trivial small eigenvalues decreasing exponentially when $R$ goes to infinity.
By Definition \ref{d.13}, (\ref{e.89}) and (\ref{e.200}), we get
$\dim  \mathbb{F}^{[0,e^{-\frac{R\sqrt{\delta}}{16}}], (p)}_{R}=h^{(p)}_{1, \infty}+h^{(p)}_{2, \infty}
=h^{(p)}_{1}+h^{(p)}_{2}
=h^{(p)}=\dim  \mathbb{F}^{\{0\},(p)}_{R}$,
this implies (\ref{e.201}). The proof is completed.
\end{proof}

Finally, \textbf{Theorem \ref{t.10}} follows from Propositions \ref{p.4}, \ref{p.13}, \label{p.14} and \ref{p.7}.

\begin{lemma}\label{l.22}For $i=1, 2$, there exist $R_{0}>0$ such that for $R>R_{0}$ we have the linear isomorphism
\begin{align}\begin{aligned}\label{e.205}
P^{\{0\}}_{R}:\mathscr{W}_{R}\cong \Ker (D^{Z_{R}})^{2}\quad (\text{resp.}\,\,
P^{\{0\}}_{i, R}:\mathscr{W}_{i, R}\cong \Ker (D^{Z_{i, R}})^{2}).
\end{aligned}\end{align}
\end{lemma}
\begin{proof}The injectivity is a consequence of Propositions \ref{p.4}. The surjectivity is a
 consequence of Propositions \ref{p.13}, \ref{p.14a} and \ref{p.7}.
\end{proof}

\subsection{Adiabatic limit of the large time contribution $L(R)$}\label{ss3.6}
In this subsection we will treat the large time contribution (\ref{e.93}). We have the following theorem.

\begin{thm}\label{t.8} Under the assumption {\rm(\ref{e.356})}, we have
$\lim_{R\rightarrow\infty}L(R)=0$.
\end{thm}
Then the rest part of this subsection will be contributed to prove this theorem.
By Definition \ref{d.9} and Theorem \ref{t.10}, there exists $c>0$ such that for $R$ sufficiently large and $b\in U$
\begin{align}\begin{aligned}\label{e.154}
\text{Spec}(\mathscr{F}_{R}^{(0)})\subset\{0\}\cup[c, +\infty), \quad \text{Spec}(\mathscr{F}_{i,R}^{(0)})\subset\{0\}\cup[c, +\infty).
\end{aligned}\end{align}
By (\ref{e.101}), we have
\begin{align}\begin{aligned}\label{e.155}
|\text{Spec}(C^{2}_{R,t})|\subset\{0\}\cup[\frac{ct}{4}, +\infty), \quad |\text{Spec}(C^{2}_{i,R,t})|\subset\{0\}\cup[\frac{ct}{4}, +\infty).
\end{aligned}\end{align}
Let $\delta$ be a circle centered at $0$ with radius $\frac{\sqrt{c}}{8}$ and $\Delta=\Delta^{+}\cup\Delta^{-}$ be the contour indicated by Figure \ref{fig.9}.

\begin{figure}
\input{Coutour.pstex_t}
\caption[Figure]{\label{fig.9}}
\end{figure}\index{Figure \ref{fig.9}}

For $i=1, 2$, let
\begin{align}\begin{aligned}
B_{R}=(d^{M_{R}})^{*}-d^{M_{R}}\quad (\text{resp.}\,\, B_{i,R}=(d^{M_{i,R}})^{*}-d^{M_{i,R}}\,\,).
\end{aligned}\end{align}
Then we put
\begin{align}\begin{aligned}\label{e.156}
&\mathbf{P}^{\{0\}}_{R, t}=\frac{1}{2i\pi}\psi^{-1}_{t}\int_{\delta}\frac{f'(\sqrt{t}\lambda)}{\lambda-B_{R}}d\lambda\cdot\psi_{t},
\quad \mathbf{K}_{R, t}=\frac{1}{2i\pi}\psi^{-1}_{t}\int_{\Delta}\frac{f'(\sqrt{t}\lambda)}{\lambda-B_{R}}d\lambda\cdot\psi_{t},\\
&\mathbf{P}^{\{0\}}_{i, R, t}=\frac{1}{2i\pi}\psi^{-1}_{t}\int_{\delta}\frac{f'(\sqrt{t}\lambda)}{\lambda-B_{i, R}}d\lambda\cdot\psi_{t},
\quad \mathbf{K}_{i, R, t}=\frac{1}{2i\pi}\psi^{-1}_{t}\int_{\Delta}\frac{f'(\sqrt{t}\lambda)}{\lambda-B_{i, R}}d\lambda\cdot\psi_{t},
\end{aligned}\end{align}\index{$\mathbf{P}^{\{0\}}_{R, t}$, $\mathbf{K}_{R, t}$, $\mathbf{P}^{\{0\}}_{i, R, t}$,\\ $\mathbf{K}_{i, R, t}$,
$P^{\{0\}}_{R, t}$, $P^{\{0\}}_{i, R, t}$}
then by (\ref{e.155}) and (\ref{e.156}), we have
\begin{align}\begin{aligned}\label{e.159}
f'(D_{R, t})=\mathbf{P}^{\{0\}}_{R, t}+\mathbf{K}_{R, t}, \quad f'(D_{i, R, t})=\mathbf{P}^{\{0\}}_{i, R, t}+\mathbf{K}_{i, R, t}.
\end{aligned}\end{align}

  Let $L\in \Lambda (T^{*}S)\otimes \text{End}(\Omega(Z_{R}, F_{R}))$, then for all $q\in \mathbb{N}$ we define
\begin{align}\begin{aligned}\label{e.160}
  \|L\|_{q}:=\{\tr(L^{*}L)^{q/2}\}^{1/q},
\end{aligned}\end{align}
 when $q=\infty$ we set $\|\cdot\|_{\infty}$ to be the operator norm, i.e., $\|A\|_{\infty}=\sup_{\|s\|_{L^{2}}=1}\|As\|_{L^{2}}$.\index{$\parallel\cdot\parallel_{q}$, $\parallel\cdot\parallel_{\infty}$}
Moreover, for $p\in \mathbb{N}$, if $\|A\|_{p}$ and $\|B\|_{\infty}$ exist, then we have a useful inequality
\begin{align}\begin{aligned}\label{e.766}
\|AB\|_{p}\leq \|A\|_{p}\|B\|_{\infty}.
\end{aligned}\end{align}

\begin{lemma}
\label{l.53} For $\lambda_{0}\in \Delta$ fixed, $i=1,2,$ and $p>\dim(Z_{R})$, there exist $C>0$ and $N_{0}\in\mathbb{N}^{*}$ such that
for any $R\geq 1$
\begin{align}\begin{aligned}\label{e.681}
\|(\lambda_{0}-B_{R})^{-1}\|_{p}\leq CR^{\frac{N_{0}}{p}}\quad (\text{resp.}\,\,\|(\lambda_{0}-B_{i,R})^{-1}\|_{p}\leq CR^{\frac{N_{0}}{p}}).
\end{aligned}\end{align}
\end{lemma}
\begin{proof} By (\ref{e.154}), for $\lambda_{0}\in \Delta$, we see that
$(\lambda_{0}-B_{R}^{(0)})^{*}(\lambda_{0}-B_{R}^{(0)})$ is a self-adjoint and strictly positive operator.
If we set $\widehat{H}_{R}:=(\lambda_{0}-B_{R}^{(0)})^{*}(\lambda_{0}-B_{R}^{(0)})$, then we have
\begin{align}\begin{aligned}\label{e.682}
\widehat{H}_{R}=(D^{Z_{R}})^{2}+2i\Im(\lambda_{0})B^{(0)}_{R}+|\lambda_{0}|^{2}.
\end{aligned}\end{align}
Hence $\widehat{H}_{R}$ is a self-adjoint positive generalized Laplacian. By \cite[Thm. 2.38]{BGV92}, for $k>1+\frac{\dim(Z_{R})+l}{2}$
, the operator $\widehat{H}^{-k}$
has a $C^{l}-$kernel given by
\begin{align}\begin{aligned}\label{e.683}
\widehat{H}_{R}^{-k}(x,x')=\frac{1}{(k-1)!}\int_{0}^{\infty}e^{-t\widehat{H}_{R}}(x,x')t^{k-1}dt.
\end{aligned}\end{align}
From (\ref{e.683}) and the proof of \cite[Thm. 2.38]{BGV92}, we see that there exists $C_{l}>0$ such that for $R\geq 1$ and $(x,x')\in Z_{R}\times Z_{R}$
\begin{align}\begin{aligned}\label{e.707}
\left|\widehat{H}_{R}^{-k}(x,x')\right|_{\mathscr{C}^{l}}\leq C_{l}R.
\end{aligned}\end{align}
(We remark that the factor $R$ at the right side of (\ref{e.707}) comes from the dependance of the volume of $Z_{R}$
on the length $R$ of its cylinder part in the proof of \cite[Prop. 2.37]{BGV92}, which was applied to prove \cite[Thm. 2.38]{BGV92} in our case. To prove a similar proposition as \cite[Prop. 2.37]{BGV92}, we have essentially used the fact that
 the spectral of $(D^{Z_{R}})^{2}$ has a uniform gap away from $0$ (see Theorem \ref{t.10}).) By (\ref{e.160}), (\ref{e.683}) and (\ref{e.707}), we get for $p>\dim(Z_{R})$
\begin{align}\begin{aligned}\label{e.684}
\|(\lambda_{0}-B^{(0)}_{R})^{-1}\|_{p}=&\left\{\tr[\widehat{H}_{R}^{-p/2}]\right \}^{\frac{1}{p}}\\
=&\left\{\int_{Z_{R}}\tr\big[\widehat{H}_{R}^{-p/2}(x,x)\big]dv_{Z_{R}}\right \}^{\frac{1}{p}}
\leq  CR^{\frac{2}{p}}.
\end{aligned}\end{align}
We have the expansion
\begin{align}\begin{aligned}\label{e.685}
(\lambda_{0}-B_{R})^{-1}=(\lambda_{0}-B^{(0)}_{R})^{-1}+(\lambda_{0}-B^{(0)}_{R})^{-1}B^{(\geq 1)}(\lambda_{0}-B^{(0)}_{R})^{-1}+\cdots,
\end{aligned}\end{align}
where the expansion only contains a finite number of terms and $B^{(\geq 1)}$ is an operator of order $0$. Since $\|B^{(\geq 1)}\|_{\infty}\leq \infty$,
by (\ref{e.684}) and (\ref{e.685}), we see that (\ref{e.681}) holds. For $i=1,2$, we follows the same proof. The proof is completed.
\end{proof}

\begin{lemma}\label{l.12}For $i=1, 2$, under assumption (\ref{e.356}) we have
\begin{align}\begin{aligned}\label{e.152}
\lim_{R\rightarrow\infty}\int_{R^{2-\varepsilon}}^{\infty}\varphi\tr_{s}\big[\frac{N}{2}\mathbf{K}_{R, t}\big]\frac{dt}{t}=0\quad
(\text{resp.}\,\,\lim_{R\rightarrow\infty}\int_{R^{2-\varepsilon}}^{\infty}\varphi\tr_{s}\big[\frac{N}{2}\mathbf{K}_{i, R, t}\big]\frac{dt}{t}=0).
\end{aligned}\end{align}
\end{lemma}
\begin{proof}
Take $p\in \mathbb{N}, p>m $. There exist a unique function $k_{p}(\lambda)$, holomorphic on $\mathbb{C}\backslash \mathbb{R}$, such that (cf.
\cite[Prop. 3.41]{BGo})

-As $\lambda\rightarrow \pm i\infty$, $k_{p}(\lambda)\rightarrow 0$.

-The following identity holds
\begin{align}\begin{aligned}\label{e.161}
\frac{k_{p}^{(p-1)}(\lambda)}{(p-1)!}=f'(\lambda).\end{aligned}\end{align}
Clearly, if $\lambda\in \Delta$,
\begin{align}\begin{aligned}\label{e.162}
|\text{Re}(\lambda)|\leq \frac{1}{2}|\text{Im}(\lambda)|.\end{aligned}\end{align}
Using (\ref{e.162}), we find that there exist $C>0$, $C'>0$ such that if $\lambda\in \Delta$,
\begin{align}\begin{aligned}\label{e.163}
|k_{p}(\sqrt{t}\lambda)|\leq C\exp(-C't|\lambda|^{2}).\end{aligned}\end{align}
Clearly,
\begin{align}\begin{aligned}\label{e.164}
\frac{1}{2i\pi}\int_{\Delta}\frac{f'(\sqrt{t}\lambda)}{\lambda-B_{R}}d\lambda=\frac{1}{2i\pi}\int_{\Delta}\frac{k_{p}(\sqrt{t}\lambda)}
{t^{\frac{p-1}{2}}(\lambda-B_{R})^{p}}d\lambda.
\end{aligned}\end{align}
If $\lambda\in \Delta$, we have the expansion
\begin{align}\begin{aligned}\label{e.165}
(\lambda-B_{R})^{-1}=(\lambda-B^{(0)}_{R})^{-1}+(\lambda-B^{(0)}_{R})^{-1}B^{(\geq 1)}(\lambda-B^{(0)}_{R})^{-1}+\cdots,
\end{aligned}\end{align}
where the expansion only contains a finite number of terms and $B^{(\geq 1)}$ is an operator of order $0$.
By (\ref{e.154}) and (\ref{e.165}), we find that there
exist $C>0$ such that if $\lambda\in\Delta$, for any $R\geq 1$
\begin{align}\begin{aligned}\label{e.166}
\|(\lambda-B_{R})^{-1}\|_{\infty}\leq C.
\end{aligned}\end{align}
Fix $\lambda_{0}\in \Delta$.
If $\lambda\in \Delta$, we have
\begin{align}\begin{aligned}\label{e.168}
(\lambda-B_{R})^{-1}=(\lambda_{0}-B_{R})^{-1}-(\lambda-\lambda_{0})(\lambda_{0}-B_{R})^{-1}(\lambda-B_{R})^{-1}.\end{aligned}\end{align}

By (\ref{e.681}), (\ref{e.166}), (\ref{e.766}) and (\ref{e.168}), we find that for $\lambda\in \Delta$
\begin{align}\begin{aligned}\label{e.169}
\|(\lambda-B_{R})^{-1}\|_{p}&\leq \|(\lambda_{0}-B_{R})^{-1}\|_{p}+|\lambda-\lambda_{0}|\cdot\|(\lambda_{0}-B_{R})^{-1}(\lambda-B_{R})^{-1}\|_{p}\\
&\leq \|(\lambda_{0}-B_{R})^{-1}\|_{p}-|\lambda-\lambda_{0}|\cdot\|(\lambda_{0}-B_{R})^{-1}\|_{p}\|(\lambda-B_{R})^{-1}\|_{\infty}\\
&\leq C(1+|\lambda-\lambda_{0}|)R^{\frac{N_{0}}{p}}\leq C'(1+|\lambda|)R^{\frac{N_{0}}{p}}.
\end{aligned}\end{align}
From (\ref{e.169}), we get
\begin{align}\begin{aligned}\label{e.170}
\|(\lambda-B_{R})^{-p}\|_{1}\leq (\|(\lambda-B_{R})^{-1}\|_{p})^{p}\leq C(1+|\lambda|)^{p}R^{N_{0}}.
\end{aligned}\end{align}
By (\ref{e.758}), (\ref{e.156}), (\ref{e.163}), (\ref{e.164}) and (\ref{e.170}), we get for $t\geq 1$ and $R$ sufficiently large
\begin{align}\begin{aligned}\label{e.171}
\left|\varphi\tr_{s}[\frac{N}{2}\mathbf{K}_{R, t}]\right|&\leq C(1+t^{-\frac{n}{2}})\cdot\Big\|\int_{\Delta}\frac{k_{p}(\sqrt{t}\lambda)}{\sqrt{t}^{p-1}(\lambda-B_{R})^{p}}d\lambda\Big\|_{1}\\
&\leq C_{1}(1+t^{-\frac{n}{2}})t^{-\frac{p-1}{2}}\int_{\Delta}\exp(-C't|\lambda|^{2})\big\|(\lambda-B_{R})^{-p}\big\|_{1}d\lambda\\
&\leq C_{2}R^{N_{0}}\exp(-C_{3}t).
\end{aligned}\end{align}
From (\ref{e.171}), we get
\begin{align}\begin{aligned}\label{e.172}
\Big|\int_{R^{2-\varepsilon}}^{\infty}\varphi\tr_{s}[\frac{N}{2}\mathbf{K}_{R, t}]\frac{dt}{t}\Big|&\leq
 CR^{N_{0}}\int_{R^{2-\varepsilon}}^{\infty}\exp(-C't)\frac{dt}{t}\\
&\leq C''R^{N_{0}}\exp(-C'R^{2-\varepsilon}),
\end{aligned}\end{align}
from which we get (\ref{e.152}). The proof of (\ref{e.152}) for $i=1,2,$ is the same. The proof is completed.
\end{proof}

\begin{defn}\label{l.13}
For $i=1, 2$, let $V_{R}=(d^{Z_{R}})^{*}-d^{Z_{R}}$ and $V_{i,R}=(d^{Z_{i,R}})^{*}-d^{Z_{i,R}}$. Let $P^{\{0\}}_{R, t}$ (resp. $P^{\{0\}}_{i, R, t}$) denote the $0-$degree component of $\mathbf{P}^{\{0\}}_{R, t}$
(resp. $\mathbf{P}^{\{0\}}_{i, R, t}$), which is a projection from $\pi^{*}(\Lambda(T^{*}S))\widehat{\otimes}\Omega(Z_{R}, F_{R})$
(resp. $\pi^{*}(\Lambda(T^{*}S))\widehat{\otimes}\Omega(Z_{i, R}, F_{R})$) to the kernel of $V_{R}$ (resp. $V_{i, R}$).
\end{defn}

\begin{lemma}\label{l.56}
For $i=1,2$, there exist $t_{0}>1$ and $C>0$, such that for any $t>t_{0}$, $R>0$ we have
\begin{align}\begin{aligned}\label{e.759}
&\Big|\varphi\tr_{s}\big[\frac{N}{2}\mathbf{P}^{\{0\}}_{R, t}\big]-\frac{1}{2}\chi'(Z, F)\Big|\leq \frac{C}{\sqrt{t}}\\
(\text{resp.}\,\,&\Big|\varphi\tr_{s}\big[\frac{N}{2}\mathbf{P}^{\{0\}}_{i, R, t}\big]-\frac{1}{2}\chi'_{\rm bd}(Z_{i}, F)\Big|\leq \frac{C}{\sqrt{t}}\,\,).
\end{aligned}\end{align}
\end{lemma}
\begin{proof} To prove this lemma, we will follow the method used in \cite[Thm. 3.42]{BGo} (cf. also \cite[Thm. 2.11]{BGo}). By (\ref{e.155}), (\ref{e.156}) and the fact that $D_{R,t}=\psi^{-1}_{t}\sqrt{t}B_{R}\psi_{t}$ (cf. \cite[Prop. 3.17]{BGo}), we get for $t>1$
\begin{align}\begin{aligned}\label{e.760}
\mathbf{P}^{\{0\}}_{R, t}&=\frac{1}{2i\pi}\psi^{-1}_{t}\int_{\frac{\delta}{\sqrt{t}}}\frac{f'(\sqrt{t}\lambda)}{\lambda-B_{R}}d\lambda\cdot\psi_{t}\\
&=\frac{1}{2i\pi}\psi^{-1}_{t}\int_{\delta}\frac{f'(\lambda)}{\lambda-\sqrt{t}B_{R}}d\lambda\cdot\psi_{t}
=\frac{1}{2i\pi}\int_{\delta}\frac{f'(\lambda)}{\lambda-D_{R,t}}d\lambda.
\end{aligned}\end{align}
Now we have the expansion
\begin{align}\begin{aligned}\label{e.761}
(\lambda-D_{R,t})^{-1}=(\lambda-\sqrt{t}B_{R}^{(0)}&)^{-1}\\
&+(\lambda-\sqrt{t}B_{R}^{(0)})^{-1}D_{R,t}^{(\geq 1)}(\lambda-\sqrt{t}B_{R}^{(0)})^{-1}+\cdots\, ,
\end{aligned}\end{align}
and the expansion in (\ref{e.761}) only contains a finite number of terms. By (\ref{e.155}), $0$ is the only element inside the domain bounded by $\delta$ which may lie in the spectrum of $B_{R}^{(0)}$. Using (\ref{e.760}), (\ref{e.761}) and the theorem of residues, we get for $t\geq 1$
\begin{align}\begin{aligned}\label{e.762}
\mathbf{P}^{\{0\}}_{R, t}=\sum_{p=0}^{\dim S}\sum_{\begin{subarray}{l}
1\leq i_{0}\leq p+1 \\j_{1},\cdots,j_{p+1-i_{0}}\geq 0\\
\sum_{m=1}^{p+1-i_{0}}j_{m}\leq i_{0}-1
 \end{subarray}}&
\frac{f^{(i_{0}-\sum_{m=0}^{p+1-i_{0}}j_{m})}(0)}{(i_{0}-1-\sum_{m=0}^{p+1-i_{0}}j_{m})!}
(-1)^{p+1-i_{0}}\\
&\qquad T_{R,1}D^{(\geq 1)}_{R,t}T_{R,2}\cdots D^{(\geq 1)}_{R,t}T_{R,p+1}.
\end{aligned}\end{align}
In (\ref{e.762}), $i_{0}$ of the $T_{R,j}$ are equal to $P^{\{0\}}_{R}$, and the other $T_{R,j}$ are equal respectively to
$(\sqrt{t}B^{(0)}_{R})^{-(1+j_{1})},\cdots,(\sqrt{t}B^{(0)}_{R})^{-(1+j_{p+1-i_{0}})}$. In particular, $\mathbf{P}^{\{0\}}_{R, t}$ is a polynomial in the variable $\frac{1}{\sqrt{t}}$, whose the constant term is given by
\begin{align}\begin{aligned}\label{e.767}
P_{R}^{\{0\}}f'(\frac{\omega(\nabla^{W_{R}}, h^{W_{R}})}{2})P_{R}^{\{0\}},
\end{aligned}\end{align}
when all $T_{R,j}$ are equal to $P^{\{0\}}_{R}$. We have (cf. \cite[Prop. 3.14]{BL})
\begin{align}\begin{aligned}\label{e.763}
\omega(\nabla^{H(Z_{R}, F_{R})}, h_{L^{2}}^{H(Z_{R}, F_{R})})=P_{R}^{\{0\}}\omega(\nabla^{W_{R}}, h^{W_{R}})P_{R}^{\{0\}}.
\end{aligned}\end{align}
We observe that in (\ref{e.762}), $i_{0}\geq 1$, so that $P^{\{0\}}_{R}$ appears at least once. Now $P^{\{0\}}_{R}$ is a projector on a finite dimensional vector bundle, and in particular it is trace class whose $\|\,\,\|_{1}$ norm is bounded uniformly with repect to $R$. We note that all the coefficients of $\frac{1}{\sqrt{t}}$ in $D^{(\geq 1)}_{R,t}$ are bounded operators in norm $\|\,\,\|_{\infty}$ uniformly with respect to $R>0$. And using the existence of the uniform spectral gap of $B^{(0)}_{R}$ with respect to $R>0$, there exists $C>0$ such that for all $R>0$
\begin{align}\begin{aligned}\label{e.768}
\|(B_{R}^{(0)})^{-(1+j_{m})}\|_{\infty}\leq C.
\end{aligned}\end{align}
Hence, by (\ref{e.766}), (\ref{e.762}), (\ref{e.767}) and (\ref{e.763}), we get
\begin{align}\begin{aligned}\label{e.764}
\left\|\mathbf{P}^{\{0\}}_{R, t}-f'(\frac{\omega(\nabla^{H(Z_{R}, F_{R})}, h_{L^{2}}^{H(Z_{R}, F_{R})})}{2})\right\|_{1}\leq \frac{C}{\sqrt{t}}.
\end{aligned}\end{align}
By \cite[Prop. 1.3]{BL}, we have (cf. also \cite[(2.55), (2.56)]{BL})
\begin{align}\begin{aligned}\label{e.765}
\varphi\tr_{s}\big[\frac{N}{2}f'(\frac{\omega(\nabla^{H(Z_{R}, F_{R})}, h_{L^{2}}^{H(Z_{R}, F_{R})})}{2})\big]=\frac{1}{2}\chi'(Z, F).
\end{aligned}\end{align}
Then (\ref{e.759}) follows from (\ref{e.764}) and (\ref{e.765}). In the same way, we get (\ref{e.759}) for $i=1,2$. The proof is completed.
\end{proof}

By Lemmas \ref{l.12}, \ref{l.56}, (\ref{e.89}), (\ref{e.93}) and (\ref{e.159}), we get Theorem \ref{t.8}.

\section{Contribution of the long exact sequence in the adiabatic limit}\label{s.1}
In Section \ref{ss3.8}, we compute the limit of the torsion form $T_{f}(A^{\mathscr{H}_{R}}, h_{L^{2}}^{\mathscr{H}_{R}})$
when $R\rightarrow \infty$.
In Section \ref{ss3.9}, we prove Lemma \ref{l.16}. In Section \ref{ss3.10}, we prove Lemma \ref{l.26}.

\subsection{Adiabatic limit of the torsion forms $T_{f}(A^{\mathscr{H}_{R}}, h_{L^{2}}^{\mathscr{H}_{R}})$}\label{ss3.8}
In this section we will treat the limit of the torsion forms $T_{f}(A^{\mathscr{H}_{R}}, h_{L^{2}}^{\mathscr{H}_{R}})$,
when $R\rightarrow \infty$.

\begin{thm}\label{t.9} Under the assumption {\rm(\ref{e.356})}, we have
$\lim_{R\rightarrow\infty}T_{f}(A^{\mathscr{H}_{R}}, h_{L^{2}}^{\mathscr{H}_{R}})=0$.
\end{thm}
Then \textbf{Theorem \ref{t.13}} follows from Theorems \ref{t.6}, \ref{t.8}, \ref{t.9}, (\ref{e.94}), (\ref{e.92}) and (\ref{e.93}).

In the rest part of this section we will prove this theorem.
Let
 \begin{align}\begin{aligned}\label{e.188}
m^{(p)}:=\rk C^{p}\left(\mathbb{K}_{Z}, F\right), \quad &m^{(p)}_{1}:=\rk C^{p}\left(\mathbb{K}_{Z_{1}}, F\right), \\
&m^{(p)}_{2}:=\rk C^{p}\left(\mathbb{K}_{Z_{2}}/\mathbb{K}_{Y}, F\right),
 \end{aligned}\end{align}\index{$m^{(p)}$, $m^{(p)}_{1}$, $m^{(p)}_{2}$}
then we have $m_{p}=m_{1, p}+m_{2, p}$.

\begin{lemma}\label{l.15}
The short exact sequence {\rm(\ref{e.86})} splits canonically such that
\begin{align}\begin{aligned}\label{e.179}
H^{p}(Z_{R}, F_{R})\cong H^{p}(Z_{1, R}, F_{R})\oplus H^{p}(Z_{2, R}, Y, F_{R}),
\end{aligned}\end{align}
\begin{align}\begin{aligned}\label{e.180}
\nabla^{H^{p}(Z_{R}, F_{R})}= \nabla^{H^{p}(Z_{1, R}, F_{R})}\oplus \nabla^{H^{p}(Z_{2, R}, Y, F_{R})}.\end{aligned}\end{align}
\end{lemma}
\begin{proof}
Since we have the following long exact sequence
\begin{align}\begin{aligned}\label{e.720}
\cdots\rightarrow H^{p}(Z_{1,R},Y,F_{1,R})\overset{k_{p}^{*}}\rightarrow H^{p}(Z_{1,R},F_{1,R})\rightarrow H^{p}(Y,F)\rightarrow\cdots,
\end{aligned}\end{align}
by our assumption (\ref{e.356}) we get an isomorphism $k_{p}^{*}:H^{p}(Z_{1,R},Y,F_{1,R})\cong H^{p}(Z_{1,R},F_{1,R})$. Using the following commutative diagram
\begin{align}\begin{aligned}\label{e.181}
\xymatrix{
0\ar[r]&H^{p}(Z_{2, R}, Y, F_{R})\ar[r]^{j^{*}_{p}}&H^{p}(Z_{R}, F_{R})\ar[r]^{i^{*}_{p}}&H^{p}(Z_{1, R}, F_{R})\ar[r]&0, \\
&&&H^{p}(Z_{1, R}, Y, F_{R})\ar[u]_{\wr}^{k^{*}_{p}}\ar[lu]_{l^{*}_{p}}&
}
\end{aligned}\end{align}
we choose a special inverse of $i^{*}_{p}$, denoted by $(i^{*}_{p})^{-1}$, to embed $H^{p}(Z_{1, R}, F_{R})$ into $H^{p}(Z_{R}, F_{R})$ as a subbundle, such that
\begin{align}\begin{aligned}\label{e.182}
(i^{*}_{p})^{-1}=l^{*}_{p}\circ (k^{*}_{p})^{-1}.\end{aligned}\end{align}
On other hand, we embed $H^{p}(Z_{2, R}, Y, F_{R})$ into $H^{p}(Z_{R}, F_{R})$ trivially by $j^{*}_{p}$.
We will show that (\ref{e.180}) is true under the isomorphisms selected as above. These
canonical connections induced by $\nabla^{F}$ on
these vector bundles of fiberwise cohomology group in diagram (\ref{e.181}) satisfy
\begin{align}\begin{aligned}\label{e.183}
j^{*}_{p}\circ\nabla^{H^{p}(Z_{2, R}, Y, F_{R})}&=\nabla^{H^{p}(Z_{R}, F_{R})}\circ j^{*}_{p}, \quad
i^{*}_{p}\circ\nabla^{H^{p}(Z_{R}, F_{R})} =\nabla^{H^{p}(Z_{1, R}, F_{R})}\circ i^{*}_{p}, \\
k^{*}_{p}\circ\nabla^{H^{p}(Z_{1, R}, Y, F_{R})}&=\nabla^{H^{p}(Z_{1, R}, F_{R})}\circ k^{*}_{p},\quad
l^{*}_{p}\circ\nabla^{H^{p}(Z_{1, R}, Y, F_{R})}=\nabla^{H^{p}(Z_{R}, F_{R})}\circ l^{*}_{p}.
\end{aligned}\end{align}
By (\ref{e.182}) and (\ref{e.183}), we get
\begin{align}\begin{aligned}\label{e.184}
\nabla^{H^{p}(Z_{R}, F_{R})}\circ (i^{*}_{p})^{-1} =(i^{*}_{p})^{-1}\circ \nabla^{H^{p}(Z_{1, R}, F_{R})}.\end{aligned}\end{align}
Then (\ref{e.180}) follows from the first identity of (\ref{e.183}) and (\ref{e.184}).\end{proof}
Under the identification of (\ref{e.179}),  $h^{H^{p}(Z_{R}, F_{R})}_{L^{2}}$
 and $h^{H^{p}(Z_{1, R}, F_{R})}_{L^{2}}\oplus h^{H^{p}(Z_{2, R}, Y, F_{R})}_{L^{2}}$ are two $L^{2}-$metrics on $H^{p}(Z_{R}, F_{R})$.
Next, we will show that these two Hermitian metrics differ by a term decreasing exponentially when $R$ goes to infinity.
We formulate it as the following lemma.
\begin{lemma}\label{l.16}
There exist $c>0$ such that for $R\geq 1$
 \begin{align}\begin{aligned}\label{e.185}
 h^{H^{p}(Z_{R}, F_{R})}_{L^{2}}=\left(h^{H^{p}(Z_{1, R}, F_{R})}_{L^{2}}\oplus h^{H^{p}(Z_{2, R}, Y, F_{R})}_{L^{2}}\right)\cdot(1+\mO(e^{-cR})).
\end{aligned}\end{align}
 \end{lemma}

 By (\ref{e.180}) and (\ref{e.185}), we get (cf. \cite[Appendix I]{BL})
 \begin{align}\begin{aligned}\label{e.186}
 \lim_{R\rightarrow \infty}T_{f}(A^{\mathscr{H}^{p}_{R}}, h_{L^{2}}^{\mathscr{H}^{p}_{R}})=0, \quad 0\leq p\leq m, \end{aligned}\end{align}
then \textbf{Theorem \ref{t.9}} follows from (\ref{e.87}) and (\ref{e.186}). Next subsection will be contributed to prove Lemma \ref{l.16}.

\subsection{Proof of Lemma \ref{l.16}}\label{ss3.9}

Let $\{\mathfrak{a}_{i}|1\leq i\leq m^{(p)}_{1}\}$ (resp. $\{\mathfrak{b}_{j}|1\leq j\leq m^{(p)}_{2}\}$) be a local frame of the bundle of chain group $C_{p}\left(\mathbb{K}_{Z_{1}}, F^{*}\right)$ (resp. $C_{p}\left(\mathbb{K}_{Z_{2}}/\mathbb{K}_{Y}, F^{*}\right)$). Recall that the diffeomorphism $\phi_{R}:M\rightarrow M_{R}$
has been constructed in Lemma \ref{l.1}.
We set  $\mathfrak{a}_{R, i}=\phi_{R}(\mathfrak{a}_{i})$ (resp. $\mathfrak{b}_{R, j}=\phi_{R}(\mathfrak{b}_{j})$)
 which constitute a local frame of $C_{p}\left(\mathbb{K}_{Z_{1, R}}, F^{*}_{1, R}\right)$ (resp. $C_{p}\left(\mathbb{K}_{Z_{2, R}}/\mathbb{K}_{Y}, F^{*}_{2, R}\right)$).
 Let $\{\mathfrak{a}^{i}_{R}\}$ (resp. $\{\mathfrak{b}_{R}^{j}\}$) be a local frame of $C^{p}\left(\mathbb{K}_{Z_{1, R}}, F_{R}\right)$ (resp. $C^{p}\left(\mathbb{K}_{Z_{2, R}}/\mathbb{K}_{Y}, F_{R}\right)$)
such that
\begin{align}\begin{aligned}\label{e.187}
 \langle\mathfrak{a}_{R}^{i_{1}}, \mathfrak{a}_{R, i_{2}}\rangle=\delta^{i_{1}}_{i_{2}}\quad
(\text{resp.}\,\,\langle\mathfrak{b}_{R}^{j_{1}}, \mathfrak{b}_{R, j_{2}}\rangle=\delta^{j_{1}}_{j_{2}}),
\end{aligned}\end{align}
where $\langle\cdot, \cdot\rangle$ denotes the paring between the cochain group and chain group.

\begin{rem}\label{r.4}
In all the rest of this subsection, we use $\{\mathfrak{a}, a, \alpha, \widehat{\alpha}, \widetilde{\alpha}\}$ (resp. $\{\mathfrak{b}, b, \beta, \widehat{\beta}, \widetilde{\beta}\}$) with low or upper indices
 to denote the various objects related to $Z_{1, R}$ (resp. $Z_{2, R}$).
\end{rem}

For any $R\geq 0$, let $\{\alpha_{i, R}\in \mathscr{H}^{p}\left(Z_{1, R}, F_{R}\right)|1\leq i \leq h^{(p)}_{1}\}$ be an orthonormal
frame of $\mathscr{H}^{p}\left(Z_{1, R}, F_{R}\right)$ and
$\{\beta_{j, R}\in \mathscr{H}^{p}\left(Z_{2, R}, Y, F_{R}\right)|1\leq j \leq h_{2}^{(p)}\}$ be an orthonormal frame of
$\mathscr{H}^{p}\left(Z_{2, R}, Y, F_{R}\right)$, so we have
\begin{align}\begin{aligned}\label{e.223}
 \langle\alpha_{i, R}, \alpha_{i', R} \rangle_{L^{2}(Z_{1, R})}=\delta_{ii'}, \quad\langle\beta_{j, R}, \beta_{j', R} \rangle_{L^{2}(Z_{2, R})}=\delta_{jj'}.\end{aligned}\end{align}
 By (\ref{e.177}) and Definition \ref{d.14}, there exist
\begin{align}\begin{aligned}\label{e.220}
&\{a^{i}_{R}\in \Ker\widetilde{\partial}\cap C^{p}\left(\mathbb{K}_{Z_{1, R}}, F_{R}\right)|1\leq i \leq h_{1}^{(p)}\}, \\
&\{b^{j}_{R}\in \Ker\widetilde{\partial}\cap C^{p}\left(\mathbb{K}_{Z_{2, R}}/\mathbb{K}_{Y}, F_{R}\right)|1\leq j \leq h_{2}^{(p)}\},
 \end{aligned}\end{align}
given by
\begin{align}\begin{aligned}\label{e.222}
a^{i}_{R}=\sum_{i}\big(\int_{\mathfrak{a}_{R, i}}\alpha_{i, R}\big)\cdot\mathfrak{a}_{R}^{i},\quad
 b^{j}_{R}=\sum_{j}\big(\int_{\mathfrak{b}_{R, j}}\beta_{j, R}\big)\cdot\mathfrak{b}_{R}^{j},
\end{aligned}\end{align}
such that
\begin{align}\begin{aligned}\label{e.207}
P^{\infty}_{1, R}(\alpha_{i, R})= [a^{i}_{R}], \quad P^{\infty}_{2, R}(\beta_{j, R})=[b^{j}_{R}].
\end{aligned}\end{align}
\index{$[a^{i}_{R}],\,[b^{j}_{R}]$}
Their cohomology classes $\{[a^{i}_{R}]|1\leq i \leq h_{1}^{(p)}\}$, $\{[b^{j}_{R}]|1\leq j \leq h_{2}^{(p)}\}$ constitute frames of
 $H^{p}(Z_{1, R}, F_{R})$, $H^{p}(Z_{2, R}, Y, F_{R})$ respectively.

\begin{lemma}\label{l.23}When $R\rightarrow\infty$, the coefficients appearing at the right of {\rm(\ref{e.222})} increase with an order of $\mO(R)$,
 i.e.,
\begin{align}\begin{aligned}\label{e.224}
\int_{\mathfrak{a}_{R, i'}}\alpha_{i, R}=\mO(R),  \quad \int_{\mathfrak{b}_{R, j'}}\beta_{j, R}=\mO(R).
\end{aligned}\end{align}
\end{lemma}
\begin{proof}
By (\ref{e.223}), we have that $\|\alpha_{i, R}\|_{L^{2}(Z_{1, R})}=\|\beta_{j, R}\|_{L^{2}(Z_{2, R})}=1$ and
\begin{align}\begin{aligned}\label{e.445}
\big(D^{Z_{1, R}}\big)^{2k}\alpha_{i, R}=0, \quad \big(D^{Z_{2, R}}\big)^{2k}\beta_{j, R}=0, \text{  for any } k\in \mathbb{N}.
\end{aligned}\end{align}
For $k>m+l$, by Sobolev inequality and elliptic estimates we get that there exist $C_{l}>0$ such that for any $R\geq 1$, $x\in Z_{1,R}$ and $x'\in Z_{2,R}$
\begin{align}\begin{aligned}\label{e.687}
\left\|\alpha_{i, R}(x) \right\|_{\mathscr{C}^{l}}\leq C_{l},\quad \left\|\beta_{j, R}(x') \right\|_{\mathscr{C}^{l}}\leq C_{l}.
\end{aligned}\end{align}
By  Remark \ref{r.1} and (\ref{e.687}), we get
\begin{align}\begin{aligned}\label{e.688}
\big|\int_{\mathfrak{a}_{R, i'}}\alpha_{i, R}\big|\leq C_{l}\cdot{\rm Vol}(\mathfrak{a}_{R, i'})\leq CR,
\end{aligned}\end{align}
so we have proved the first estimate in (\ref{e.224}). In the same way, we get the second estimate.
\end{proof}
By the definition of the $L^{2}-$metric on $H^{p}(Z_{1, R}, F_{R})$ and $H^{p}(Z_{2, R}, Y, F_{R})$, we have by (\ref{e.207})
\begin{align}\begin{aligned}\label{e.210}
 &\langle[a^{i}_{R}], [a^{i'}_{R}]\rangle_{ h^{H^{p}(Z_{1, R}, F_{R})}_{L^{2}}}=\langle\alpha_{i, R}, \alpha_{i', R} \rangle_{L^{2}(Z_{1, R})}=\delta_{ii'}, \\
 &\langle[b^{j}_{R}], [b^{j'}_{R}]\rangle_{ h^{H^{p}(Z_{2, R}, Y, F_{R})}_{L^{2}}}=\langle\beta_{j, R}, \beta_{j', R} \rangle_{L^{2}(Z_{2, R})}=\delta_{jj'}.
\end{aligned}\end{align}

By the identification (\ref{e.179}), we obtain a frame of $H^{p}(Z_{R}, F_{R})$, that's
\begin{align}\begin{aligned}\label{e.208}
\{(i^{*}_{p})^{-1}[a^{k}_{R}], j^{*}_{p}[b^{l}_{R}]\big|1\leq k \leq h_{1}^{(p)}, 1\leq l \leq h_{2}^{(p)}\}.
\end{aligned}\end{align}
By (\ref{e.177}), there exist $\{\widetilde{\alpha}_{k, R}, \widetilde{\beta}_{l, R}\in \mathscr{H}^{p}\left(Z_{R}, F_{R}\right)|1\leq k \leq h_{1}^{(p)}, 1\leq l \leq h_{2}^{(p)}\}$,
a frame of $\mathscr{H}^{p}\left(Z_{R}, F_{R}\right)$, such that\index{$\widetilde{\alpha}_{k, R},\, \widetilde{\beta}_{l, R}$}
\begin{align}\begin{aligned}\label{e.209}
P^{\infty}_{R}(\widetilde{\alpha}_{k, R})= (i^{*}_{p})^{-1}[a^{k}_{R}], \quad P^{\infty}_{R}(\widetilde{\beta}_{l, R})=j^{*}_{p}[b^{l}_{R}].
\end{aligned}\end{align}

By Lemma \ref{l.22}, there exist $s_{1, i}\in \Ker_{L^{2}}(D^{Z_{1, \infty}})^{2, (p)}$, $1\leq i\leq h_{1}^{(p)}$, and $s_{2, j}\in \Ker_{L^{2}}(D^{Z_{2, \infty}})^{2, (p)}$, $1\leq j\leq h_{2}^{(p)}$, such that
\begin{align}\begin{aligned}\label{e.212}
P^{\{0\}}_{1, R}\big(f_{1,R} s_{1, i}\big)=\alpha_{i, R}, \quad P^{\{0\}}_{2, R}\big(f_{2,R} s_{2, j}\big)=\beta_{j, R}.
\end{aligned}\end{align}

\begin{lemma}\label{l.27} For $1\leq i\leq h_{1}^{(p)}, 1\leq j\leq h_{2}^{(p)}$, $\nu=1, 2$ and $R$ large enough, we have
\begin{align}\begin{aligned}\label{e.228}
\|f_{\nu, R} s_{\nu, i}\|_{L^{2}(Z_{\nu, R})}\leq 2.
\end{aligned}\end{align}
\end{lemma}

\begin{proof}
From Lemma \ref{l.20}, (\ref{e.223}) and (\ref{e.212}) we get
\begin{align}\begin{aligned}\label{e.230}
\|f_{\nu, R} s_{\nu, k}\|_{L^{2}(Z_{\nu, R})}&\leq\left\|P_{\nu, R}^{\{0\}}f_{\nu, R} s_{\nu, k}\right\|_{L^{2}(Z_{\nu, R})}+\left\|\left({\rm Id}-P_{\nu, R}^{\{0\}}\right)f_{\nu, R} s_{\nu, k}\right\|_{L^{2}(Z_{\nu, R})}\\
&\leq1+Ce^{-cR} \cdot \|f_{\nu, R}s_{\nu, k}\|_{L^{2}(Z_{\nu, R})}.
\end{aligned}\end{align}
Take $R$ large enough, we get (\ref{e.228}) from (\ref{e.230}).
\end{proof}

\begin{defn}\label{d.15}
We define, for $1\leq k\leq h_{1}^{(p)}, 1\leq l\leq h_{2}^{(p)}$,
\begin{align}\begin{aligned}\label{e.213}
\widehat{\alpha}_{k, R}=P^{\{0\}}_{R}\big(f_{1,R} s_{1, k}\big), \quad \widehat{\beta}_{l, R}=P^{\{0\}}_{R}\big(f_{2,R} s_{2, l}\big),
\end{aligned}\end{align}\index{$\widehat{\alpha}_{k, R}$, $\widehat{\beta}_{l, R}$}
and by (\ref{e.205}) of Lemma \ref{l.22}, they constitute a frame of $\mathscr{H}^{p}\left(Z_{R}, F_{R}\right)$ for $R$ large enough.
\end{defn}

Now we expand the frame $\{\widetilde{\alpha}_{k, R}, \widetilde{\beta}_{l, R}\}$ in
term of the frame $\{\widehat{\alpha}_{k, R}, \widehat{\beta}_{l, R}\}$, and denote the matrix of coefficients by
${\rm \Theta}_{R}=\left(
                \begin{array}{cc}
                  {\rm \Theta}^{1}_{R} & {\rm \Theta}^{2}_{R} \\
                  {\rm \Theta}^{3}_{R} & {\rm \Theta}^{4}_{R} \\
                \end{array}
              \right)
$
such that
\begin{align}\begin{aligned}\label{e.241}
\left(
 \begin{array}{c}
  \big(\widetilde{\alpha}_{k, R}\big)_{h^{(p)}_{1}\times 1} \\
  \big(\widetilde{\beta}_{l, R}\big)_{h^{(p)}_{2}\times 1} \\
 \end{array}
\right)=
 \left(
                \begin{array}{cc}
                  {\rm \Theta}^{1}_{R} & {\rm \Theta}^{2}_{R} \\
                  {\rm \Theta}^{3}_{R} & {\rm \Theta}^{4}_{R} \\
                \end{array}
              \right)
\cdot
 \left(
 \begin{array}{c}
  \big(\widehat{\alpha}_{k', R}\big)_{h^{(p)}_{1}\times 1}\\
  \big(\widehat{\beta}_{l', R}\big)_{h^{(p)}_{2}\times 1} \\
 \end{array}
\right).
\end{aligned}\end{align}

We use
$
{\rm H}_{R}=\left(
          \begin{array}{cc}
            {\rm H}^{1}_{R} & {\rm H}^{2}_{R} \\
            {\rm H}^{3}_{R} & {\rm H}^{4}_{R} \\
          \end{array}
        \right)
$\index{${\rm H}_{R}$}
to represent the matrix form of the $L^{2}-$metric $h^{H^{p}(Z_{R}, F_{R})}_{L^{2}}$ with
respect to the basis $\{(i^{*}_{p})^{-1}[a^{i}_{R}],\,j^{*}_{p}[b^{l}_{R}]|1\leq i\leq h^{(p)}_{1},\,1\leq l\leq h^{(p)}_{2}\}$,
such that for $\{1\leq k, k'\leq h_{1}^{(p)}, 1\leq l, l' \leq h_{2}^{(p)}\}$ we have
\begin{align}\begin{aligned}\label{e.211}
 ({\rm H}^{1}_{R})_{kk'}&:=\langle (i^{*}_{p})^{-1}[a^{k}_{R}], (i^{*}_{p})^{-1}[a^{k'}_{R}]\rangle_{ h^{H^{p}(Z_{R}, F_{R})}_{L^{2}}}=\langle\widetilde{\alpha}_{k, R}, \widetilde{\alpha}_{k', R} \rangle_{L^{2}(Z_{R})}, \\
 ({\rm H}^{2}_{R})_{kl'}&:=\langle (i^{*}_{p})^{-1}[a^{k}_{R}], j^{*}_{p}[b^{l'}_{R}]\rangle_{ h^{H^{p}(Z_{R}, F_{R})}_{L^{2}}}=\langle\widetilde{\alpha}_{k, R}, \widetilde{\beta}_{l', R} \rangle_{L^{2}(Z_{R})}, \\
 ({\rm H}^{3}_{R})_{lk'}&:=\langle j^{*}_{p}[b^{l}_{R}], (i^{*}_{p})^{-1}[a^{k'}_{R}]\rangle_{ h^{H^{p}(Z_{R}, F_{R})}_{L^{2}}}=\langle \widetilde{\beta}_{l, R}, \widetilde{\alpha}_{k', R}\rangle_{L^{2}(Z_{R})}, \\
 ({\rm H}^{4}_{R})_{ll'}&:=\langle j^{*}_{p}[b^{l}_{R}], j^{*}_{p}[b^{l'}_{R}]\rangle_{ h^{H^{p}(Z_{R}, F_{R})}_{L^{2}}}= \langle\widetilde{\beta}_{l, R}, \widetilde{\beta}_{l', R} \rangle_{L^{2}(Z_{R})}.
\end{aligned}\end{align}

Set
$
{\rm G}_{R}=\left(
          \begin{array}{cc}
            {\rm G}^{1}_{R} & {\rm G}^{2}_{R} \\
            {\rm G}^{3}_{R} & {\rm G}^{4}_{R} \\
          \end{array}
        \right)
$\index{${\rm G}_{R}$}, such that for $\{1\leq k, k'\leq h_{1}^{(p)}, 1\leq l, l' \leq h_{2}^{(p)}\}$
\begin{align}\begin{aligned}\label{e.217}
 ({\rm G}^{1}_{R})_{kk'}&:=\langle\widehat{\alpha}_{k, R}, \widehat{\alpha}_{k', R} \rangle_{L^{2}(Z_{R})}, \quad
 &\,\,\,({\rm G}^{2}_{R})_{kl'}:=\langle\widehat{\alpha}_{k, R}, \widehat{\beta}_{l', R} \rangle_{L^{2}(Z_{R})}, \\
 ({\rm G}^{3}_{R})_{lk'}&:=\langle \widehat{\beta}_{l, R}, \widehat{\alpha}_{k', R}\rangle_{L^{2}(Z_{R})}, \quad
 &({\rm G}^{4}_{R})_{ll'}:= \langle\widehat{\beta}_{l, R}, \widehat{\beta}_{l', R} \rangle_{L^{2}(Z_{R})}.
\end{aligned}\end{align}
By (\ref{e.179}), (\ref{e.210}) and (\ref{e.211}), the equation (\ref{e.185}) is equivalent to prove that there exists
$c>0$ such that for $R\rightarrow \infty$
\begin{align}\begin{aligned}\label{e.218}
{\rm H}_{R}={\rm I}_{h^{(p)}\times h^{(p)}}+\mO(e^{-cR}).
\end{aligned}\end{align}
By (\ref{e.241}), (\ref{e.211}) and (\ref{e.217}), we get
\begin{align}\begin{aligned}\label{e.219}
 {\rm H}_{R}={\rm \Theta}_{R}{\rm G}_{R}{\rm \Theta}_{R}^{*},
\end{aligned}\end{align}
so in order to prove (\ref{e.218}) we need to study the asymptotic behaviors of ${\rm G}_{R}$ and ${\rm \Theta}_{R}$ when $R$ goes to infinity.
The following lemma will help us to establish the asymptotic estimate of ${\rm G}_{R}$ for $R\rightarrow \infty$.

\begin{lemma}\label{l.24} There exist $C>0$, $c>0$ and $R_{0}>0$ such that for any $R>R_{0}$, $b\in U$, $1\leq k\leq h^{(p)}_{1}$,
$1\leq l\leq h^{(p)}_{2}$,
\begin{align}\begin{aligned}\label{e.225}
&|\widehat{\alpha}_{k, R}(x)- \alpha_{k, R}(x)|_{\mathscr{C}^{0}}\leq Ce^{-cR}, \quad&\text{for any }\,x\in Z_{1,R} \\
 &|\widehat{\beta}_{l, R}(x)- \beta_{l, R}(x)|_{\mathscr{C}^{0}}\leq Ce^{-cR}, \quad&\text{for any }\,x\in Z_{2,R}.
\end{aligned}\end{align}
Moreover,
\begin{align}\begin{aligned}\label{e.740}
&|\widehat{\alpha}_{k, R}(x)|_{\mathscr{C}^{0}}\leq Ce^{-cR}, \quad&\text{for any }\,x\in Z_{2,R} \\
 &|\widehat{\beta}_{l, R}(x)|_{\mathscr{C}^{0}}\leq Ce^{-cR}, \quad&\text{for any }\,x\in Z_{1,R}.
\end{aligned}\end{align}
\end{lemma}
\begin{proof} We set the $\alpha_{i, R}(x)=0$ for $x\in Z_{2,R}$ and $\beta_{j,R}(x)=0$ for $x\in Z_{1,R}$.
By (\ref{e.212}) and (\ref{e.213}), we have for $x\in Z_{R}$ that
\begin{align}\begin{aligned}\label{e.231}
\big(\widehat{\alpha}_{i, R}&-\alpha_{i, R}\big)(x)=\big(P_{R}^{\{0\}}-e^{-R(D^{Z_{R}})^{2}}\big)f_{1,R} s_{1, k}(x)\\
&+\big(e^{-R(D^{Z_{R}})^{2}}-e^{-R(D^{Z_{1, R}})^{2}}\big)f_{1,R} s_{1, k}(x)
+\big(e^{-R(D^{Z_{1, R}})^{2}}-P_{1, R}^{\{0\}}\big)f_{1,R} s_{1, k}(x).
\end{aligned}\end{align}
 For all $x, y\in Z_{R}$ we have
\begin{align}\begin{aligned}\label{e.232}
\big|(P_{R}^{\{0\}}-e^{-R(D^{Z_{R}})^{2}})(x, y)\big|_{\mathscr{C}^{0}}\leq Ce^{-cR}.
\end{aligned}\end{align}
By Lemma \ref{l.27} and (\ref{e.232}), we get for all $x\in Z_{R}$ that
\begin{align}\begin{aligned}\label{e.233}
\big|\big(P_{R}^{\{0\}}-e^{-R(D^{Z_{R}})^{2}}\big)&f_{1,R} s_{1, k}(x)\big|_{\mathscr{C}^{0}}
\leq Ce^{-cR}\int_{Z_{1, R}}\big|(f_{1,R} s_{1, k})(y) \big|dv_{Z_{b}}(y)\\
&\leq Ce^{-cR}\text{Vol}(Z_{1, R})^{\frac{1}{2}}\|f_{1,R} s_{1, k}\|_{L^{2}(Z_{R})}
\leq C'{e}^{-c'R}.
\end{aligned}\end{align}
In the same way, we get
\begin{align}\begin{aligned}\label{e.234}
\big|\big(e^{-R(D^{Z_{1, R}})^{2}}-P_{1, R}^{\{0\}}\big)f_{1,R}s_{1, k}(x)\big|_{\mathscr{C}^{0}}\leq Ce^{-cR}.
\end{aligned}\end{align}
We treat the second term at the right of (\ref{e.231}) by two cases. First, for $t>0$, $x\in Z_{1}\cup Y_{[-R, -\frac{R}{8}]}$ and
\begin{align}\begin{aligned}\label{e.237}
y\in \text{supp}(f_{1,R} s_{1, k})\subset Z_{1}\cup Y_{[-R, -\frac{R}{4}]},
\end{aligned}\end{align}
 by method of comparing the heat kernel (see Lemma \ref{l.43}) we have
\begin{align}\begin{aligned}\label{e.235}
\big|\big(e^{-t(D^{Z_{R}})^{2}}-e^{-t(D^{Z_{1, R}})^{2}}\big)(x, y)\big|_{\mathscr{C}^{0}}\leq Ce^{-cR^{2}/t}.
\end{aligned}\end{align}
Similar to (\ref{e.233}), we get by Lemma \ref{l.27} and (\ref{e.235}) that for all $x\in Z_{1}\cup Y_{[-R, -\frac{R}{8}]}$
\begin{align}\begin{aligned}\label{e.235}
&\left|\left(\big(e^{-t(D^{Z_{R}})^{2}}-e^{-t(D^{Z_{1, R}})^{2}}\big)f_{1,R} s_{1, k}\right)(x)\right|\\
&\quad\quad\quad\quad\quad\quad\leq Ce^{-cR^{2}/t}\text{Vol}(Z_{1, R})^{\frac{1}{2}}\|f_{1,R}s_{1, k}\|_{L^{2}(Z_{1, R})}
\leq C'R^{\frac{1}{2}}e^{-cR^{2}/t}.
\end{aligned}\end{align}
By taking $t=R$, we find for all $x\in Z_{1}\cup Y_{[-R, -\frac{R}{8}]}$
\begin{align}\begin{aligned}\label{e.236}
\left|\left(\big(e^{-R(D^{Z_{R}})^{2}}-e^{-R(D^{Z_{1, R}})^{2}}\big)f_{1,R} s_{1, k}\right)(x)\right|\leq CR^{\frac{1}{2}}e^{-cR}\leq C'{e}^{-c'R}.
\end{aligned}\end{align}
Second, for $t>0$ and $x\in  Y_{[-\frac{R}{8}, R]}\cup Z_{2}$,
by off-diagonal estimate of heat kernel (see Lemma \ref{l.42}), Lemma \ref{l.27} and (\ref{e.237}), we get for all $t>0$ and
$x\in  Y_{[-\frac{R}{8}, R]}\cup Z_{2}$
\begin{align}\begin{aligned}\label{e.238}
&\left|\left(\big(e^{-t(D^{Z_{R}})^{2}}-e^{-t(D^{Z_{1, R}})^{2}}\big)f_{1,R} s_{1, k}\right)(x)\right|\\
&\leq\int_{Z_{1, R}}\big|e^{-t(D^{Z_{R}})^{2}}(x, x')\big|_{\mathscr{C}^{0}}\big|(f_{1,R}s_{1, k})(x')\big|dv_{Z_{b}}(x')\\
&\qquad\qquad\qquad+\int_{Z_{1, R}}\big|e^{-t(D^{Z_{1, R}})^{2}}(x, x')\big|_{\mathscr{C}^{0}}\big|(f_{1,R} s_{1, k})(x')\big|dv_{Z_{b}}(x')\\
&\leq\int_{Z_{1, R}}Ce^{-cd^{2}(x, x')/t}\big|(f_{1,R}s_{1, k})(x')\big|dv_{Z_{b}}(x')\\
&\qquad\qquad\qquad+\int_{Z_{1, R}}Ce^{-cd^{2}(x, x')/t}\big|(f_{1,R} s_{1, k})(x')\big|dv_{Z_{b}}(x')\\
&\leq 2Ce^{-cR^{2}/t}\text{Vol}(Z_{1, R})^{\frac{1}{2}}\|f_{1,R} s_{1, k}\|_{L^{2}(Z_{1, R})}\leq C'R^{\frac{1}{2}}e^{-cR^{2}/t},
\end{aligned}\end{align}
where $d(x, x')$ denotes the distance function. Let $t=R$, we get for all $x\in Z_{2}\cup Y_{[-\frac{R}{8}, R]}$
(\ref{e.236}) still holds.
Thus we get for all $x\in Z_{R}$ that
\begin{align}\begin{aligned}\label{e.240}
\left|\left(\big(e^{-R(D^{Z_{R}})^{2}}-e^{-R(D^{Z_{1, R}})^{2}}\big)f_{1,R} s_{1, k}\right)(x)\right|_{\mathscr{C}^{0}}\leq Ce^{-cR}.
\end{aligned}\end{align}
Finally, the first inequality of (\ref{e.225}) follows from (\ref{e.231}), (\ref{e.233}), (\ref{e.234}) and (\ref{e.240}).
Following a similar proof, we can get the second inequality of (\ref{e.225}). Using (\ref{e.232}) and off-diagonal estimates,
we get the uniform estimates (\ref{e.740}) by a similar argument.
\end{proof}

\begin{lemma}\label{l.25}
We have for $R\rightarrow \infty$
\begin{align}\begin{aligned}\label{e.226}
 {\rm G}_{R}={\rm I}_{h^{(p)}\times h^{(p)}}+\mO(e^{-cR}).
\end{aligned}\end{align}
\end{lemma}
\begin{proof}
As we see that the volume of $Z_{R}$ grows of order $\mO(R)$, this lemma is an easy consequence of Lemma \ref{l.24}, (\ref{e.223}) and (\ref{e.217}).
\end{proof}
To prove Lemma \ref{l.16}, we need the following lemma.

\begin{lemma}\label{l.26} We have for $R\rightarrow \infty$
\begin{align}\begin{aligned}\label{e.227}
{\rm \Theta}_{R}
 ={\rm I}_{h^{(p)}\times h^{(p)}}+\mO(e^{-cR}).
\end{aligned}\end{align}
\end{lemma}
Before proving this lemma, we see that (\ref{e.218}) follows from (\ref{e.219}), (\ref{e.226}) and (\ref{e.227}),
so in order to prove Lemma \ref{l.16} we only need to prove Lemma \ref{l.26}. The last subsection of this paper will be contributed to prove Lemma
\ref{l.26}.

\subsection{Proof of Lemma \ref{l.26}}\label{ss3.10}

\begin{defn}\label{d.17} For
 $[\sigma^{\bullet}_{R}]\in H^{p}(Z_{R}, F_{R}), \,\big(\text{resp. } H^{p}(Z_{1, R}, F_{R}),\, H^{p}(Z_{2, R}, Y, F_{R})\big)$
 and
 $ [\sigma_{R, \bullet}]\in H_{p}(Z_{R}, F^{*}_{R}),\, \big(\text{resp. } H_{p}(Z_{1, R}, F^{*}_{1, R}),\, H_{p}(Z_{2, R}, Y, F^{*}_{2, R})\big).$
There is a well-defined paring between the cohomology groups and the homology groups induced by the natural paring between the cochain groups and chain groups, that's
\begin{align}\begin{aligned}\label{e.242}
\langle [\sigma^{\bullet}_{R}], [\sigma_{R, \bullet}]\rangle:=\langle \sigma^{\bullet}_{R}, \sigma_{R, \bullet}\rangle.
\end{aligned}\end{align}
This paring between the cohomology groups and the homology groups is non-degenerate.
\end{defn}

Using the following commutative diagram
\begin{align}\begin{aligned}\label{e.252}
\xymatrix{
&&&0\ar[d]&0\ar[ld]\\
&0\ar[d]&&H_{p}(Z_{2, R}, F^{*}_{2, R})\ar[d]_{\wr}^{f_{p}}\ar[ld]_{h_{p}}&\\
0\ar[r]&H_{p}(Z_{1, R}, F^{*}_{1, R})\ar[r]^{i_{p}}\ar[d]^{\wr}_{k_{p}}&H_{p}(Z_{R}, F^{*}_{R})\ar[r]^{j_{p}}\ar[ld]_{l_{p}}&H_{p}(Z_{2, R}, Y, F^{*}_{2, R})\ar[r]\ar[d]&0, \\
& H_{p}(Z_{1, R}, Y, F^{*}_{1, R})\ar[d]\ar[ld] &&0&\\
0&0&&&}
\end{aligned}\end{align}
we define
\begin{align}\begin{aligned}\label{e.256}
j^{-1}_{p}=h_{p}\circ f^{-1}_{p}.\end{aligned}\end{align}
 From the diagram (\ref{e.252}), we read off that
\begin{align}\begin{aligned}\label{e.253}
j_{p}\circ i_{p}=0, \quad l_{p}\circ h_{p}=0.
\end{aligned}\end{align}

Recall that the diffeomorphisms $\phi_{R}:M\longrightarrow M_{R}$ and $\phi_{i, R}:M_{i}\longrightarrow M_{i, R}$ have been constructed respectively
in Lemma \ref{l.1}.

\begin{defn}\label{d.37}Let $\{[\sigma_{i}]|1\leq i\leq h^{(p)}_{1}\}$ (resp. $\{[\tau_{j}]|1\leq j\leq h^{(p)}_{2}\}$) be a local frame of
$H_{p}(Z_{1}, F^{*})$ (resp. $H_{p}(Z_{2}, Y, F^{*})$).
For any $R\geq 0$, we put
 \begin{align}\begin{aligned}\label{e.696}
[\sigma_{R, i}]=\phi_{R}[\sigma_{i}],\quad 1\leq i\leq h^{(p)}_{1}\quad (\text{resp.}\,\,[\tau_{R, j}]=\phi_{R}[\tau_{j}],\quad 1\leq j\leq h^{(p)}_{2}),
\end{aligned}\end{align}
which constitute a local frame of
$H_{p}(Z_{1, R}, F^{*}_{1, R})$ (resp. $H_{p}(Z_{2, R}, Y, F^{*}_{2, R})$).
Then we define for $1\leq j\leq h^{(p)}_{2}$
\begin{align}\begin{aligned}\label{e.254}
[\widetilde{\tau}_{R, j}]=f^{-1}_{p}[\tau_{R, j}],
\end{aligned}\end{align}
which constitute a frame of $H_{p}(Z_{2, R}, F^{*}_{2, R})$.
\end{defn}

By our construction, we know that the coefficients of expansion of $\{[\sigma_{R, i}], [\tau_{R, j}], [\widetilde{\tau}_{R, j}]\}$ in term of the frame
$\{\mathfrak{a}_{R, i}, \mathfrak{b}_{R, j}\}$ are constants with respect to $R$, so we get by Remark \ref{r.1}
\begin{align}\begin{aligned}\label{e.248}
\text{Vol}(\sigma_{R, i})=\mO(R), \quad \text{Vol}(\tau_{R, j})=\mO(R), \quad
\text{Vol}(\widetilde{\tau}_{R, j})=\mO(R).
\end{aligned}\end{align}
Let $\mu_{R}=(\mu_{Rii'})$ be a $h^{(p)}_{1}\times h^{(p)}_{1}$ matrix and $\eta_{R}=(\eta_{Rjj'})$ be a $h^{(p)}_{2}\times h^{(p)}_{2}$ matrix,
defined by
\begin{align}\begin{aligned}\label{e.243}
&\mu_{Rii'}=\big\langle [a^{i}_{R}], [\sigma_{R, i'}]\big\rangle,\quad
\eta_{Rjj'}=\big\langle [b^{j}_{R}], [\tau_{R, j'}]\big\rangle.
\end{aligned}\end{align}
By Lemma \ref{l.23}, we get for $1\leq i,\,i'\leq h^{(p)}_{1},\,1\leq j,\,j'\leq h^{(p)}_{2}$
\begin{align}\begin{aligned}\label{e.244}
\mu_{Rii'}=\mO(R), \quad \eta_{Rjj'}= \mO(R).
\end{aligned}\end{align}
We set $\mu^{-1}_{R}=\big((\mu^{-1}_{R})_{ii'}\big)$ to be the inverse of $\mu_{R}$ and
 $\eta^{-1}_{R}=\big((\eta^{-1}_{R})_{jj'}\big)$ to be that of $\eta_{R}$,
then we define
\begin{align}\begin{aligned}\label{e.246}
 a_{R, i}:=\sum_{i'}(\mu^{-1}_{R})_{ii'}\sigma_{R, i'}, \quad b_{R, j}:=\sum_{j'}(\eta^{-1}_{R})_{jj'}\tau_{R,j'},
\end{aligned}\end{align}
whose homology classes, $[a_{R,i}],\,[b_{R,j}]$, are respectively new frames of $H_{p}(Z_{1, R}, F^{*}_{1, R})$
and $H_{p}(Z_{2, R}, Y, F^{*}_{2, R})$.
Consequently, by (\ref{e.254}) and (\ref{e.246}), we get
\begin{align}\begin{aligned}\label{e.255}
f^{-1}_{p}(b_{R,j})=\sum_{j'}(\eta^{-1}_{R})_{jj'}f^{-1}_{p}(\tau_{R, j'}).
\end{aligned}\end{align}
By (\ref{e.243}) and (\ref{e.246}) we get
\begin{align}\begin{aligned}\label{e.247}
\big(\langle [a^{i}_{R}], [ a_{R, i'}]\rangle\big)_{h^{(p)}_{1}\times h^{(p)}_{1}}={\rm I}_{h^{(p)}_{1}\times h^{(p)}_{1}},
 \quad \big(\langle [b^{j}_{R}], [b_{R, j'}]\rangle\big)_{h^{(p)}_{2}\times h^{(p)}_{2}}={\rm I}_{h^{(p)}_{2}
\times h^{(p)}_{2}}.
\end{aligned}\end{align}
\begin{lemma}\label{l.45} For $R$ sufficiently large, we have for any $1\leq i,\,i'\leq h^{(p)}_{1},\,1\leq j,\,j'\leq h^{(p)}_{2}$,
\begin{align}\begin{aligned}\label{e.444}
(\mu^{-1}_{R})_{ii'}=\mathcal{O}(R), \quad (\eta^{-1}_{R})_{jj'}=\mathcal{O}(R).
\end{aligned}\end{align}
\end{lemma}
\begin{proof} Let $\{[\sigma^{i}]|1\leq i\leq h^{(p)}_{1}\}$ (resp. $\{[\tau^{j}]|1\leq j\leq h^{(p)}_{2}\}$) be a frame of
$H^{p}(Z_{1}, F)$ (resp. $H^{p}(Z_{2}, Y, F)$), such that
\begin{align}\begin{aligned}\label{e.698}
\langle[\sigma^{i}],[\sigma_{i'}]\rangle=\delta^{i}_{i'}\quad (\text{resp.}\,\,\langle[\tau^{j}],[\tau_{j'}]\rangle=\delta^{j}_{j'}\,).
\end{aligned}\end{align}
Let $\{\xi^{i}|\, 1\leq i\leq h^{(p)}_{1}\}$ (resp. $\{\zeta^{j}|\, 1\leq j\leq h^{(p)}_{2}\}$) be a frame of the space of harmonic forms $\mathscr{H}^{p}(Z_{1},F)$ (resp. $\mathscr{H}^{p}(Z_{2},Y,F)$) such that
\begin{align}\begin{aligned}\label{e.699}
P^{\infty}_{1}(\xi^{i})=[\sigma^{i}]\quad (\text{resp.}\,\,P^{\infty}_{2}(\zeta^{j})=[\tau^{j}]\,).
\end{aligned}\end{align}
Let $\{\xi^{i}_{R}|\, 1\leq i\leq h^{(p)}_{1}\}$ (resp. $\{\zeta^{j}_{R}|\, 1\leq j\leq h^{(p)}_{2}\}$) be a frame of $\mathscr{H}^{p}(Z_{1,R},F_{1,R})$ (resp. $\mathscr{H}^{p}(Z_{2,R},Y,F_{2,R})$) such that
\begin{align}\begin{aligned}\label{e.700}
P^{\infty}_{1,R}(\xi^{i}_{R})=(\phi_{R}^{-1})^{*}[\sigma^{i}],\quad (\text{resp.}\,\,P^{\infty}_{2,R}(\zeta^{j}_{R})=(\phi_{R}^{-1})^{*}[\tau^{j}]\,).
\end{aligned}\end{align}
By (\ref{e.696}), (\ref{e.243}), (\ref{e.699}) and (\ref{e.700}), we have
\begin{align}\begin{aligned}\label{e.701}
[a^{i}_{R}]=\sum_{i'}\mu_{Rii'}(\phi_{R}^{-1})^{*}[\sigma^{i'}],\quad [b^{j}_{R}]=\sum_{j'}\eta_{Rjj'}(\phi_{R}^{-1})^{*}[\tau^{j'}].
\end{aligned}\end{align}
Consequently we have
\begin{align}\begin{aligned}\label{e.702}
{\rm I}_{h^{(p)}_{1}\times h^{(p)}_{1}}&=\big\langle[a^{i}_{R}],[a^{k}_{R}]\big\rangle_{h^{H^{p}(Z_{1,R},F_{1,R})}_{L^{2}}}\\
&=\big(\mu_{Rii'}\big)\left(\big\langle(\phi_{R}^{-1})^{*}[\sigma^{i'}], (\phi_{R}^{-1})^{*}[\sigma^{k'}]\big\rangle_{h^{H^{p}(Z_{1,R},F_{1,R})}_{L^{2}}}\right)\big(\mu_{Rkk'}\big)^{*}\\
&=\big(\mu_{Rii'}\big)\left(\big\langle\xi^{i'}_{R},\xi^{k'}_{R} \big\rangle_{L^{2}(Z_{1,R})}\right)\big(\mu_{Rkk'}\big)^{*},
\end{aligned}\end{align}
and similarly
\begin{align}\begin{aligned}\label{e.703}
{\rm I}_{h^{(p)}_{2}\times h^{(p)}_{2}}&=\big\langle[b^{j}_{R}],[b^{l}_{R}]\big\rangle_{h^{H^{p}(Z_{2,R},Y,F_{2,R})}_{L^{2}}}
=\big(\eta_{Rjj'}\big)\left(\big\langle\zeta^{j'}_{R},\zeta^{l'}_{R} \big\rangle_{L^{2}(Z_{2,R})}\right)\big(\eta_{Rll'}\big)^{*}.
\end{aligned}\end{align}
Let $\|\cdot\|_{HS}$ be the Hilbert-Schmidt norm of the matrix, then by (\ref{e.702}) and (\ref{e.703}) we have
\begin{align}\begin{aligned}\label{e.704}
\|\mu^{-1}_{R}\|^{2}_{HS}&=\tr\left(\big\langle\xi^{i'}_{R},\xi^{k'}_{R} \big\rangle_{L^{2}(Z_{1,R})}\right)=\sum_{k=1}^{h_{1}^{(p)}}\|\xi^{k}_{R}\|^{2}_{L^{2}(Z_{1,R})},\\
\|\eta^{-1}_{R}\|^{2}_{HS}&=\tr\left(\big\langle\zeta^{j'}_{R},\zeta^{k'}_{R} \big\rangle_{L^{2}(Z_{2,R})}\right)=\sum_{l=1}^{h_{2}^{(p)}}\|\zeta^{l}_{R}\|^{2}_{L^{2}(Z_{2,R})}.
\end{aligned}\end{align}
For $i=1,2$, let $L^{2}_{R}(Z_{i})$ be the $L^{2}-$metric on $Z_{i}$ with respect to $g^{TZ_{i}}_{R}$ and $h^{F}$. Let
$(P^{\{0\}}_{i})_{R}$ be the orthogonal projection on the space of harmonic forms on $Z_{i}$ with respect to $L^{2}_{R}(Z_{i})$.
By Lemma \ref{l.44} and (\ref{e.704}), we get the first estimate in (\ref{e.444})
\begin{align}\begin{aligned}\label{e.706}
\|\mu^{-1}_{R}\|^{2}_{HS}&=\sum_{k=1}^{h_{1}^{(p)}}\|\phi^{*}_{R}\xi^{k}_{R}\|^{2}_{L^{2}_{R}(Z_{1})}
=\sum_{k=1}^{h_{1}^{(p)}}\|(P^{\{0\}}_{i})_{R}\xi^{k}\|^{2}_{L^{2}_{R}(Z_{1})}\\
&\leq \sum_{k=1}^{h_{1}^{(p)}}\|\xi^{k}\|^{2}_{L^{2}_{R}(Z_{1})}=\mathcal{O}(R^{2}).
\end{aligned}\end{align}
Similarly, by Lemma \ref{l.44} and (\ref{e.704}) we get the second estimate of (\ref{e.444}).
The proof is completed.
\end{proof}

By (\ref{e.248}), (\ref{e.246}), (\ref{e.255}) and (\ref{e.444}), we get
\begin{align}\begin{aligned}\label{e.249}
\text{Vol}(a_{R, i})=\mO(R^{2}), \quad \text{Vol}(b_{R, j})=\mO(R^{2}), \quad \text{Vol}(f^{-1}_{p}b_{R, j})=\mO(R^{2}).
\end{aligned}\end{align}

By (\ref{e.182}), (\ref{e.209}), (\ref{e.256}), (\ref{e.253}) and (\ref{e.247}), we get for
$1\leq i',k\leq h^{(p)}_{1},\,1\leq j',l\leq h^{(p)}_{2}$,
\begin{align}\begin{aligned}\label{e.723}
\int_{i_{p}a_{R, i'}}\widetilde{\alpha}_{k, R}&= \big\langle (i^{*}_{p})^{-1}[a^{k}_{R}], i_{p}[ a_{R, i'}]\big\rangle
=\big\langle [a^{k}_{R}], [ a_{R, i'}]\big\rangle=\delta^{k}_{i'},\\
\int_{j^{-1}_{p}b_{R, j'}} \widetilde{\alpha}_{k, R}&=\big\langle (i^{*}_{p})^{-1}[a^{k}_{R}], j^{-1}_{p}[b_{R, j'}]\big\rangle\\
&=\big\langle (k^{*}_{p})^{-1}[a^{k}_{R}], (l_{p}\circ h_{p})\circ f^{-1}_{p}[b_{R, j'}]\big\rangle=0,\\
\int_{i_{p}a_{R, i'}}\widetilde{\beta}_{l, R}&=\big\langle j^{*}_{p}[b^{l}_{R}], i_{p}[ a_{R, i'}]\big\rangle
=\big\langle [b^{l}_{R}], (j_{p}\circ i_{p})[ a_{R, i'}]\big\rangle=0,\\
\int_{j^{-1}_{p}b_{R, j'}} \widetilde{\beta}_{l, R}&=\big\langle j^{*}_{p}[b^{l}_{R}], j^{-1}_{p}[b_{R, j'}]\big\rangle
=\big\langle [b^{l}_{R}], [b_{R, j'}]\big\rangle=\delta^{l}_{j'}.
\end{aligned}\end{align}
This means that $\{P^{\infty}_{R}\widetilde{\alpha}_{k, R},\,P^{\infty}_{R}\widetilde{\beta}_{l, R}\}$ is the dual basis of $\{i_{p}a_{R, i'},\,j^{-1}_{p}b_{R, j'}\}$.
We set ${\rm A}_{R}=\left(
                      \begin{array}{cc}
                        {\rm A}_{R}^{1} & {\rm A}_{R}^{2} \\
                        {\rm A}_{R}^{3} & {\rm A}_{R}^{4} \\
                      \end{array}
                    \right)
$ such that for $1\leq i,k\leq h^{(p)}_{1},\,1\leq j,l\leq h^{(p)}_{2}$,
\begin{align}\begin{aligned}\label{e.722}
&({\rm A}_{R}^{1})_{ki}= \int_{i_{p}a_{R, i}}\widehat{\alpha}_{k, R} ;
&({\rm A}_{R}^{2})_{kj}= \int_{j^{-1}_{p}b_{R, j}}\widehat{\alpha}_{k, R} ;\\
&({\rm A}_{R}^{3})_{li}= \int_{i_{p}a_{R, i}}\widehat{\beta}_{l, R} ;
&({\rm A}_{R}^{4})_{lj}= \int_{j^{-1}_{p} b_{R, j}}\widehat{\beta}_{l, R} .
\end{aligned}\end{align}
By (\ref{e.241}), (\ref{e.723}) and  (\ref{e.722}), we get
\begin{align}\begin{aligned}\label{e.250}
{\rm \Theta}_{R}\cdot{\rm A}_{R}
=
\left(
 \begin{array}{ll}
  \Big(\int_{i_{p}a_{R, i'}}\widetilde{\alpha}_{k, R} \Big)_{h^{(p)}_{1}\times h^{(p)}_{1}}& \Big(\int_{j^{-1}_{p}b_{R, j'}} \widetilde{\alpha}_{k, R} \Big)_{h^{(p)}_{1}\times h^{(p)}_{2}} \\
  \Big(\int_{i_{p}a_{R, i'}}\widetilde{\beta}_{l, R}  \Big)_{h^{(p)}_{2}\times h^{(p)}_{1}}& \Big(\int_{j^{-1}_{p}b_{R, j'}} \widetilde{\beta}_{l, R} \Big)_{h^{(p)}_{2}\times h^{(p)}_{2}}\\
 \end{array}
\right)={\rm I}_{h^{(p)}\times h^{(p)}}.
\end{aligned}\end{align}

By (\ref{e.250}), to study the asymptotic behavior of ${\rm \Theta}_{R}$ we need to get an estimate for each block of ${\rm A}_{R}$.

\textbf{1)\quad} For ${\rm A}^{1}_{R}$, by Lemma \ref{l.24}, (\ref{e.207}), (\ref{e.247}), (\ref{e.249}) and (\ref{e.722}), we get
\begin{align}\begin{aligned}\label{e.257}
({\rm A}^{1}_{R})_{k'i'}&=\int_{i_{p}a_{R, i'}}\widehat{\alpha}_{k', R}
=\int_{a_{R, i'}}\widehat{\alpha}_{k', R}
=\int_{a_{R, i'}}\alpha_{k', R}+\int_{a_{R, i'}}\big(\widehat{\alpha}_{k', R}-\alpha_{k', R}\big)\\
&=\big\langle [a^{k'}_{R}], [a_{R, i'}]\big\rangle+\int_{a_{R, i'}}\big(\widehat{\alpha}_{k', R}-\alpha_{k', R}\big)
=\delta^{k'}_{i'}+\int_{a_{R, i'}}\big(\widehat{\alpha}_{k', R}-\alpha_{k', R}\big),
\end{aligned}\end{align}
and
\begin{align}\begin{aligned}\label{e.258}
\big|\int_{a_{R, i'}}\big(\widehat{\alpha}_{k', R}-\alpha_{k', R}\big)\big|&\leq Ce^{-cR}\cdot \text{Vol}(a_{R, i'})\leq C'{e}^{-c'R}.
\end{aligned}\end{align}
By (\ref{e.257}) and (\ref{e.258}), we find
\begin{align}\begin{aligned}\label{e.259}
({\rm A}^{1}_{R})_{k'i'}=\delta^{k'}_{i'}+\mO(e^{-cR}).
\end{aligned}\end{align}

\textbf{2)\quad} For ${\rm A}^{2}_{R}$, by (\ref{e.740}), (\ref{e.249}) and (\ref{e.722}), we get
\begin{align}\begin{aligned}\label{e.268}
\big|({\rm A}^{2}_{R})_{k'j'}\big|&= \big|\int_{f^{-1}_{p}b_{R, j'}}h^{*}_{p}\widehat{\alpha}_{k', R}\big|=\big|\int_{f^{-1}_{p}b_{R, j'}}\widehat{\alpha}_{k', R}\big|\\
 &\leq Ce^{-cR}\text{Vol}(f^{-1}_{p}b_{R, j'})\leq C{e}^{-c'R}.
\end{aligned}\end{align}

\textbf{3)\quad } For ${\rm A}^{3}_{R}$, by (\ref{e.740}), (\ref{e.249}) and (\ref{e.722}), we have
\begin{align}\begin{aligned}\label{e.260}
\big|({\rm A}^{3}_{R})_{l'i'}\big|
&=\big|\int_{i_{p}a_{R, i'}}\widehat{\beta}_{l', R}\big|
\leq Ce^{-cR}\cdot \text{Vol}(a_{R, i'})\leq C'{e}^{-c'R}.
\end{aligned}\end{align}

\textbf{4)\quad } For ${\rm A}^{4}_{R}$, we observe that
\begin{align}\begin{aligned}\label{e.262}
b_{R, j'}-f^{-1}_{p}\big(b_{R, j'}\big)\in C_{p}(\mathbb{K}_{Y}, F^{*}).
\end{aligned}\end{align}
Since $\{\beta_{l', R}\in \mathscr{H}^{p}\left(Z_{2, R}, Y, F_{R}\right)|1\leq l' \leq h^{(p)}_{2}\}$
satisfy the relative boundary conditions, so we get from (\ref{e.262}) that
\begin{align}\begin{aligned}\label{e.263}
\int_{f^{-1}_{p}\big(b_{R, j'}\big)} \beta_{l', R}=\int_{b_{R, j'}}\beta_{l', R}=\delta^{l'}_{j'}.
\end{aligned}\end{align}
 By (\ref{e.207}), Lemma \ref{l.24}, (\ref{e.247}), (\ref{e.249}), (\ref{e.722}) and (\ref{e.263}), we get
\begin{align}\begin{aligned}\label{e.264}
 ({\rm A}^{4}_{R})_{l'j'}&= \int_{f^{-1}_{p}b_{R, j'}}h^{*}_{p}\widehat{\beta}_{l', R}=\int_{f^{-1}_{p}b_{R, j'}}\widehat{\beta}_{l', R}\\
& =\int_{f^{-1}_{p}b_{R, j'}}\beta_{l', R}+\int_{f^{-1}_{p}b_{R, j'}}\big(\widehat{\beta}_{l', R}- \beta_{l', R}\big)
 =\delta^{l'}_{j'}+\int_{f^{-1}_{p}b_{R, j'}}\big(\widehat{\beta}_{l', R}- \beta_{l', R}\big)
\end{aligned}\end{align}
and
\begin{align}\begin{aligned}\label{e.265}
\big|\int_{f^{-1}_{p}b_{R, j'}}\big(\widehat{\beta}_{l', R}- \beta_{l', R}\big)\big|\leq Ce^{-cR}\text{Vol}(f^{-1}_{p}b_{R, j'})\leq C'e^{-c'R}.
\end{aligned}\end{align}
Then from (\ref{e.264}) and (\ref{e.265}), we get
\begin{align}\begin{aligned}\label{e.266}
  ({\rm A}^{4}_{R})_{l'j'}=\delta^{l'}_{j}+ \mO(e^{-cR}).
\end{aligned}\end{align}

Finally, by (\ref{e.259}), (\ref{e.268}), (\ref{e.260}) and (\ref{e.266}), we get
\begin{align}\begin{aligned}\label{e.269}
& {\rm A}_{R}={\rm I}_{h^{(p)}\times h^{(p)}} +\mO(e^{-cR}).
\end{aligned}\end{align}
By (\ref{e.250}) and (\ref{e.269}), we have proved Lemma \ref{l.26}, which is the main result of this subsection.\\

\bibliographystyle{plain}

\end{document}